\documentclass[a4paper,11pt]{article}
\usepackage[top=2.5cm, bottom=2.5cm, left=2.5cm, right=2.5cm]{geometry}
\usepackage{settings}
\allowdisplaybreaks

\begin{document}
\title{Efficient computation of the asymptotics of extensive-rank HCIZ integrals}
\date{\today}
\author{Antoine Maillard$^\star$, Jean-Christophe Mourrat$^\circ$}

\maketitle

{\let\thefootnote\relax\footnote{
    \noindent
$\star$ Inria Paris \& DI ENS, PSL University, Paris, France. \href{mailto:antoine.maillard@inria.fr}{antoine.maillard@inria.fr}\\
$\circ$ Department of Mathematics, ENS Lyon and CNRS, Lyon, France. \href{mailto:jean-christophe.mourrat@ens-lyon.fr}{jean-christophe.mourrat@ens-lyon.fr}.
}}
\setcounter{footnote}{0}

\begin{abstract}
    We study the high-dimensional asymptotics of Harish-Chandra--Itzykson--Zuber 
    (HCIZ) integrals in the extensive-rank regime. The limit of these integrals 
    is governed by a one-dimensional boundary-value hydrodynamical problem 
    originally derived by Matytsin (1994) and rigorously proved by Guionnet 
    and Zeitouni (2002). Despite its wide-ranging applications, explicit 
    solutions to this problem are known only in a few specific cases. In this 
    work, we introduce an efficient numerical scheme based on a particle 
    discretization and prove its convergence to the continuous boundary-value 
    problem for generic boundary densities. We validate our approach against 
    known analytical solutions and apply it to generic densities, uncovering 
    interesting dynamical phenomena. 
    The high-dimensional limit of HCIZ integrals appears in various contexts,
    from the large deviations of random matrix spectra to the limiting free 
    energy of disordered systems, high-dimensional statistics, and machine 
    learning. As such, our contribution opens the way towards the numerical
    exploration of a wide range of high-dimensional models that were previously 
    intractable.
\end{abstract}

\setcounter{tocdepth}{2}
\tableofcontents

\section{Introduction and main results}\label{sec:intro}
\subsection{Asymptotics of the HCIZ integral}

\subsubsection{Notation and definition}

\textbf{Notation --}
Throughout the paper, we adopt the following notation.
For $E \subseteq \bbR$, we denote by $\mcP(E)$ the set of probability measures on $E$. 
We also let $\mcP_{\abc}(\bbR)$ denote the set of probability measures that are absolutely continuous with respect to the Lebesgue measure 
and $\mcP_c(\bbR)$ the set of compactly supported probability measures.
Moreover, for $\mu \in \mcP(\bbR)$, we let $\Sigma(\mu) \coloneqq \int \log |x-y| \rd \mu(x) \, \rd \mu(y)$.
For a given $\alpha \geq 0$, we denote by $\alpha \# \mu$ the push-forward measure under multiplication by $\alpha$, i.e.\ 
$\int f(x) (\alpha \# \mu)(\rd x) = \int f(\alpha x) \mu(\rd x)$ for test functions $f$.
For $\beta \in \{1, 2\}$, we set $\bbK = \bbR$ for $\beta = 1$ and $\bbK = \bbC$ for $\beta = 2$.
For $A \in \bbK^{d \times d}$, we denote by $A^\dagger$ the Hermitian conjugate of $A$, 
and set $\tr(A) \coloneqq (1/d) \Tr(A)$.
We use $\mcH_d$ to denote the set of symmetric ($\beta = 1$) or Hermitian ($\beta = 2$) $d \times d$ matrices.
For $A \in \mcH_d$, we denote by $\hmu_A \coloneqq (1/d) \sum_{i=1}^d \delta_{\lambda_i}$ its empirical eigenvalue distribution, where $(\lambda_i)_{i=1}^d$ are the (real) eigenvalues of $A$.
$\bbS_\beta^{d-1}$ denotes the unit sphere in $\bbK^d$.
Finally, $\mcU_\beta(d)$ denotes the group of orthogonal ($\beta = 1$) or unitary ($\beta = 2$) matrices, and $\rd \rho_{\beta}$ is the Haar measure on this compact group.

\myskip 
\textbf{Definition --}
We consider the ``Harish-Chandra--Itzykson--Zuber'' (HCIZ) integral~\cite{harish1957differential,itzykson1980planar}, 
defined as follows.
For $A, B \in \mcH_d$, let
\begin{equation}\label{eq:def_Id}
    I_d^\tbet(A, B) \coloneqq \frac{2}{\beta d^2} \log \int_{\mcU_\beta(d)} \rd \rho_\beta(O) \exp\left\{\frac{\beta d}{2} \Tr[O A O^\dagger B]\right\}.
\end{equation}
We also define, for $\theta \geq 0$,
\begin{equation}\label{eq:def_Id_theta}
    I_d^\tbet(\theta, A, B) \coloneqq I_d^\tbet(\sqrt{\theta} A, \sqrt{\theta} B) = \frac{2}{\beta d^2} \log \int_{\mcU_\beta(d)} \rd \rho_\beta(O) \exp\left\{\frac{\beta \theta d}{2} \Tr[O A O^\dagger B]\right\}.
\end{equation}
While one could fix $\theta = 1$ without loss of generality, it will sometimes be useful in our discussion to be able to vary this scaling parameter.

\subsubsection{Related literature}\label{subsubsec:related_literature}

The ``spherical'' Harish-Chandra--Itzykson--Zuber (HCIZ) integral of eq.~\eqref{eq:def_Id} is a ubiquitous mathematical object with wide-ranging applications across multiple disciplines. 
It was initially introduced in the study of Lie algebras~\cite{harish1957differential}, and has found applications 
in fields as diverse as harmonic analysis, representation theory~\cite{mcswiggen2021harish}, quantum gravity and quantum field theory~\cite{itzykson1980planar}, and integrable systems~\cite{menon2017complex}. 
We refer to~\citet{mcswiggen2021harish} for a review of these applications, 
and discuss other motivations in more detail below.

\myskip
\textbf{High-dimensional asymptotics --}
In this work, we consider the high-dimensional regime, i.e.\ the behavior of $I_d^\tbet$ in eq.~\eqref{eq:def_Id} when the dimension $d$ gets very large.
Notably, an exact finite-dimensional determinantal formula exists in the unitary case $\beta = 2$~\cite{harish1957differential,itzykson1980planar,tao2013hciz}:
\begin{equation}\label{eq:hciz_finite}
    I_d^{(2)}(A_d, B_d) = \frac{1}{d^2} \log \det \left[\left(e^{d a_i b_j}\right)_{1 \leq i,j \leq d}\right] - \frac{1}{d^2} \sum_{c \in \{a, b\}}\sum_{i<j} \log (c_j - c_i)  + K_d,
\end{equation}
where $K_d \in \bbR$ is a constant independent of $(A_d, B_d)$, and $a_1 \leq \cdots \leq a_d$, $b_1 \leq \cdots \leq b_d$ are the eigenvalues of $A_d$ and $B_d$.
However, evaluating the asymptotic limit of eq.~\eqref{eq:hciz_finite} as $d \to \infty$ is highly non-trivial, and current results on the limit of $I_d^\tbet$ as $d \to \infty$ do not rely on 
this explicit formula.
There are many motivations for considering the HCIZ integral in the high-dimensional regime: 
without aiming to be exhaustive, we detail three such motivations below.

\myskip
\textbf{I: Large deviations in random matrix theory --}
A central motivation for evaluating high-dimensional HCIZ integrals stems from the theory of large deviations for the spectra of random matrices~\cite{guionnet2023rare}.
Notably, a recent line of work has directly related the large deviations of the extreme eigenvalues
of many random matrix ensembles to the asymptotics of the HCIZ integral when one of the two matrices $A_d, B_d$ has finite rank~\cite{maillard2021large,mergny2022right,niu2025role,husson2024large,guionnet2020large,guionnet2022asymptotics,guionnet2020large2,husson2022large,augeri2021large}.
The more generic case of $A_d, B_d$ with growing rank as $d \to \infty$, on the other hand, is related to the
rate function for the large deviations of the entire spectral density of the Dyson Brownian motion~\cite{guionnet2002large,guionnet2004addendum,bun2014instanton,potters2020first}.
It also appears in the large deviations of the spectra of correlated random matrix models~\cite{maillard2024average}, 
and in the joint distribution of the eigenvalues of the sample covariance matrix of Gaussian data with a generic covariance matrix~\cite{james1964distributions,guionnet2004large}. 

\myskip
\textbf{II: Disordered systems and spin glasses --}
In statistical mechanics, the high-dimensional asymptotics of the HCIZ integral frequently appear in the limiting free energies of disordered systems
and mean-field spin glasses under rotationally-invariant couplings~\cite{marinari1994replica,parisi1995mean,maillard2019high}.
Recently, it was also shown to be an important ingredient in the free energy of constraint satisfaction problems over symmetric matrices, such as ellipsoid fitting~\cite{maillard2024fitting} 
or an average-case variant of the matrix discrepancy problem~\cite{maillard2024average}.
The partial differential equations governing the asymptotics of the HCIZ integral have also been shown to arise in the macroscopic fluctuation theory of the Dyson gas~\cite{dandekar2024current}.
Finally, it also connects to the theory of stochastic processes, e.g.\ via the statistics of non-intersecting Brownian bridges; see~\cite{grela2021non} and the references therein.

\myskip
\textbf{III: High-dimensional statistics and machine learning --}
In recent years, the HCIZ integral has also emerged in several studies in theoretical data science and high-dimensional statistics 
to characterize the information-theoretic limits and computational phase transitions of statistical estimation problems.
Concretely, the low-rank integral (where one of the two matrices $A, B$ in eq.~\eqref{eq:def_Id} has a finite rank as $d \to \infty$) 
plays an important role to derive such quantities in low-rank matrix estimation, compressed sensing, and generalized linear models~\cite{maillard2019high,pourkamali2023bayesian,barbier2023fundamental,niu2025role}.
Conversely, the extensive-rank HCIZ integral (when $A$ and $B$ both have rank $\Theta(d)$) appears in 
high-rank matrix denoising and factorization under rotationally invariant setups~\cite{maillard2022perturbative,pourkamali2023rectangular,pourkamali2024matrix,pourkamali2025rectangular,troiani2022optimal}.
Most recently, this extensive-rank regime has been actively deployed to derive exact asymptotics for the learning dynamics and generalization capabilities of models of extensive-width neural networks and simple models of attention~\cite{barbier2025statistical,boncoraglio2026bayes,erba2026nuclear,defilippis2025scaling,xu2025fundamental,erba2025bilinear,maillard2024bayes}.

\myskip 
\textbf{Previous results --}
The limit of $I_d^\tbet$ in the low-rank regime has been extensively studied and characterized~\cite{marinari1994replica,guionnet2005fourier,collins2007new}.
Conversely, the extensive-rank regime is governed by a one-dimensional boundary-value hydrodynamical problem, first derived by~\citet{matytsin1994large}
and later proved by~\citet{guionnet2002large,guionnet2004addendum}, see also \cite{collins2009asymptotics,bun2014instanton,menon2017complex} for extensions and discussions.
These results form the starting point of our study, and are stated below in detail.
Recent work has also expanded these extensive-rank asymptotic evaluations to a generalization of the HCIZ integral to rectangular matrices~\cite{guionnet2021large}.

\subsubsection{Matytsin's solution}

We now state the asymptotics of $I_d$, first obtained by~\citet{matytsin1994large} and proved by~\citet{guionnet2002large,guionnet2004addendum}.
First we need to define admissible trajectories between two probability measures.
\begin{definition}[Admissible trajectories]\label{def:adm}
    For any $\mu, \nu \in \mcP(\bbR)$, we denote by $\adm(\mu,\nu)$ the set of all pairs $(\rho, v)$ of measurable functions from $(0,1) \times \bbR$ to $\bbR$ such that the following conditions hold.
\begin{enumerate}
\item For each $t \in (0,1)$, the function $\rho(t,\cdot)$ is the density of a probability measure which is absolutely continuous with respect to the Lebesgue measure.
\item We have the integrability condition
\begin{equation}
\label{eq:integrability}
\int_0^1 \int_\bbR |v(t,x)| \rho(t,x) \, \rd x \, \rd t < +\infty.
\end{equation}
\item The continuity equation
\begin{equation}
\label{eq:continuity}
\partial_t \rho + \partial_x (\rho v) = 0
\end{equation}
is satisfied in the weak sense.
\item
 We have
\begin{equation*}  
\lim_{t \to 0} \rho(t,x) \, \rd x = \mu \quad \text{ and } \quad \lim_{t \to 1} \rho(t,x) \, \rd x = \nu,
\end{equation*}
in the sense of weak convergence of probability measures.
\end{enumerate}
\end{definition}
\noindent
Notice that the condition in eq.~\eqref{eq:integrability} is required in order to make sense of the continuity equation~\eqref{eq:continuity}. 
The conditions of eqs.~\eqref{eq:integrability} and~\eqref{eq:continuity} imply the continuity of the path $t \mapsto \rho(t,x) \rd x$ for the topology of weak convergence of probability measures, see e.g.\ \cite[Lemma~8.1.2]{ambrosio2008gradient}.
We are now ready to state the asymptotics of $I_d^\tbet$.
\begin{theorem}[\cite{matytsin1994large,guionnet2002large,guionnet2004addendum,guionnet2004first}]
    \label{thm:limit_hciz}
    Let $\mu, \nu \in \mathcal P(\bbR)$, and for each $d \ge 1$, let $A_d, B_d \in \mcH_d$, with empirical eigenvalue distributions $\hmu_{A_d}, \hmu_{B_d}$. Assume that
    \begin{enumerate}[label=(\roman*)]
        \item $\hmu_{A_d} \to \mu$ and $\hmu_{B_d} \to \nu$ as $d \to \infty$ (in the sense of weak convergence);
        \item There exists $K > 0 $ such that $\supp(\hmu_{A_d}) \subseteq [-K, K]$ and $\supp(\hmu_{B_d}) \subseteq [-K, K]$ for all $d \geq 1$;
    \end{enumerate}
    Then $I_\HCIZ(\mu, \nu) \coloneqq \lim_{d \to \infty} I_d^\tbet(A_d, B_d)$ exists, is independent of $\beta \in \{1, 2\}$, 
    and is continuous in both $\mu$ and $\nu$ with respect to the weak convergence on $\mcP([-K, K])$.

    \myskip
    Moreover, if we assume further that $\min\{\Sigma(\mu), \Sigma(\nu)\} > -\infty$, 
    then the following statements hold.
    \begin{enumerate}[label=(\arabic*)]
        \item 
        $I_\HCIZ(\mu, \nu)$ is given by
        \begin{equation}
            \label{eq:I_HCIZ_matytsin}
            I_\HCIZ(\mu, \nu) = - \frac{3}{4} + \frac{1}{2} \left[\int \mu(\rd x) x^2 + \int \nu(\rd x) x^2 \right] - \frac{1}{2} [\Sigma(\mu) + \Sigma(\nu)] - J(\mu, \nu),
        \end{equation}
        where
        \begin{align}
            \label{eq:def_J}
            J(\mu, \nu) \coloneqq \frac{1}{2} \inf_{(\rho, v) \in \adm(\mu, \nu)} \int_0^1 \rd t \, \int \rd x \, \rho(t,x) \Big[v(t,x)^2 + \frac{\pi^2}{3} \rho(t,x)^2\Big].
        \end{align}
        \item The infimum in eq.~\eqref{eq:def_J} is attained by a unique path $(\rho, v)$, which satisfies (in the weak sense) 
        the equation
        \begin{align}\label{eq:euler_2nd_equation}
            \partial_t (\rho v) + \partial_x \left[\rho v^2 - \frac{\pi^2}{3} \rho^3\right] = 0.
        \end{align}
        Further, 
        $(\rho, v)$ are smooth in the interior of $\Omega \coloneqq \{(t, x) \in (0,1) \times \bbR \, : \, \rho(t,x) > 0\}$, which is bounded.
        In particular, eqs.~\eqref{eq:continuity} and~\eqref{eq:euler_2nd_equation} hold everywhere in the interior of $\Omega$.
    \end{enumerate}
\end{theorem}
\noindent
We notice that having both measures uniformly compactly supported is not strictly necessary for some of these results.
For instance, one can define $I_\HCIZ(\mu, \nu)$ whenever $\mu$ is compactly supported and $\nu$ has bounded first moment: we refer to~\cite[Section~2]{belinschi2022large} for more details.

\myskip 
\textbf{Scaling parameter $\theta$ --}
Under the same assumptions as in Theorem~\ref{thm:limit_hciz}, we define
\begin{equation}\label{eq:def_Itheta_mu_nu}
I_\HCIZ(\theta, \mu, \nu) \coloneqq \lim_{d \to \infty} I_d^\tbet(\theta, A_d, B_d)
\end{equation}
where $I_d^\tbet$ is defined in eq.~\eqref{eq:def_Id_theta},
and we have 
\begin{equation}\label{eq:Itheta_Jtheta}
I_\HCIZ(\theta, \mu, \nu) = - \frac{3}{4} - \frac{1}{2} \log \theta + \frac{\theta}{2} \left[\int \mu(\rd x) x^2 + \int \nu(\rd x) x^2 \right] - \frac{1}{2} [\Sigma(\mu) + \Sigma(\nu)] - J(\theta, \mu, \nu),
\end{equation}
where $J(\theta, \mu, \nu) \coloneqq J(\sqrt{\theta} \# \mu, \sqrt{\theta} \# \nu)$ as given by eq.~\eqref{eq:def_J}. 
We will keep the shorthand notation $I_\HCIZ(\mu, \nu) = I_\HCIZ(\theta = 1, \mu, \nu)$.

\myskip
We now make a series of additional remarks on Theorem~\ref{thm:limit_hciz}.
\begin{enumerate}[label=(\Roman*)]
    \item 
    Interestingly, eqs.~\eqref{eq:continuity} and~\eqref{eq:euler_2nd_equation} can be interpreted as the 
    hydrodynamical description of a fluid under an \emph{attractive} force, i.e.\ a negative pressure field $P = - (\pi^2 / 3) \rho^3$~\cite{bun2014instanton}.
    Another key interpretation of these dynamics arises from random matrix theory, see e.g.~\cite{guionnet2023rare}. 
    To describe this interpretation, let us consider $\beta = 1$ for simplicity.
    One can show that $\rho(t,x) \, \rd x$ is the limiting spectral distribution of the constrained Dyson Brownian motion
    \begin{equation}\label{eq:dbm_def}
        M_t \coloneqq A_d + \sqrt{t} W,
    \end{equation}
    where $(W_{ij})_{i \leq j} \iid \mcN(0, (1 + \delta_{ij})/d)$, under the constraint that $M_{t=1}$ has asymptotic eigenvalue density equal to $\nu$.
    This interpretation is also key to some derivations of Theorem~\ref{thm:limit_hciz}, which use this relation of the HCIZ integral to the large deviations of the Dyson Brownian motion~\cite{guionnet2002large,guionnet2004addendum,bun2014instanton}.
    \item Notice that the limit in eq.~\eqref{eq:I_HCIZ_matytsin} is independent of $\beta \in \{1, 2\}$. This is a generic property known as ``Zuber's $1/2$ rule'~\cite{zuber2008large,bun2014instanton}.
    \item 
    It was shown in~\cite[Theorem~6]{collins2007new} that (under some technical assumptions)
    \begin{equation}\label{eq:high_to_low_rank}
        \lim_{a \to 0} \left[\frac{1}{a} I_\HCIZ((1-a) \delta_0 + a \mu, \nu)\right] = \int_\bbR \tI_\nu(t) \mu(\rd t),
    \end{equation}
    where (with $\hmu_B \to \nu$ as $d \to \infty$)
    \begin{equation*}
        \tI_\nu(t) \coloneqq \lim_{d \to \infty}\frac{2}{\beta d} \log \int_{\mcU_\beta(d)} \rd \rho_\beta(O) \exp\left\{\frac{\beta d t}{2} (O^\dagger B_d O)_{11}\right\}
    \end{equation*}
    is the ``rank-one'' HCIZ integral.
    While eq.~\eqref{eq:high_to_low_rank} shows a smooth transition between the asymptotic formulas for the ``low-rank'' and ``extensive-rank'' regimes of the HCIZ integral, 
    the proof of~\citet{collins2007new} shows it using bounds on the finite-dimensional integral. 
    As pointed out in~\cite{guionnet2023rare},
    it is an interesting question to derive eq.~\eqref{eq:high_to_low_rank} directly from Theorem~\ref{thm:limit_hciz} instead, 
    and to characterize the behavior of the trajectories when one of the two densities has vanishingly small rank.
\end{enumerate}

\myskip 
\textbf{The need for a numerical scheme --}
Somewhat disappointingly, explicit analytical solutions to the hydrodynamical problem of Theorem~\ref{thm:limit_hciz} are very scarce.
To our knowledge, the following cases are the only ones amenable to an exact analytical treatment.
\begin{enumerate}
    [label=\textbf{(C\arabic*)},ref=(C\arabic*)]
    \item
    For any $\mu, \nu$ satisfying the hypotheses of Theorem~\ref{thm:limit_hciz}, one can expand $I_\HCIZ(\theta, \mu, \nu)$ in powers of $\theta$ as $\theta \downarrow 0$ (a \emph{high-temperature} expansion)~\cite{collins2003moments}.
    Conversely, there are formulas 
    in the physics literature~\cite{bun2014instanton}
    for the first orders of the expansion of $I_\HCIZ(\theta, \mu, \nu)$ as $\theta \to \infty$ (a \emph{low-temperature} expansion).
    \item
    \label{case:free_conv}
    If $\nu = \mu \boxplus \sigma_{\sci}$, where 
    \begin{equation*}
        \sigma_{\sci}(\rd x) \coloneqq \frac{\sqrt{4-x^2}}{2\pi} \indi\{|x| \leq 2\} \rd x
    \end{equation*}
    is the semicircle law and $\boxplus$ is the additive free convolution~\cite{anderson2010introduction}, then the Dyson Brownian motion in eq.~\eqref{eq:dbm_def} is unconstrained, as $\nu$ corresponds then to the \emph{typical} eigenvalue density of $M_{t=1}$.
    Writing $\rho_t \coloneqq \rho(t,x) \, \rd x$ for the probability measure at time $t$, we have (see e.g.\ \cite{schmidt2018statistical,maillard2022perturbative}) that $\rho_t = \mu \boxplus \sigma_{\sci, \sqrt{t}}$ for any $t \in [0,1]$, where $\sigma_{\sci, r}$ is the semicircle law with variance $r^2$.
    Further,
    this allows us to express $I_\HCIZ(\mu, \mu \boxplus \sigma_{\sci})$ in terms of a simple Gaussian integral and obtain the formula 
    \begin{equation}
    \label{eq:I_conv}
        I_\HCIZ(\mu, \mu \boxplus \sigma_{\sci}) = - \frac{1}{4} - \Sigma(\mu \boxplus \sigma_{\sci}) + \int \mu(\rd x) x^2.
    \end{equation}
    The measure $\rho_t$ can be computed analytically from the knowledge of $\mu$ using its $\mcR$-transform: we refer to~\cite{anderson2010introduction,potters2020first} for 
    more details.
    Moreover, the velocity field can be expressed as:
    \begin{equation}\label{eq:vt_free_conv}
        v(t,x) = - \lim_{\eps \downarrow 0} \Re[g_{\rho_t}(x + i \eps)],
    \end{equation}
    where we define the Stieltjes transform of any probability measure $\mu \in \mcP(\bbR)$ as 
    \begin{equation}\label{eq:def_stieltjes}
       g_\mu(z) \coloneqq \int \frac{1}{x - z} \rd \mu(x),
    \end{equation}
    with $z \in \bbC_+ \coloneqq \{w \in \bbC \, : \, \Im(w) > 0\}$.
    \item
    \label{case:semicircle}
    In~\cite{bun2014instanton}, it is shown that when $\mu$ and $\nu$ are both centered\footnote{Centering can be assumed without loss of generality, since $I_d^{(\beta)}(A, B) = I_d^{(\beta)}[A - \tr(A) \Id_d, B - \tr(B) \Id_d] + \tr(A) \tr(B)$ for any $A, B \in \mcH_d$.} semicircular laws, the measure $\rho_t$ remains semicircular for any $t \in [0,1]$, 
    and  $I_\HCIZ(\mu, \nu)$ admits a simple expression. 
    Without loss of generality, we can consider $I_\HCIZ(\theta, \mu, \nu)$ for $\mu = \nu = \sigma_{\sci}$, which yields:
    \begin{equation}\label{eq:HCIZ_analytical_sc_to_sc}
        I_\HCIZ(\theta, \sigma_{\sci}, \sigma_{\sci}) = 
        \frac{1}{2} \left[\sqrt{1+4\theta^2} - 1 - \log \left(\frac{1+\sqrt{1+4\theta^2}}{2}\right)\right].
    \end{equation}
    \item 
    Finally, one can leverage the finite-dimensional determinantal formula of eq.~\eqref{eq:hciz_finite} if $\mu$ or $\nu$ is the uniform distribution over a real interval. 
    Indeed, in this case the computation of the determinant in eq.~\eqref{eq:hciz_finite} reduces to computing 
    \begin{equation*}
        \det \left[\left(e^{j \rho_k}\right)_{1 \leq j,k \leq d}\right] = \det \left[\left((e^{\rho_k})^j\right)_{1 \leq j,k \leq d}\right] = \prod_{j < k} \left|e^{\rho_j} - e^{\rho_k}\right|,
    \end{equation*}
    for some coefficients $(\rho_j)_{j=1}^d$. One can then exploit this Vandermonde determinant to extract high-dimensional asymptotics for the HCIZ integral, see e.g.~\cite{grela2021non} for an example
    in the physics literature.
\end{enumerate}
While these specific solvable scenarios serve as the backbone for some of the statistical applications mentioned previously, computing $I_\HCIZ(\mu, \nu)$ for generic distributions $\mu, \nu$ has largely remained mathematically intractable.
In this work, we overcome this limitation by designing a convergent numerical discretization scheme capable of evaluating the extensive-rank HCIZ integral for generic boundary densities.
We benchmark our algorithm against several of these known analytical cases, and open the door to the numerical exploration of a wide range of previously intractable high-dimensional models.
In the context of high-dimensional statistics, we anticipate that this will enable the exploration of much more general settings beyond Bayes-optimal estimation or extreme-temperature limits of free energies, 
and we plan to explore some of these applications in subsequent work.

\subsection{Main results}

As is clear from eq.~\eqref{eq:I_HCIZ_matytsin}, in order to compute $I_\HCIZ(\mu, \nu)$ for arbitrary $(\mu, \nu)$, it is enough to compute the functional $J(\mu, \nu)$ defined in eq.~\eqref{eq:def_J}. 
Recall that we have 
\begin{equation}\label{eq:def_J_2}
    J(\mu,\nu) \coloneqq \inf_{(\rho, v) \in \adm(\mu,\nu)} G(\rho, v),
\end{equation}
with 
\begin{equation}\label{eq:def_G}
    G(\rho,v) \coloneqq \frac 1 2 \int_0^1 \rd t \int_\bbR \rd x \, \left(v(t,x)^2 + \frac {\pi^2}{3} \rho(t,x)^2\right) \rho(t,x).
\end{equation}
Our main theoretical result is the validity of a discretization scheme to compute $J(\mu, \nu)$ in eq.~\eqref{eq:def_J_2} up to arbitrary accuracy. 
We first define the scheme, before stating our theorem.

\subsubsection{The discretization scheme}\label{subsubsec:def_discretization}

We fix $\mu_0$, $\mu_1 \in \mcP(\bbR)$ with $\Sigma(\mu_0)$ and $\Sigma(\mu_1)$ finite. 
We design a particle-based approximation scheme to compute $J(\mu_0,\mu_1)$, 
keeping time as a continuous variable for now.

\myskip
For every integer $N \ge 1$, we denote by $\adm_N(\mu_0,\mu_1)$ the set of all trajectories $\{(x_i, v_i)\}_{i=1}^N$, with $x_i, v_i \in \mcC([0,1], \bbR)$, that satisfy the following constraints. 
\begin{enumerate}[label=\textbf{(H\arabic*)},ref=(H\arabic*)]
\item
\label{bound_constraint_admn}
$(x_i(0))_{1 \le i \le N}$ and $(x_i(1))_{1 \le i \le N}$ are fixed such that for every $i \in \{1,\ldots, N\}$:
\begin{equation}
\label{eq:def_xi_01}
    \int_{-\infty}^{x_{i}(0)} \rd \mu_0(x) = \frac{i}{N+1} \quad \text{ and } \quad \int_{-\infty}^{x_{i}(1)} \rd \mu_1(x) = \frac{i}{N+1}.
\end{equation}
We may at times find it convenient to also use the convention $x_0(t) = -\infty$ and $x_{N+1}(t) = +\infty$ for every $t \in [0,1]$. 
\item
\label{noncrossing_constraint_admn}
For every $i \in \{1,\ldots, N\}$ and $t \in [0,1]$, we have $x_{i}(t) \le x_{i+1}(t)$.
\item 
\label{velocity_constraint_admn}
For every $i \in \{1,\ldots, N\}$ and $t \in [0,1]$:
\begin{equation*}  
x_i(t) = x_i(0) + \int_0^t v_i(s) \, \rd s.
\end{equation*}
\end{enumerate}
Notice that $\mu_0$ and $\mu_1$ have no atoms since $\Sigma(\mu_0)$ and $\Sigma(\mu_1)$ are finite, which ensures 
that $(x_i(0))_i, (x_i(1))_i$ in eq.~\eqref{eq:def_xi_01} are well-defined and strictly increasing.
Within the set $\adm_N(\mu_0, \mu_1)$, the trajectories $x_1,\ldots, x_N$ are affine functions of the trajectories $v_1,\ldots, v_N$. 
By~\ref{velocity_constraint_admn} we must have:
\begin{equation}\label{eq:boundary_v_admn}
        \int_0^1 v_i (s) \, \rd s = x_i(1) - x_i(0) \hspace{10pt} \textrm{ for all } i \in [N].
\end{equation}
Moreover, \ref{noncrossing_constraint_admn} imposes that for every $i \in \{1,\ldots, N-1\}$ and $t \in [0,1]$:
\begin{equation}\label{eq:noncrossing_v_admn}
\int_0^t (v_{i+1}(s) - v_i(s))\, \rd s \ge  x_{i}(0) - x_{i+1}(0).
\end{equation}
Conversely, any $(v_1, \cdots, v_N)$ satisfying eqs.~\eqref{eq:boundary_v_admn} and~\eqref{eq:noncrossing_v_admn} uniquely defines an element of $\adm_N(\mu_0, \mu_1)$.
Finally, for every $\{(x_i, v_i)\}_{i=1}^N \in \adm_N(\mu_0,\mu_1)$, and any $\eps > 0$, we define
\begin{equation}  \label{eq:def_GNeps}
G_{N,\eps}(\{(x_i, v_i)\}) 
\coloneqq \frac 1 {2N} \sum_{i = 1}^N \int_0^1\left( v_i(t)^2 + \frac{\pi^2}{3(N+1)^2} \left( x_{i+1}(t) - x_i(t) + \frac \eps {N+1} \right) ^{-2}  \right) \rd t,
\end{equation}
and we set
\begin{equation}\label{eq:def_JN}
J_{N,\eps}(\mu_0,\mu_1) \coloneqq \inf_{\adm_N(\mu_0,\mu_1)} G_{N,\eps}. 
\end{equation}

\subsubsection{Validity of the discretization scheme}

The following theorem shows that the discretization scheme introduced above
asymptotically approaches the value of the HCIZ integral.
\begin{theorem}[Validity of the numerical method]
\label{thm:main}
Let $\mu_0, \mu_1 \in \mcP_c(\bbR) \cap \mcP_\abc(\bbR)$ be
such that $\mu_0(\rd x) = \rho_0(x) \rd x$, $\mu_1(\rd x) = \rho_1(x) \rd x$ with $\rho_0, \rho_1 \in L^3(\bbR)$. 
We have 
\begin{equation*} 
\lim_{\eps \to 0} \liminf_{N \to \infty} J_{N,\eps}(\mu_0, \mu_1) = \lim_{\eps \to 0} \limsup_{N \to \infty} J_{N,\eps}(\mu_0, \mu_1) = J(\mu_0, \mu_1).
\end{equation*}
\end{theorem}
\noindent
Notice that the actual numerical scheme that we implement in Section~\ref{sec:numerics} also involves a discretization of the time variable into small intervals. 
It is however straightforward to verify that for each fixed $N$ and $\eps$, this fully discretized optimization problem converges to $J_{N,\eps}(\mu,\nu)$ as the time-step gaps are sent to zero,
and we focus on proving Theorem~\ref{thm:main}.

\subsubsection{First applications}

In Section~\ref{sec:numerics}, we present an efficient numerical scheme, which involves an additional time discretization step, 
and is based on the Newton update combined with a preconditioned conjugate gradient method. 
We also discuss potential alternatives, as well as limitations of our scheme, in Section~\ref{sec:numerics}. 
Here, we present some applications, focusing first on the known ``benchmark'' cases mentioned above.

\myskip 
\textbf{1. Free convolution~\ref{case:free_conv} --}
We consider~\cref{case:free_conv}, i.e.\ $\nu = \mu \boxplus \sigma_{\sci}$. 
\begin{figure}[!htbp]
    \centering
    \includegraphics[width=1.0\textwidth]{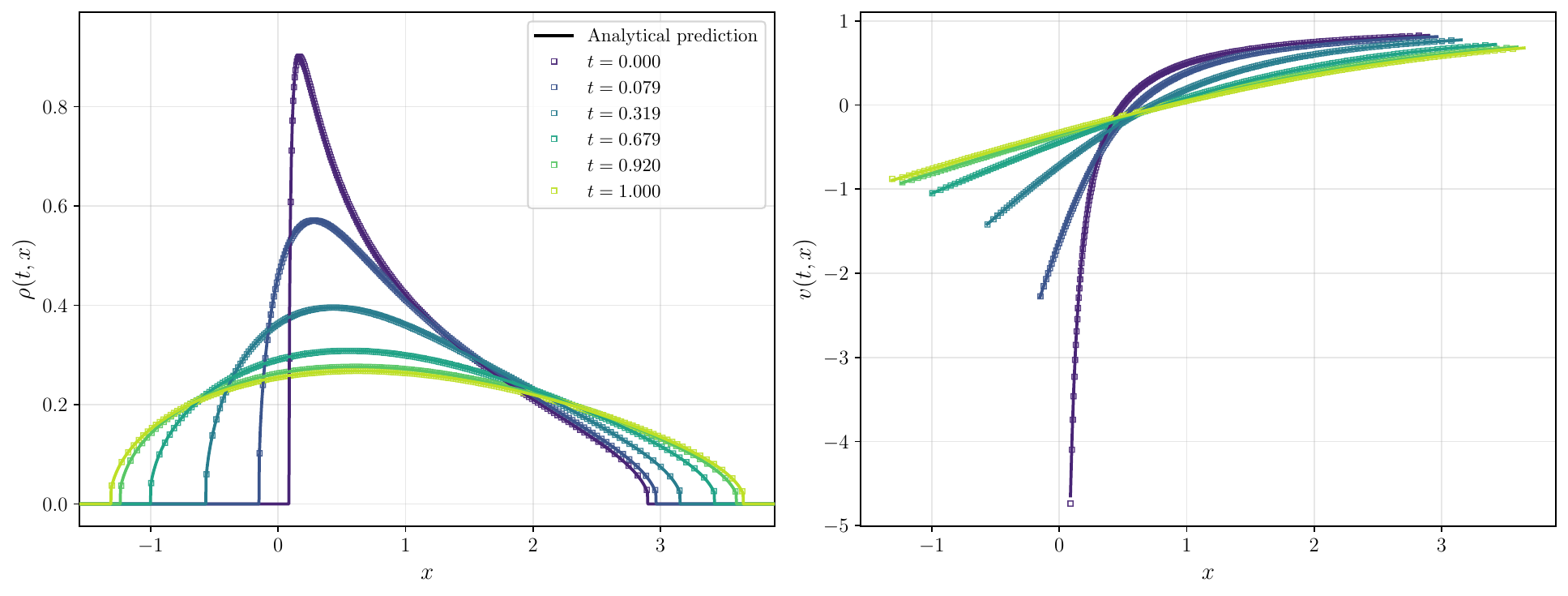}
    \caption{
        The density $\rho(t, x)$ (left) and the velocity field $v(t, x)$ (right) when $\mu = \mu_{\MP, \kappa}$ is the Marchenko-Pastur density with ratio $\kappa = 2$ (see eq.~\eqref{eq:mp_distribution}),
        and $\nu = \mu \boxplus \sigma_{\sci}$. We show the values of the density and velocity at various times $t \in [0,1]$. Solid lines are the analytical predictions detailed in~\ref{case:free_conv}, 
        and squares represent a sample of $200$ particles. 
        For this plot, we used $N = 2^{12}$ total particles, and $T = 2^{10}$ time steps, see Section~\ref{sec:numerics} for more details.
    }
    \label{fig:benchmark_freeconv}
\end{figure}
In Fig.~\ref{fig:benchmark_freeconv}, we show the density and velocity evolution in the example where $\mu = \mu_{\MP, \kappa}$ is the Marchenko-Pastur density with ratio $\kappa \geq 1$~\cite{marchenko1967distribution}:
\begin{equation}\label{eq:mp_distribution}
\frac{\rd\mu_{\MP, \kappa}}{\rd x}=
\frac{\kappa}{2\pi}
\frac{\sqrt{(\lambda_+-x)(x-\lambda_-)}}{x} \indi\{\lambda_- \leq x \leq \lambda_+\},
\qquad
\lambda_\pm\coloneq(1\pm\kappa^{-1/2})^2,
\end{equation}
and we compare our numerical solver to the free convolution 
$\rho_t = \mu \boxplus \sigma_{\sci, \sqrt{t}}$ and the velocity field $v$ given by eq.~\eqref{eq:vt_free_conv}.
We find excellent agreement with the predictions.

\myskip 
\textbf{2. Semicircular laws~\ref{case:semicircle} --}
In Fig.~\ref{fig:benchmark_semicircle}, we benchmark our method in the case where $\mu_0 = \mu_1 = \sigma_\sci$, 
while varying the scaling parameter $\theta$.
\begin{figure}[!htbp]
    \centering
    \includegraphics[width=0.75\textwidth]{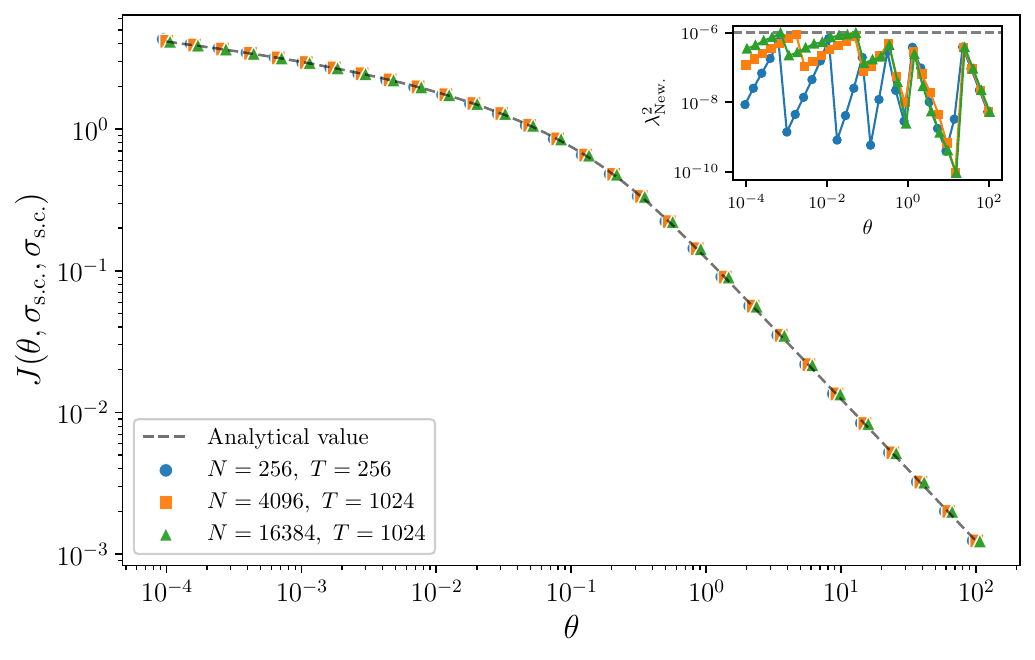}
    \caption{
        The value of $J(\theta, \sigma_{\sci}, \sigma_{\sci})$, as computed by Algorithm~\ref{algo:newton_cg}, against the asymptotic value of eq.~\eqref{eq:J_theta_sc_to_sc}.
        This comparison is perfomed for a wide range of values of $\theta$, and different values of $(N,T)$.
        The inset shows the Newton decrement at the solution (see Section~\ref{sec:numerics}).
    }
    \label{fig:benchmark_semicircle}
\end{figure}
We compare the predictions to the analytical result of~\cite{bun2014instanton}, cf.\ eq.~\eqref{eq:HCIZ_analytical_sc_to_sc}.
 Equivalently, by eq.~\eqref{eq:Itheta_Jtheta}:
\begin{equation}\label{eq:J_theta_sc_to_sc}
    J(\theta, \sigma_{\sci}, \sigma_{\sci}) = \theta - \frac{1}{2} \log \theta - \frac{\sqrt{1+4\theta^2}}{2} + \frac{1}{2} \log \left(\frac{1+\sqrt{1+4\theta^2}}{2}\right).
\end{equation}
We find excellent agreement, validating here as well our method, for both large and small values of the scaling parameter $\theta$, 
and already at relatively small values of $N, T$.

\myskip
\textbf{3. Singular boundary measures --}
For $\mu = \mu_{\MP,1}$ as in \eqref{eq:mp_distribution} and $\nu = \mu \boxplus \sigma_{\sci, 1/\theta}$, we compute $I_\HCIZ(\theta, \mu, \nu)$ for a range of values of $\theta$ 
    and compare to the analytical predictions. 
    This is Fig.~\ref{fig:benchmark_theta_mp_free_convolution}.
    Note that here the density of $\mu$ is not in $L^3$ since $\rd\mu(x)/\rd x \sim x^{-1/2}$ as $x \downarrow 0$, but the formula in \eqref{eq:I_conv}, suitably generalized to arbitrary $\theta$, is still valid. We can thus check that our solver also performs very well in this case, 
    suggesting that the $L^3$ assumption in Theorem~\ref{thm:main} might not be necessary.
\begin{figure}[!htbp]
    \centering
    \includegraphics[width=0.75\textwidth]{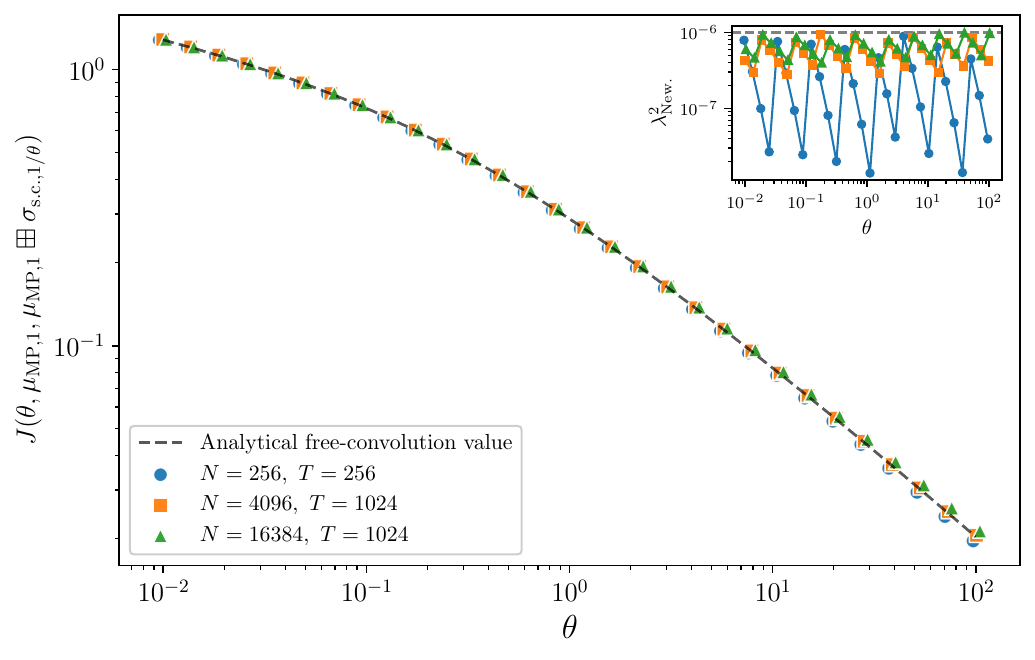}
    \caption{
        Similar quantities as in Fig.~\ref{fig:benchmark_semicircle}, 
        but for $\mu = \mu_{\MP,1}$ and $\nu = \mu_{\MP,1} \boxplus \sigma_{\sci, 1/\theta}$.
    }
    \label{fig:benchmark_theta_mp_free_convolution}
\end{figure}

\myskip 
\textbf{4. Application to generic densities --}
Crucially, our method applies to generic boundary measures $(\mu, \nu)$ for which the value of $I_\HCIZ(\mu, \nu)$ is not known. 
\begin{figure}[!htbp]
    \centering
    \includegraphics[width=1.0\textwidth]{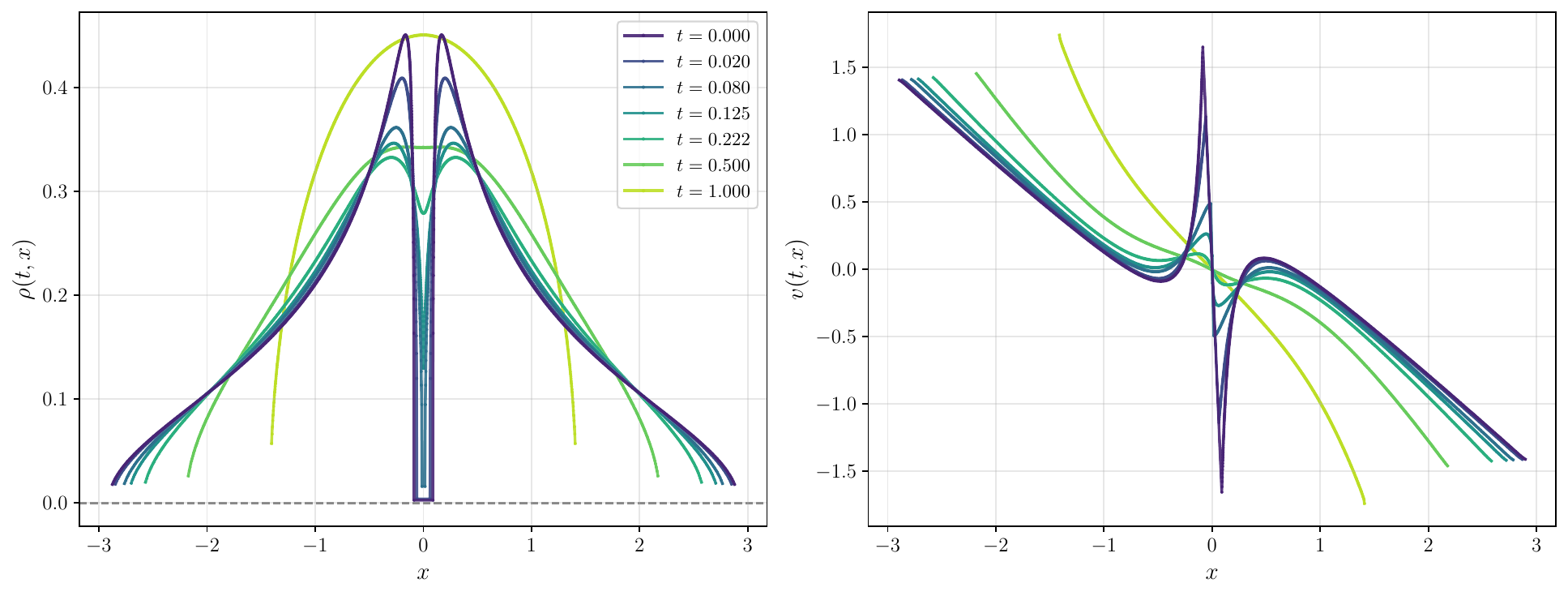}
    \caption{
        The density $\rho(t, x)$ (left) and the velocity field $v(t, x)$ (right) when $\mu$ is the \emph{symmetrized} Marchenko-Pastur density with ratio $\kappa = 2$ (see eq.~\eqref{eq:sym_MP}), and 
        $\nu = \sigma_{\sci, 0.5}$, a semicircle law with variance $1/2$.
        We show the values of the density and velocity at various times $t \in [0,1]$. 
        Squares represent a sample of the particles, and the solid lines are linear interpolations of the density and velocity evaluated from the numerical discretization. 
        For this plot, we used $N = 2^{12}$ total particles, and $T = 2^{10}$ time steps, see Section~\ref{sec:numerics} for more details.
    }
    \label{fig:sym_mp_to_sc}
\end{figure}
As a first example, we show in Fig.~\ref{fig:sym_mp_to_sc} 
the trajectories found by our numerical solver when $\nu = \sigma_{\sci,1/2}$ is a semicircle with variance $1/2$, and $\mu$ is the \emph{symmetrized} Marchenko-Pastur density for 
$\kappa \geq 1$: 
\begin{equation}\label{eq:sym_MP}
    \frac{\rd\mu}{\rd x}=
    \frac{\kappa}{4\pi}
    \frac{\sqrt{(\lambda_+-|x|)(|x|-\lambda_-)}}{|x|} \indi\{\lambda_- \leq |x| \leq \lambda_+\},
    \qquad
    \lambda_\pm\coloneq(1\pm\kappa^{-1/2})^2,
\end{equation}
Notice that the initial measure $\mu$ has disconnected support components, which merge during the evolution at around $t_{\rm merge} \simeq 0.08$. We also witness the behavior of the velocity field $v$ around the merging time numerically; since the components are merging, the velocity field cannot be Lipschitz continuous, and we witness its rather singular behavior numerically. 
In this setting, no analytical solution is known to the variational problem of eq.~\eqref{eq:def_J}, and our work
is the first to provide a simple numerical scheme to evaluate the asymptotic value of the HCIZ integral and the corresponding trajectories.

\myskip 
\textbf{5. Singular matrices and regularization --}
Let $\mu, \nu \in \mcP_c(\bbR)$, 
and assume that either $\mu$ or $\nu$ is singular: for concreteness we assume $\Sigma(\mu) = -\infty$ while $\Sigma(\nu) > -\infty$.
While there is no known asymptotic formula for $I_\HCIZ(\mu, \nu)$
in this case,
one can use the continuity of $I_\HCIZ(\mu, \nu)$ given in Theorem~\ref{thm:limit_hciz}
to write 
\begin{equation*}
    I_\HCIZ(\mu, \nu) = \lim_{\eta \to 0} I_\HCIZ(\mu_\eta, \nu),
\end{equation*}
where $\mu_\eta$ is any regularization, that is compactly supported (uniformly in $\eta$), and such that $\mu_\eta \to \mu$ as $\eta \downarrow 0$ in the sense of weak convergence.
Applying our numerical scheme to the regularized measure allows to access the value of $I_\HCIZ(\mu, \nu)$ even for strongly singular measures: we illustrate this 
in Fig.~\ref{fig:mp_boxplus_eps_to_sc}, where we consider $\mu= \mu_{\MP, \kappa = 0.5}$, which has an atom of mass $1/2$ in $x = 0$.
\begin{figure}[!htbp]
    \centering
    \begin{subfigure}[t]{0.5\textwidth}
        \centering
        \includegraphics[width=1.0\textwidth]{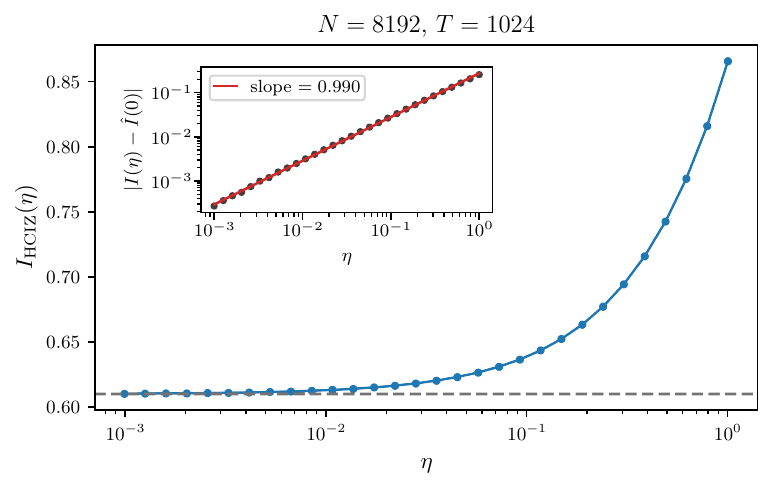}
        \caption{
        The value of $I_\HCIZ(\eta)$  as a function of $\eta$ for a large value of $N, T$. We estimate the value at $\eta = 0$ from 
        a power-law fit $I_\HCIZ(\eta) \sim \hat{I}(0) + A \eta^\alpha$, and the inset suggests a linear decay ($\alpha = 1$).
        }
    \end{subfigure}%
    ~ 
    \begin{subfigure}[t]{0.5\textwidth}
        \centering
        \includegraphics[width=1.0\textwidth]{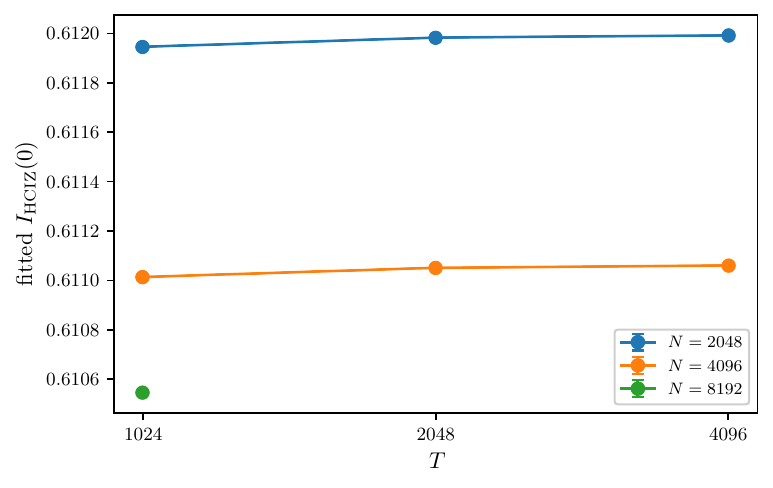}
        \caption{
    The value of $I_\HCIZ(0)$, as fitted from $I_\HCIZ(\eta)$ for small $\eta$ (see the left figure), for various values of $N$ and $T$. Notice the small vertical range of order $10^{-3}$.
        }
    \end{subfigure}
    \caption{
            Let $\mu = \mu_{\MP,1/2}$ and $\nu = \sigma_{\sci}$. 
            We let $\mu_\eta \coloneqq \mu \boxplus \sigma_{\sci, \eta}$ for $\eta > 0$ and 
            compute $I_\HCIZ(\eta) \coloneqq I_\HCIZ(\mu_\eta, \nu)$. 
    \label{fig:mp_boxplus_eps_to_sc}
    }
\end{figure}

\subsection{Discussion and consequences}\label{subsec:discussion_consequences}

We discuss here some properties of our discretization scheme and consequences of our analysis. 
A more detailed discussion of the implementation detail, as well as an analysis of convergence properties and computation time, 
is postponed to Section~\ref{sec:numerics}.

\myskip 
\textbf{Mathematical remarks --}
We make here a few remarks stemming from these first examples.
\begin{itemize}[leftmargin=*]
    \item 
    Perhaps the most transparent derivation of Theorem~\ref{thm:limit_hciz} is through the large deviations of the Dyson Brownian motion~\cite{guionnet2002large,guionnet2004addendum,bun2014instanton}. 
    At finite $N$, the corresponding particle action contains the
    nonlocal inverse-square interaction
    \begin{equation}\label{eq:non_local}
        \Psi_{N,i}(\boldsymbol{\lambda})
        \coloneqq
        \frac{1}{N^2}
        \sum_{j\neq i}
        \frac{1}{(\lambda_i-\lambda_j)^2},
    \end{equation}
    where $\lambda_1<\cdots<\lambda_N$ are the eigenvalues evolving under the
    Dyson Brownian motion (see e.g.~\cite{anderson2010introduction}).
    Heuristically, assuming that $(1/N)\sum_{i=1}^N\delta_{\lambda_i}
    \to \rho(x)\,\mathrm{d}x$, 
    for an index $i$ such that $\rho(\lambda_i)>0$, the local quantile
    spacing satisfies, for $k$ small compared with $N$,
    \begin{equation}\label{eq:linearization_density}
        \lambda_{i+k}-\lambda_i
        \simeq
        \frac{k}{N\rho(\lambda_i)}.
    \end{equation}
    Substitution into eq.~\eqref{eq:non_local} gives the approximation
    \begin{equation}\label{eq:inverse_square_local_limit}
        \Psi_{N,i}(\boldsymbol{\lambda})
        \simeq
        \rho(\lambda_i)^2
        \sum_{k\in\mathbb{Z}\setminus\{0\}}\frac{1}{k^2}
        =
        \frac{\pi^2}{3}\rho(\lambda_i)^2.
    \end{equation}
    Indeed, particles at macroscopic distance contribute only
    $O(N^{-1})$ to eq.~\eqref{eq:non_local}, whereas the singular
    near-neighbor terms contribute at order one. Averaging over the
    particles then yields
    \begin{equation}
        \frac{1}{N}\sum_{i=1}^N
        \Psi_{N,i}(\blambda)
        \simeq
        \frac{\pi^2}{3}
        \int \rho(x)^3\,\rd x,
    \end{equation}
    which explains the term proportional to $\rho^3$ in eq.~\eqref{eq:def_J}.
    This (very) heuristic argument is at the heart of the physics derivation of Theorem~\ref{thm:limit_hciz}, see~\cite{matytsin1994large,bun2014instanton}.
    Our discretization in eq.~\eqref{eq:def_GNeps} does not retain or truncate the full interaction in eq.~\eqref{eq:non_local}. Instead,
    we directly discretize its continuum limit and use nearest-neighbor spacings to approximate the density while, preserving
    the coefficient $\pi^2/3$ in eq.~\eqref{eq:inverse_square_local_limit}. 
    This makes our numerical scheme completely local in space, 
    and simplifies its implementation and analysis.
    \item The assumption in Theorem~\ref{thm:main} that $\rho(0, \cdot), \rho(1, \cdot) \in L^3$ can probably be weakened, as seen in the example of Fig.~\ref{fig:benchmark_theta_mp_free_convolution}. In any case, singular boundary densities can be appropriately regularized, as seen in the example of Fig.~\ref{fig:mp_boxplus_eps_to_sc}. A more quantitative understanding of the tradeoffs involved in regularizing the density would be interesting to explore.
    \item 
    Fig.~\ref{fig:sym_mp_to_sc} illustrates a topological change in the
    support of the density: the two connected components of
    $\operatorname{supp}\rho(0,\cdot)$ merge during the evolution.
    The numerical results do not determine whether the velocity field
    $v(t,\cdot)$ remains continuous at the merging time. 
    More generally, investigating the regularity properties of the solutions, in particular near such critical times, is an interesting mathematical question, and a likely necessary ingredient for a sharp quantitative understanding of convergence rates.
\end{itemize}

\myskip 
\textbf{Generalization: free boundary --}
There are several natural extensions and generalizations of our approach.
Among them is the numerical evaluation of variational principles in which one of the boundary measures entering the HCIZ functional is itself optimized.
More precisely, given a confining potential $F$ and a fixed measure $\nu\in\mcP(\bbR)$, one may consider
\begin{equation}\label{eq:free_boundary}
    L(\theta,F,\nu)
    \coloneqq
    \sup_{\mu\in\mcP(\bbR)}
    \left\{
        \frac{1}{2}\Sigma(\mu)
        -\frac{1}{2}\int_{\bbR} F(x)\,\mu(\mathrm{d}x)
        +\frac{1}{2}I_{\HCIZ}(\theta,\mu,\nu)
    \right\},
\end{equation}
under suitable assumptions ensuring that the functional is well defined and that the supremum is finite.
Variational problems of this form arise in several contexts discussed in Section~\ref{subsubsec:related_literature}, including spectral large-deviation problems for random matrices, the computation of free energies in disordered systems, and problems in high-dimensional statistics and machine learning.
At the level of our particle discretization, the optimization over $\mu$ can be incorporated by treating the initial particle positions $x_i(0)$ as additional optimization variables, while keeping the terminal configuration associated with $\nu$ fixed.
The discrete objective must then be augmented by consistent discretizations of the boundary terms $\Sigma(\mu)$ and $\int F\,\mathrm{d}\mu$.
This leads to a free-endpoint version of the algorithm introduced in Section~\ref{sec:numerics}.
In an extended version of this work, we plan to study this extension systematically, including its numerical stability and convergence properties, and to analyze some of its first consequences for the applications mentioned above.

\myskip 
\textbf{Alternative approaches.---}
We conclude by discussing alternative approaches to the numerical
computation of the high-dimensional asymptotics of the HCIZ integral.
A natural possibility is to solve directly the Euler equation
in eq.~\eqref{eq:euler_2nd_equation}, coupled with the continuity equation
in eq.~\eqref{eq:continuity}.
Defining
\begin{equation*}
    f(t,x) \coloneqq v(t,x) + \mathrm{i}\pi\rho(t,x),
\end{equation*}
these two equations are equivalent to the complex inviscid Burgers' equation
\begin{equation}\label{eq:complex_burgers}
    \partial_t f + f\,\partial_x f = 0.
\end{equation}
The corresponding real system for $(\rho,v)$ has a rather complex structure:
it is elliptic wherever $\rho>0$ and hyperbolic on $\{\rho = 0\}$, see~\cite{menon2017complex}. 
Consequently, the initial-value problem (i.e.\ evaluating $(\rho(1, \cdot), v(1, \cdot))$ from $(\rho(0, \cdot), v(0, \cdot))$) is generally ill-posed in standard function spaces, although
some solutions may be constructed under analyticity assumptions.
For the same reason, a direct shooting method for the boundary-value
problem is expected to be numerically unstable.

\myskip
Another idea is to extend eq.~\eqref{eq:complex_burgers}, and look for a function $f(t, z)$ that is holomorphic on a suitable domain and satisfies
\begin{equation*}
    \partial_t f + f\,\partial_z f = 0.
\end{equation*}
The method of complex characteristics then yields an implicit representation
of the solution and recovers several of the explicit solutions mentioned above, as discussed in~\cite{matytsin1994large,bun2014instanton,menon2017complex}.
For general boundary densities, however, the function $f(t, x)$ cannot be analytically extended to the whole complex plane, 
which precludes the use of the method of characteristics, see~\cite{menon2017complex}.

\myskip
In~\cite{menon2017complex}, an interesting relation to integrable systems was unveiled.
The author showed that the most probable trajectory under the constrained Dyson Brownian Motion (the ``instanton'' in the language of~\cite{bun2014instanton}) 
can be exactly mapped to the attractive Calogero-Moser system.
Given initial positions and velocities $(x_i(0), v_i(0))$, this yields a formula of the type $x(t) = \mathrm{eig}(M(t))$, for a real non-symmetric matrix 
$M(t)$ that depends only on the initial conditions. However, in general $M(t)$ ceases to have real eigenvalues for a large enough $t \in [0,1]$, 
except for a fine-tuned choice of the initial velocities. 
This difficulty was noticed in~\cite{menon2017complex}, and prevented the author from leveraging this idea for a numerical solution to the boundary-value problem.

\myskip
Perhaps the most promising alternative to our particle discretization is to
solve directly the convex optimization problem in
eq.~\eqref{eq:def_J} using Eulerian techniques developed for dynamic optimal
transport.
If the term proportional to $\rho^3$ is
removed in eq.~\eqref{eq:def_J}, the resulting minimization problem is precisely the dynamic
Benamou--Brenier formulation of quadratic optimal transport~\cite{benamou2000computational}. 
Benamou and Brenier discretized this problem on a space--time grid and solved the resulting
convex problem using an augmented-Lagrangian method. These methods can be adapted to handle additional 
terms similar to our $\rho^3$, see e.g.~\cite{benamou2015augmented}.
We refer to~\cite[Chapter~7]{peyre2019computational} for more details on the fluid-dynamics formulation of optimal transport, 
and numerical schemes tailored to this formulation. 
Our particle formulation, on the other hand, does not require a fixed
spatial grid and exhibits strong numerical performance even for non-smooth densities and velocity fields 
(cf.\ the examples considered above), while this might cause difficulties in a grid-discretization approach. For instance, for each fixed $t \in (0,1)$, the density $\rho(t,\cdot)$ vanishes rather abruptly, typically with a square-root singularity, and this may require specific attention in a grid-discretization approach. 
An implementation of this alternative method for the asymptotics of the HCIZ integral, and a quantitative comparison to our algorithm, is an interesting future direction.

\subsection{Structure of the paper}

In Section~\ref{sec:numerics}, we present in detail the minimization algorithm used to compute the discretized version of the problem whose validity 
is established in Theorem~\ref{thm:main}.
In Section~\ref{sec:proof}, we present the proof of Theorem~\ref{thm:main}.
We detail some additional analytical considerations on the HCIZ integral and Matytsin's solution, as well as some technical aspects 
of the proof and derivation, 
in Appendix~\ref{sec_app:properties}.
Finally, some additional numerical experiments on the convergence properties, and the runtime of our algorithm, are presented in Appendix~\ref{sec_app:additional_numerics}.

\myskip
\textbf{Acknowledgements and use of AI --}
We thank Alice Guionnet, Govind Menon, Gabriel Peyré and Adrien Taylor for interesting discussions at various stages of this work.
We used modern AI tools, primarily GPT-5.6 Sol as well as earlier versions of GPT and Claude Opus~5, to help designing and improving the numerical code accompanying this work.
The manuscript itself was written entirely by the authors, and AI tools were used only to proofread the manuscript.
The authors take full responsibility for the content and for any remaining errors.

\section{Efficient minimization algorithm}\label{sec:numerics}
We detail here the numerical procedure used to solve the finite-dimensional variational problem obtained from Theorem~\ref{thm:main}.

\myskip
\textbf{Time discretization --}
We start by defining a simple time discretization scheme.
For any $T \geq 1$ and $\Dt \coloneqq (\Delta t_k)_{k=1}^T$ such that $\Delta t_k \geq 0$ and $\sum_{k=1}^T \Delta t_k = 1$, we define
$\adm_{N, \Dt}(\mu_0, \mu_1)$
to be the set of all trajectories $\{(x_{i,k}, v_{i,k})\}$ for $1 \leq i \leq N$ and $0 \leq k \leq T$, satisfying the following.
\begin{enumerate}[label=\textbf{($\tbH$\arabic*)},ref=($\tH$\arabic*)]
\item
\label{bound_constraint_admNT}
$(x_{i,0}) = x_i(0)$ and $x_{i,T} = x_i(1)$ as given by eq.~\eqref{eq:def_xi_01}.
\item
\label{noncrossing_constraint_admnT}
For every $i \in \{1,\ldots, N\}$ and $k \in \{0, \cdots, T\}$, we have $x_{i,k} \le x_{i+1,k}$.
\item 
\label{velocity_constraint_admnT}
For every $i \in \{1,\ldots, N\}$ and $k \in \{0, \cdots, T\}$:
\begin{equation*} 
x_{i,k} = x_{i,0} + \sum_{l=1}^k \Delta t_l v_{i,l}.
\end{equation*}
\end{enumerate}
We also let for $k \in \{0, \cdots, T\}$:
\begin{equation}\label{eq:def_wk_Dt}
    w_k \coloneqq 
    \begin{dcases}
        \frac{\Dt_k}{2} &\hspace{1cm} \textrm{if} \, k \in \{0, T\}, \\
        \frac{\Dt_k + \Dt_{k+1}}{2} &\hspace{1cm} \textrm{if} \, 1 \leq k \leq T-1.
    \end{dcases}
\end{equation}
Then, for any $\{(x_{i,k}, v_{i,k})\} \in \adm_{N,\Dt}(\mu_0, \mu_1)$, we let
\begin{align}
     \label{eq:def_GNDt_eps}
    G_{N,\Dt,\eps}(\{x_{i,k}, v_{i,k}\}) 
    &\coloneqq
    \frac 1 {2N} \sum_{k=1}^T \sum_{i=1}^N \Dt_k v_{i,k}^2  \\
    \nonumber
    &+ \frac{\pi^2}{6N(N+1)^2} \sum_{k=0}^T \sum_{i=1}^N w_k \left( x_{i+1,k} - x_{i,k} + \frac \eps {N+1} \right) ^{-2}.
\end{align}
Notice that we use a trapezoidal discretization scheme for the second term in eq.~\eqref{eq:def_GNDt_eps}. 
As we detail in Appendix~\ref{subsec_app:time_discretization_scheme}, the objective is to have a global time discretization error in $G_{N, \Dt, \eps}(\{x_{i,k}, v_{i,k}\})$ of size $\mcO(\|\Dt\|_\infty^2) = \mcO(1/T^2)$, which is consistent with what we numerically 
observe (experiments on this point are given in Appendix~\ref{subsec_app:runtime_analysis}).
We set
\begin{equation}\label{eq:def_JNDt}
J_{N,\Dt,\eps}(\mu_0,\mu_1) \coloneqq \inf_{\adm_{N, \Dt}(\mu_0,\mu_1)} G_{N,\Dt,\eps}. 
\end{equation}
As mentioned above, it is straightforward to check (by discretizing a near-optimal continuous trajectory on one hand, and building an affine interpolation of a discrete trajectory on the other hand, similarly to the proof shown in Section~\ref{sec:proof}) that
\begin{equation*}
     \lim_{\substack{T \to \infty \\ \|\Dt\|_{\infty} \to 0}} J_{N,\Dt,\eps}(\mu_0, \mu_1) = J_{N, \eps}(\mu_0, \mu_1).
\end{equation*}
Equation~\eqref{eq:def_JNDt} defines a \emph{convex} minimization problem over a set of linear constraints (equalities and inequalities). 
We now detail the specifics of the minimization procedure and algorithm.

\subsection{Parametrization by the spacings}

We first change variables, to express the objective solely as a function of two variables: the (rescaled) particle spacings
\begin{align}\label{eq:rescaled_spacings}
    \delta_{i,k} \coloneqq (N+1)(x_{i+1,k} - x_{i, k}), \hspace{20pt} (1 \leq i \leq N - 1, \, 0 \leq k \leq T),
\end{align}
and the position of the first particle $(x_{1,k})_{k=0}^{T}$. 
This allows us to rewrite the objective in eq.~\eqref{eq:def_GNDt_eps} as (notice that we removed the term $i = N$ from the first sum, since $x_{N+1, k} = + \infty$ by convention):
\begin{align}
    \label{eq:G_NDt_2}
    &G_{N, \Dt, \eps}(\delta, x_1) \\ 
    \nonumber
    &= \frac{\pi^2}{6N} \sum_{k=0}^T \sum_{i=1}^{N-1} \frac{w_k}{(\delta_{i,k} + \eps)^2} +  \frac 1 {2N} \sum_{k=1}^T \frac{1}{\Delta t_k} \sum_{i = 1}^N
    \left[x_{1,k} - x_{1,k-1} + \frac{1}{N+1} \sum_{j=1}^{i-1}(\delta_{j,k} - \delta_{j,k-1})\right]^2.
\end{align}
One should minimize eq.~\eqref{eq:G_NDt_2} over $(x_{1, k})_{k=1}^{T-1}$ and $(\delta_{i,k})_{ \substack{1 \leq i \leq N-1 \\ 1 \leq k \leq T-1}}$, 
under the sole constraint that $\delta_{i,k} \geq 0$, and recalling that the boundary values are fixed as:
\begin{align}\label{eq:boundary_delta_x1}
    \begin{dcases}
        x_{1,0} &= x_1(0), \\
        x_{1,T} &= x_1(1), \\
        \delta_{i,0} &= (N+1)(x_{i+1}(0) - x_i(0)), \\
        \delta_{i,T} &= (N+1)(x_{i+1}(1) - x_i(1)),
    \end{dcases} 
\end{align}
where $(x_i(0), x_i(1))$ are given by eq.~\eqref{eq:def_xi_01}.
Parameterizing by the spacings has thus simplified the linear constraints to $\delta_{i,k} \geq 0$, which imposes the non-crossing of the particle trajectories.
Further, the minimization over $(x_{1, k})_{k=1}^{T-1}$ is a simple quadratic optimization problem which we can solve analytically. 
This yields (using the same notation $G_{N, \Dt, \eps}$ slightly abusively):
\begin{align}\label{eq:G_NDt_3}
    &G_{N, \Dt, \eps}(\delta) \\ 
    \nonumber
    &= \frac{1}{2} \left[\ox(0) - \ox(1) \right]^2 + \frac{\pi^2}{6N} \sum_{k=0}^T \sum_{i=1}^{N-1} \frac{w_k}{(\delta_{i,k} + \eps)^2} + \frac{1}{2 (N+1)^2} \sum_{k=1}^T \Delta t_k \, (\rd \delta_k^\T) L_{N-1}^{-1} (\rd \delta_k), 
\end{align}
where we defined the following quantities: 
\begin{align}\label{eq:defs_G_NDt_3}
    \begin{dcases}
    \ox(t) &\coloneqq \frac{1}{N} \sum_{i=1}^N x_i(t), \hspace{10pt} (t \in \{0, 1\}) \\
    (\rd \delta)_{i,k} &\coloneqq \frac{\delta_{i,k} - \delta_{i,k-1}}{\Delta t_k}, \hspace{10pt} (1 \leq i \leq N-1 \textrm{ and } 1 \leq k \leq T)\\ 
    L_{N-1} &\coloneqq N
    \begin{pmatrix}
        2 & -1 &0 &\cdots &0 \\
        -1 & 2 &-1 &\cdots &0 \\
        0 & -1 &2 &\cdots &0 \\
        \vdots & \vdots &\vdots &\ddots &\vdots \\
        0 & 0 &0 &\cdots &2
    \end{pmatrix}
    \in \bbR^{(N-1) \times (N-1)}.
    \end{dcases}
\end{align}
We detail the derivation of eq.~\eqref{eq:G_NDt_3} for completeness in Appendix~\ref{subsec_app:eq_G_NDt_3}.
The first term in eq.~\eqref{eq:G_NDt_3} is a deterministic function of the boundary densities, and is not part of the optimization objective.

\myskip 
\textbf{On the constraint $\delta_{i,k} \geq 0$ --}
Notice that $G_{N, \Dt, \eps}(\delta)$ is strongly convex on its domain $\{\delta_{i,k} \geq 0\}$ and therefore has a unique minimum.
Rather than imposing the constraint $\delta_{i,k} \geq 0$, we extend the domain of definition to all $\delta_{i,k} \in \bbR$ and replace $(\delta_{i,k}+\eps)^{-2}$ by $\sigma_\eps(\delta_{i,k})$, with 
\begin{align}\label{eq:def_sigmaeta}
    \sigma_\eps(x) &\coloneqq \begin{dcases}
        (x+\eps)^{-2} &\textrm{ if } x \geq 0, \\
        6 \eps^{-2} -  8 \eps^{-3} (x+\eps) + 3 \eps^{-4} (x+\eps)^2 &\textrm{ if } x < 0.
    \end{dcases}
\end{align}
It is trivial to check that $\sigma_\eps$ is of class $\mcC^2$ and convex over $\bbR$. 
In practice, we choose $\eps > 0$ very small ($10^{-7}$ in the simulations presented), and we check that, at the output of the algorithm, we have $\delta_{i,k} \geq 0$, which ensures that the minimum we reach is indeed the global minimum of $G_{N, \Dt, \eps}$.
This will allow us to use methods designed for \emph{unconstrained} convex minimization problems.
To sum up, we extend the definition of $G_{N,\Dt,\eps}$ to arguments that are not necessarily in the nonnegative quadrant, and minimize
\begin{align}\label{eq:G_NDt_4}
    &G_{N, \Dt, \eps}(\delta) \\ 
    \nonumber
    &= \frac{1}{2} \left[\ox(0) - \ox(1) \right]^2 + \frac{\pi^2}{6N} \sum_{k=0}^T \sum_{i=1}^{N-1} w_k \sigma_\eps(\delta_{i,k}) + \frac{1}{2 (N+1)^2} \sum_{k=1}^T \Delta t_k \, (\rd \delta_k^\T) L_{N-1}^{-1} (\rd \delta_k), 
\end{align}
over all $\delta_{i,k} \in \bbR$.

\subsection{Time discretization}\label{subsec:time_schedule}

We now detail the choice of $\Dt$ used in the algorithm.
To motivate it, we use a free probability interpretation of
the solution $(\rho,v)$ in Theorem~\ref{thm:limit_hciz};
we refer to~\cite{anderson2010introduction} for an introduction to free probability in the context of random matrices. As explained in~\cite[p.~535]{guionnet2004first}, under the setting of Theorem~\ref{thm:limit_hciz}, there exist self-adjoint operators $X_0$, $X_1$, $S$ in a common $W^*$-probability space such that $X_0$ has law $\mu$, $X_1$ has law $\nu$, $S$ is a semi-circular variable that is free from $(X_0, X_1)$, and for every $t \in [0,1]$, the measure
$\rho(t, x) \, \rd x$ is the law of the free Brownian bridge
\begin{align}\label{eq:dbb}
    X_t \coloneqq (1-t) X_0 + t X_1 + \sqrt{t (1-t)} S.
\end{align}
In particular, for small times $t \ll 1$, $X_t$ is, to first order, close to the additive free convolution $X_0 + \sqrt{t} S$ (and similarly for times $t \to 1$). 
For a given value $T \geq 1$ (assumed even for simplicity), 
this motivates a non-linear time step schedule $(\Delta t_k)_{k=1}^T$, defined by:
\begin{equation}\label{eq:def_tk_sqrt_schedule}
    t_k \coloneqq
    \begin{dcases}
        \frac{2k^2}{T^2} &, \hspace{1cm} (k \in \{0, \cdots, T/2\}),\\
        \frac{4k}{T} - 1 -\frac{2k^2}{T^2} &, \hspace{1cm} (k \in \{T/2+1, \cdots, T\}).
    \end{dcases}
\end{equation}
\begin{equation}\label{eq:def_Dtk_sqrt_schedule}
    \Delta t_k = t_k - t_{k-1} =
    \begin{dcases}
        \frac{4}{T^2} \left(k - \frac{1}{2}\right) &, \hspace{1cm} (k \in \{1, \cdots, T/2\}),\\
        \Delta t_{T-k+1} &, \hspace{1cm} (k \in \{T/2+1, \cdots, T\}).
    \end{dcases}
\end{equation}

\subsection{Preconditioning}\label{subsec:preconditioning}

\subsubsection{An ill-conditioned minimization problem}

Notice that the Hessian of $G_{N,\Dt,\eps}$ is naturally ill-conditioned. 
We obtain from eq.~\eqref{eq:G_NDt_4} (we also state the gradient for completeness):
\begin{subnumcases}{\label{eq:grad_Hess_G_NDt}}
    \nabla_{i,k} \, G_{N,\Dt, \eps} = \frac{\pi^2}{6 N} w_k \sigma_\eps'(\delta_{i,k}) + \frac{1}{(N+1)^2} \left(\left[L_{N-1}^{-1} \otimes L_{T-1}(\Dt)\right] \delta\right)_{i,k}, & \label{eq:grad_G_NDt}\\
    \nabla^2 \, G_{N,\Dt, \eps} = \frac{\pi^2}{6 N} \Diag\left[w_k \sigma_\eps''(\delta_{i,k})\right] + \frac{1}{(N+1)^2} L_{N-1}^{-1} \otimes L_{T-1}(\Dt), & \label{eq:Hess_G_NDt}
\end{subnumcases}
where we define a Laplacian matrix for $p \geq 1$ and 
any $a = (a_1, \cdots, a_{p+1}) \in \bbR^{p+1}$ such that $a_i \geq 0$ and $\sum_{i=1}^{p+1} a_i = 1$:
\begin{align}
    \label{eq:def_Lp_a}
    L_p(a) &\coloneqq 
    \begin{pmatrix}
        a_1^{-1} + a_2^{-1} & -a_2^{-1} &0 &\cdots &0 \\
        -a_2^{-1} & a_2^{-1} + a_3^{-1} &-a_3^{-1} &\cdots &0 \\
        0 & -a_3^{-1} & a_3^{-1} + a_4^{-1}&\cdots &0 \\
        \vdots & \vdots &\vdots &\ddots &\vdots \\
        0 & 0 &0 &\cdots &a_p^{-1} + a_{p+1}^{-1} 
    \end{pmatrix} \in \bbR^{p \times p},
\end{align}
This is consistent with eq.~\eqref{eq:defs_G_NDt_3}, as we also set by convention 
\begin{equation*}
    L_p \coloneqq L_p\left(\left(\frac{1}{p+1}\right)_{i=1}^{p+1}\right).
\end{equation*}
From eq.~\eqref{eq:Hess_G_NDt}, we can see from an informal argument that the Hessian is generically ill-conditioned. 
It is well known that the eigenvalues of $L_{p-1}$ satisfy
\begin{align}\label{eq:evalues_Lp}
    \lambda_k(L_{p-1}) &= 4 p \sin^2 \left(\frac{k\pi}{2p}\right), \hspace{1cm} (1 \leq k \leq p-1)
\end{align}
Assuming $\Dt_k = 1/T$ for simplicity, this informal argument decomposes eq.~\eqref{eq:Hess_G_NDt} into two terms:
\begin{itemize}
    \item A diagonal and positive matrix arising from the $\rho^3$ term in eq.~\eqref{eq:def_G}, with eigenvalues of order $\Theta(N^{-1} T^{-1})$ assuming the spacings are of order $\delta_{ik} = \Theta(1)$.
    \item A second term arising from the $\rho v^2$ ``optimal transport'' objective: it is positive definite, with smallest eigenvalue of order $\Theta(N^{-3} T^{-1})$, and largest eigenvalue scaling as $\Theta(N^{-1} T)$.
\end{itemize}
From this heuristic argument, we expect the condition number of the Hessian to be $\Theta(T^2)$.  

\myskip
\textbf{Alternative parametrization --}
Notice that if we were to consider the simple optimal transport problem, we would drop the first term in eq.~\eqref{eq:Hess_G_NDt}, which would yield a very large Hessian condition number $\Theta(N^2 T^2)$. 
This apparent ill-conditioning would however only be due to a bad choice of parametrization: the optimal transport objective is proportional to $\int \rho v^2$, and its discretization is therefore quadratic in the velocities $v_{i,k}$: if we were to 
use these velocities as parameters, the ill-conditioning of the Hessian would disappear\footnote{Notice that some bad conditioning could still arise during optimization from the linear constraints arising from the non-crossing conditions, but we do not discuss this here.}.
However, this does not extend to our setting: using the velocities as parameters, the Hessian of the objective becomes 
\begin{equation}\label{eq:Hess_v}
    \nabla_v^2 G_{N,\Dt,\eps}
    =
    \frac{\pi^2}{6N}
    A_{N,T}^{\T}
    \Diag\left[
        \left(
            w_k\sigma_\eps''(\delta_{i,k})
        \right)_{\substack{1\leq i\leq N-1\\1\leq k\leq T-1}}
    \right]
    A_{N,T}
    +
    \frac{1}{N}
    \Diag\left[
        \left(\Delta t_k\right)_{k=1}^T
    \right]
    \otimes\Id_N,
\end{equation}
where 
$A_{N, T}\coloneqq (N+1)(\mcI_T\otimes D_N)$,
with $\mcI_T$ the discrete integration operator in time (restricted to the linear subspace of velocities which are compatible with the boundary conditions), and $D_N$ the first-difference operator in the particle index.
Informally, $A_{N, T}$ maps velocities to the corresponding spacings.
Considering constant time steps for simplicity, 
it is easy to see that the smallest and largest singular values of $A_{N,T}$ scale as 
\begin{equation}\label{eq:singular_values_A_NT}
    \sigma_{\min}(A_{N,T})
    =
    \Theta(T^{-1}),
    \qquad
    \sigma_{\max}(A_{N,T})
    =
    \Theta(N).
\end{equation}
Assuming that the spacings remain of order $\delta_{i,k} = \Theta(1)$ and bounded away from zero, eq.~\eqref{eq:Hess_v} thus yields a total condition number 
$\kappa(\nabla^2_v G_{N, \Dt, \eps}) = \Theta(N^2)$.
In the regimes considered numerically, where $N\gtrsim T$, the resulting condition number is no smaller than the $\Theta(T^2)$ condition number obtained in spacing variables.
Since parametrizing with the velocities also requires a careful handling of the non-crossing constraints, this is further motivation for considering a parametrization of the objective by the spacings.

\subsubsection{Preconditioning procedure}

In order to overcome this ill-conditioning issue, we use a preconditioning procedure that approximates the inverse Hessian of eq.~\eqref{eq:Hess_G_NDt}. 
We build this preconditioner by replacing the first, diagonal, term of eq.~\eqref{eq:Hess_G_NDt} by a block-diagonal average of its values, i.e.\ we define:
\begin{align}\label{eq:def_H}
    \begin{dcases}
        H(\delta) &\coloneqq 
        \Id_{N-1} \otimes \Diag[(\Gamma_k(\delta_k))_{k=1}^{T-1}] + \frac{1}{(N+1)^2} L_{N-1}^{-1} \otimes L_{T-1}(\Dt), \\ 
        \Gamma_k(\delta_k) &\coloneqq \frac{\pi^2 \omega_k}{6 N(N-1)} \sum_{i=1}^{N-1} 
        \sigma_\eps''(\delta_{i,k}).
    \end{dcases}
\end{align}
Notice that we expect $\Gamma_k(\delta_k) = \Theta(N^{-1} T^{-1})$.
In the preconditioning procedure, we need to compute $H(\delta)^{-1} z$ in a computationally efficient manner, for $z \in \bbR^{N-1} \otimes \bbR^{T-1}$.
Let us denote the eigenbasis decomposition
\begin{equation*}
        L_{N-1} \coloneqq 4 \sum_{i=1}^{N-1} \sin^2 \left(\frac{i\pi}{2N}\right) v_i^{(N)} \left[v_i^{(N)}\right]^\T.
\end{equation*}
Importantly, we have
\begin{align*}
    (v_i^{(N)})_j \coloneqq \sqrt{\frac{2}{N}} \sin\left(\frac{i j \pi}{N}\right), \hspace{1.5cm} (1 \leq i,j \leq N-1).
\end{align*}
For any $u \in \bbR^N$, this allows us to compute the vector 
\begin{align*}
    \DST(u) \coloneqq (v_i^{(N)} \cdot u)_{i=1}^{N-1} 
    = \left(\frac{2}{\sqrt{N}} \sum_{j=1}^{N-1}  \sin\left(\frac{i j \pi}{N}\right) u_j\right)_{i=1}^{N-1}
\end{align*}
in $\mcO(N \log N)$ operations using classical fast Fourier transform algorithms.

\myskip
Moreover, notice that the matrix of eq.~\eqref{eq:def_H} restricted to a given sine spatial mode is tri-diagonal and positive definite. 
Therefore, when computing $w \coloneqq H(\delta)^{-1} z$ for $z = (z_1, \cdots, z_{N-1}) \in \bbR^{N-1} \otimes \bbR^{T-1}$, 
we need to solve $(N-1)$ independent tri-diagonal inverse problems, and each of these can be solved in $\mcO(T)$ operations using a cyclic reduction algorithm~\cite{gander1998cyclic}.
Combined with the discrete sine transform algorithm for the space modes, this allows us to solve the whole linear system in $\mcO(NT \log N)$ operations.

\subsection{Global scaling and the $\theta$ parameter}\label{subsec:scaling_theta}

While the global scaling parameter $\theta > 0$ in $I_\HCIZ(\theta, \mu, \nu)$ can be formally incorporated in a rescaling of $(\mu, \nu)$, it can be useful to keep it explicitly, in particular to avoid numerical issues in the limits $\theta \ll 1$ and $\theta \gg 1$.
As detailed in Appendix~\ref{subsec_app:rescaling_invariance}, 
one can take this rescaling into account by modifying the variational principle of eqs.~\eqref{eq:def_J_2} and~\eqref{eq:def_G} as:
\begin{equation}
    \label{eq:def_JG_rescaled}
    \begin{dcases}
    J(\theta, \mu,\nu) &\coloneqq \inf_{(\rho, v) \in \adm_\theta(\mu,\nu)} G(\theta, \rho, v), \\
    G(\theta, \rho,v) &\coloneqq \frac 1 2 \int_0^{\theta^{-1}} \rd u \int_\bbR \rd x \, \left(v(u,x)^2 + \frac {\pi^2}{3} \rho(u,x)^2\right) \rho(u,x).
    \end{dcases}
\end{equation}
The definition of admissible trajectories in Definition~\ref{def:adm} remains strictly identical, except that the time interval is now $u \in (0, \theta^{-1})$ (changing the name of the variable).
The minimization algorithm can also be adapted to handle this rescaling: we detail how its different ingredients are affected in Appendix~\ref{subsec_app:rescaling_algorithm}.

\subsection{Initialization, and final algorithm}

\noindent
\textbf{Initialization --}
As mentioned above, we consider a minimization problem which generalizes the simple one-dimensional optimal transport. 
This motivates initializing the algorithm at the solution to this problem, which is simply given by a linear interpolation of the quantiles of $\mu$ and $\nu$, i.e.\ we let
\begin{align}\label{eq:ot_delta}
    \delta_{i,k}^{(0)} &\coloneqq (N+1) \left[\left(1 - t_k\right) (x_{i+1}(0) - x_i(0)) + t_k (x_{i+1}(1) - x_i(1)) \right],
\end{align}
and we recall that $(x_i(0))_{i=1}^N$ and $(x_i(1))_{i=1}^N$ are the quantiles of $\mu_0 = \mu$ and $\mu_1 = \nu$, given in eq.~\eqref{eq:def_xi_01}.

\myskip
\textbf{Preconditioned conjugate gradient --}
In Appendix~\ref{subsec_app:pcg}, we recall the classical preconditioned conjugate gradient (PCG) algorithm.
Given a positive-definite symmetric matrix $A$, a vector $b$, and a positive-definite symmetric matrix $H$ (the \emph{preconditioner}), it returns $\PCG(H ; A, b)$ which is the unique solution to the system $Ax = b$.
In our simulations, we used a relative precision target $\eps_{\rm PCG} = 10^{-3}$ for the PCG algorithm.

\myskip
\textbf{Final algorithm --}
We are now ready to state our minimization algorithm. We use a simple Newton method, using the PCG algorithm to efficiently find a solution to the Newton system. 
We state it in Algorithm~\ref{algo:newton_cg}.
Notice furthermore that the simple form of the Hessian in eq.~\eqref{eq:Hess_G_NDt} allows us to perform the associated matrix-vector multiplications in time $\mcO(NT)$ rather than the naive $\mcO(N^2 T^2)$.
All in all, the complexity of each Newton iteration is dictated to first order by multiplying the number of iterations of the preconditioned conjugate gradient algorithm, and the complexity of inverting a linear system associated to the preconditioner $H(\delta)$, which is $\mcO(N T \log N)$ as detailed in Section~\ref{subsec:preconditioning}.

\begin{algorithm}[ht]
\SetAlgoLined
\KwResult{The value
$J_{N,\Dt,\eps}(\mu_0,\mu_1)
= \inf_{\adm_{N,\Dt}(\mu_0,\mu_1)} G_{N,\Dt,\eps}$,
and the corresponding spacings $\delta_{i,k}$.}

\textbf{Input: } The quantiles $\{x_i(0),x_i(1)\}_{i=1}^N$ of
eq.~\eqref{eq:def_xi_01}, the time spacings $\{\Dt_k\}_{k=1}^T$ of
eq.~\eqref{eq:def_Dtk_sqrt_schedule}, a small $\eps>0$\;

\emph{Initialize} $t=0$ and
$\delta_{i,k}=\delta_{i,k}^{(0)}$ according to
eq.~\eqref{eq:ot_delta}\;

\While{not converged}{
Compute the Newton direction $p^{(t)} =
-\PCG\left(
H(\delta^{(t)});
\nabla^2 G_{N,\Dt,\eps}(\delta^{(t)}),
\nabla G_{N,\Dt,\eps}(\delta^{(t)})
\right)
$\;
{\color{gray} \it 
\begin{itemize}[leftmargin=30pt,itemsep=1pt,topsep=1pt]
    \item $\nabla G_{N,\Dt,\eps}$ and $\nabla^2 G_{N,\Dt,\eps}$ are given in
eq.~\eqref{eq:grad_Hess_G_NDt}.
\item $H(\delta^{(t)})$ is given in eq.~\eqref{eq:def_H}, along with a fast
FFT-based procedure for solving the associated linear systems.
\end{itemize}
}
Find the step size $\alpha > 0$ with backtracking line search~\cite{vandenberghe2004convex}\;
Set $\delta^{(t+1)}=\delta^{(t)}+\alpha p^{(t)}$\;
$t\leftarrow t+1$\;
}

\caption{Newton method with preconditioned conjugate gradient.\label{algo:newton_cg}}
\end{algorithm}

\myskip 
\textbf{Stopping criterion --}
We use as a stopping criterion the value of the Newton decrement $\lambda_{\rm New.}^2 \coloneqq [(\nabla G_{N, \Dt, \eps})^\T (\nabla^2 G_{N, \Dt, \eps})^{-1} \nabla G_{N, \Dt, \eps}]^{-1/2}$, 
which is exactly the decrease in the quadratic approximation obtained by a full Newton step~\cite{vandenberghe2004convex}.
We stop whenever $\lambda_{\rm New.}^2 < \eps_{\rm tol.}$, for a given $\eps_{\rm tol.} > 0$ (we take $\eps_{\rm tol.} = 10^{-6}$ in the results presented in this paper). 

\subsection{Practical implementation and final remarks}

\myskip 
\textbf{Implementation --}
An implementation of Algorithm~\ref{algo:newton_cg} is provided in a public \href{https://github.com/AnMaillard/extensive_rank_HCIZ_solver_public}{Github repository}~\cite{github_repo}: it uses the Pytorch library~\cite{paszke2019pytorch}, and all simulations shown were made on a single Nvidia GPU (either V100, RTX6000 or RTX8000), 
with running times for solving each instance ranging from a few seconds to a few hours for the largest sizes and hardest problems shown.
We used $\eps = 10^{-7}$ generically in all the instances shown, and we checked that the resulting spacings $\delta_{i,k}$ are positive at convergence, so that none of the non-crossing constraints $\delta_{i,k} > 0$ is saturated.

    \myskip
\textbf{Convergence properties and computation time --}
In Appendix~\ref{sec_app:additional_numerics}, we investigate the numerical convergence and total runtime of the algorithm for several pairs of boundary measures and a range of discretization sizes $N$ and $T$.
For fixed $N$, the objective values generally exhibit convergence at a rate $\mcO(T^{-2})$, 
consistent with the trapezoidal discretization used in eq.~\eqref{eq:def_GNDt_eps} for the $\rho^3$ term.
For fixed $T$, we observe an error of order $\mcO(N^{-1})$, which can be further improved by Richardson extrapolation~\cite{richardson1911approximate}.
We note that this $\mcO(1/N)$ accuracy is also commensurate with the discretization of the boundary measures we use: 
in practical applications, the evaluation of the HCIZ integral is often done alongside that of various linear statistics 
of these distributions, whose evaluation from these quantiles generically induces another error $\mcO(1/N)$. 

\myskip 
\textbf{Alternative approaches to the minimization procedure --}
There are many local optimization algorithms for large-scale convex optimization problems. 
We detail here our implementation of the Newton--conjugate gradient algorithm, as we empirically found this approach to outperform first-order methods (even preconditioned), such as Adam~\cite{kingma2014adam}, and approximate second-order methods, such as LBFGS, that are based on a low-rank approximation of the Hessian~\cite{byrd1995limited}. Alternatively to replacing the function $(x+\eps)^{-2}$ for $x \geq 0$ by the convex continuation $\sigma_\eps(x)$ defined for all $x \in \bbR$ and checking \emph{a posteriori} that the minimizer found satisfies the positivity constraint, 
this constraint can be implemented e.g.\ by a penalty function or by methods such as the augmented Lagrangian; we refer to~\cite{nocedal2006numerical} for more details.
We leave a study of a quantitative comparison to alternative procedures (including also different formulations of the problem, as discussed in Section~\ref{subsec:discussion_consequences}) to future work.

\section{Proof of Theorem~\ref{thm:main}}\label{sec:proof}
This section is devoted to the proof of Theorem~\ref{thm:main}.
As a continuous limit of $G_{N, \eps}$ in eq.~\eqref{eq:def_GNeps},
it will be convenient to introduce, for each $\eps > 0$, the quantity
\begin{equation}
\label{e.def.Geps}
G_\eps(\rho,v) \coloneqq \frac 1 2 \int_0^1 \int_\bbR \left(v(t,x)^2 + \frac {\pi^2}{3} \left(\frac{\rho(t,x)}{1+\eps \rho(t,x)}\right)^2\right) \rho(t,x) \, \rd x \, \rd t,
\end{equation}
as well as
\begin{equation}
\label{e.def.Jeps}
J_{\eps}(\mu_0,\mu_1) \coloneqq \inf_{(\rho, v) \in \adm(\mu_0,\mu_1)} G_\eps(\rho, v).
\end{equation}
We then have $G(\rho, v) = G_{\eps = 0}(\rho, v)$ and $J(\mu_0, \mu_1) = J_{\eps = 0}(\mu_0, \mu_1)$.
We first recall the following theorem from~\cite{guionnet2004first}.
\begin{theorem}[Properties of the minimizer of $G$~\cite{guionnet2004first}]
\label{t.exist.min}
For every $\mu_0, \mu_1 \in \mcP_c(\bbR)$ with $\Sigma(\mu_0)$ and $\Sigma(\mu_1)$ finite, the functional $G$ admits a unique minimizer $(\rho^*, v^*)$ over the set $\adm(\mu_0,\mu_1)$. 
Moreover, there exists a compact set $K \subset \bbR$ such that for all $t \in (0, 1)$, the function $\rho^*(t, \cdot)$ is supported on $K$. 
\end{theorem}
\begin{remark}
For $(\rho, v) \in \adm(\mu_0, \mu_1)$, the function $v$ is an element of $L^1((0,1)\times \bbR; \rho(t,x) \, \rd t \, \rd x)$, 
and it is in this sense that the uniqueness of the minimizer stated in Theorem~\ref{t.exist.min} should be understood. 
In other words, any two minimizers $(\rho_1, v_1)$ and $(\rho_2,v_2)$ of $G$ over $\adm(\mu_0,\mu_1)$ must be such that $(\rho_1, \rho_1 v_1) = (\rho_2, \rho_2 v_2)$ on a set of full Lebesgue measure. 
\end{remark}

\myskip
Recall the definition of the discretization scheme in Section~\ref{subsubsec:def_discretization}.
The proof of Theorem~\ref{thm:main} will be decomposed into a number of intermediate results. 
We first prove Proposition~\ref{p.gamma.Gep.G}, which shows continuity with respect to the limit $\eps \to 0$ in the definition of $J_{\eps}(\mu_0, \mu_1)$. 
We then move to Propositions~\ref{p.the.easier.bound} and \ref{p.one.side}, proved respectively in Sections~\ref{subsec:upper_bound} and~\ref{subsec:lower_bound}.
These propositions give upper and lower bounds on $J(\mu_0,\mu_1)$ and, together with Proposition~\ref{p.gamma.Gep.G}, 
directly imply Theorem~\ref{thm:main}. 
We finish with a remark in Section~\ref{subsec:final_remark_proof} on a conjectural alternative formulation of $J(\mu_0, \mu_1)$ that is more directly related to our discretization scheme.

\subsection{Small-regularization limit}\label{subsec:small_eps_limit}

We start by showing that the quantity $J_\eps$, defined in eq.~\eqref{e.def.Jeps}, converges to $J$ as $\eps$ tends to zero\footnote{This could be phrased as the $\Gamma$-convergence of $G_\eps$ to $G$, but we will refrain from using this language here.}.
First, we clearly have that $G_\eps \le G$, so that $J_\eps \leq J$. 
In order to show the converse inequality, we consider approximate minimizers of $J_\eps$, and show that, after passing to a subsequence, we can obtain a candidate minimizer for the functional~$J$ as a limit of these approximate minimizers. 
In essence, the argument is very similar to one that would guarantee the existence of minimizers for $J$, and is based on the convexity of the functional $G$ seen as a function of $(\rho, \rho v)$. 

\myskip
We remark that the existence of a (unique) minimizer for $G_\eps$ for a fixed $\eps  > 0$ would be more subtle to prove, as the functional $G_\eps$ is less coercive than $G$; in particular, approximate minimizers of $G_\eps$ might be such that, for $t \in (0,1)$, the measure $\rho(t,x) \, \rd x$ converges to a probability measure that is not absolutely continuous with respect to the Lebesgue measure. 
This difficulty might be addressed e.g.\ by assuming that $\mu_0$ and $\mu_1$ are absolutely continuous with respect to the Lebesgue measure and using some arguments in the spirit of the theory of optimal transport, or by enlarging the space of minimizer candidates; but for our purposes we will simply avoid this issue throughout. 
\begin{proposition}
\label{p.gamma.Gep.G}
For every $\mu_0, \mu_1 \in \mcP(\bbR)$, we have
\begin{equation}
\label{e.gamma.Gep.G}  
\lim_{\eps \to 0} J_\eps(\mu_0,\mu_1) = J(\mu_0,\mu_1).
\end{equation}
\end{proposition}
\begin{proof}[Proof of Proposition~\ref{p.gamma.Gep.G}]
The existence of the limit on the left-hand side is guaranteed by monotone convergence.
The inequality $\le$ between the left and the right sides of eq.~\eqref{e.gamma.Gep.G} is clear, so we turn to the converse bound. 
Without loss of generality, we can assume that the left side of eq.~\eqref{e.gamma.Gep.G} is finite; we denote it by $\mcI \in [0,+\infty)$. 
For each $\eps \in (0,1]$, we let $(\rho_\eps, v_\eps) \in \adm(\mu_0,\mu_1)$ be such that 
\begin{equation}  
\label{e.this.will.come.back}
G_\eps(\rho_\eps,v_\eps) \le \eps + J_\eps(\mu_0,\mu_1) ,
\end{equation}
and we write $m_\eps \coloneqq \rho_\eps v_\eps$. 
By the continuity equation and \cite[Lemma~8.1.2]{ambrosio2008gradient}, we have for every $\phi \in C^\infty_c(\bbR)$ and $0 < s < t < 1$ that
\begin{equation*}  
\int \phi(x) (\rho_\eps(t,x) - \rho_\eps(s,x)) \,  \rd x = \int_s^t \int_\bbR\partial_x \phi(x) m_\eps(u,x) \, \rd x \, \rd u,
\end{equation*}
and thus, by the Kantorovich duality (e.g.~\cite[Theorem~1.14]{villani2003topics}), we have
\begin{equation*}  
W_1(\rho_\eps(s,x) \rd x, \rho_\eps(t,x) \rd x) \le \int_s^t \int_\bbR |m_\eps(t,x)| \, \rd x \,\rd t,
\end{equation*}
where $W_1$ denotes the $L^1$-Wasserstein distance. We also observe that 
\begin{align*}  
\int_s^t \int_\bbR |m_\eps(t,x)| \, \rd x \, \rd t 
& 
\le \left(\int_s^t \int_\bbR \frac{|m_\eps(t,x)|^2}{\rho_\eps(t,x)} \, \rd x \, \rd t\right)^\frac 1 2\left(\int_s^t \int_\bbR {\rho_\eps(t,x)} \, \rd x \, \rd t\right)^\frac 1 2
\\
& \le \sqrt{(2\mcI+2)|t-s|}. 
\end{align*}
The functions $t \mapsto \rho_\eps(t,x) \rd x$ are thus uniformly equicontinuous with respect to the $W_1$ distance on probability measures. Since $\lim_{t \to 0} \rho_\eps(t,x) \, \rd x = \mu_0$ is fixed, after passing to a subsequence, we can assume that these functions converge to some continuous map $t \mapsto \mu(t,\rd x)$ by the Arzel\`a--Ascoli theorem.
The display above also shows that $|m_\eps(t,x)| \rd x \rd t$ is a non-negative measure with a total mass that is uniformly bounded over $\eps$; 
we can therefore extract a subsequence that converges vaguely to a non-negative measure $\mathsf{m}$ over $[0,1] \times \bbR$. One can directly check that the pair $(\mu(t,\rd x) \, \rd t, \mathsf{m})$ satisfies the continuity equation in the weak sense, in the sense that $\partial_t \mu + \partial_x \mathsf m = 0$. 

\myskip
We now observe that $\mu(t,\rd x) \rd t$ is in fact absolutely continuous with respect to the Lebesgue measure. Indeed, we first observe that for all $M \in (0,+\infty)$ and $\eps > 0$ sufficiently small (in terms of $M$), 
\begin{align*}  
\int_0^1 \int_\bbR \indi_{\{\rho_\eps(t,x) \ge M\}} \rho_\eps(t,x) \, \rd x \, \rd t 
&
\le 4 M^{-2}  \int_0^1 \int_\bbR  \left(\frac{\rho_\eps(t,x)}{1+\eps \rho_\eps(t,x)}\right)^2 \rho_\eps(t,x) \, \rd x \, \rd t
\\
& \le 24 \pi^{-2} M^{-2} \left( \mcI + 1 \right) .
\end{align*}
Hence, for every measurable set $A \subset [0,1] \times \bbR$ and $\eps > 0$ sufficiently small, we have 
\begin{equation}\label{eq:bound_A_rhoeps}
\int_0^1 \int_\bbR \indi_{A}(t,x) \rho_\eps(t,x) \, \rd x \, \rd t \le M |A| +  24 \pi^{-2}M^{-2} (\mcI + 1),
\end{equation}
where $|A|$ denotes the Lebesgue measure of $A$.
Recall that for every bounded Borel set $A$, one can find an open set $U$ containing $A$ such that $|U\setminus A|$ is as small as desired (see e.g.\ \cite[Exercise~A.4 and solution]{dominguez2024statistical}). 
Using this, the estimate above, and the Portmanteau theorem, we obtain that $\mu(t,\rd x) \rd t$ is absolutely continuous with respect to the Lebesgue measure. We can therefore write $\mu(t,\rd x)  \rd t = \rho(t,x)  \rd x \rd t$ for some probability density $\rho \in L^1([0,1]\times \bbR)$. Let us fix $K \in [0,+\infty)$ and denote 
\begin{equation}\label{eq:def_chiK}
\chi_K(x) \coloneqq  (3-K^{-1}) \int_0^x (u \wedge K)^2 \, \rd u.
\end{equation}
The function $\chi_K$ is convex over $[0,+\infty)$, and for $\eps > 0$ sufficiently small (in terms of $K$), we have $\chi_K(x) \le x^3/(1+\eps x)^2$ for all $x \ge 0$. 
The weak convergence of $\rho_\eps$ to $\rho$ and the convexity of $\chi_K$ ensure that 
\begin{equation*}  
\int_0^1\int_\bbR \chi_K(\rho(t,x)) \, \rd x \, \rd t \le \liminf_{\eps \to 0} \int_0^1\int_\bbR \chi_K(\rho_\eps(t,x)) \, \rd x \, \rd t,
\end{equation*}
and the comparison between $\chi_K$ and $x \mapsto x^3/(1+\eps x)^2$ for small $\eps > 0$ yields that
\begin{equation*}  
\int_0^1\int_\bbR \chi_K(\rho(t,x)) \, \rd x \, \rd t \le \liminf_{\eps \to 0} \int_0^1\int_\bbR \left(\frac{\rho_\eps(t,x)}{1+\eps \rho_\eps(t,x)}\right)^2 \rho_\eps(t,x) \, \rd x \, \rd t.
\end{equation*}
By the monotone convergence theorem, we can then send $K$ to infinity on the left side to obtain that
\begin{equation}  
\label{e.cubic.bound}
\int_0^1\int_\bbR \rho(t,x)^3 \, \rd x \, \rd t \le \liminf_{\eps \to 0} \int_0^1\int_\bbR \left(\frac{\rho_\eps(t,x)}{1+\eps \rho_\eps(t,x)}\right)^2 \rho_\eps(t,x) \, \rd x \, \rd t.
\end{equation}
We next notice that, for every $x,y \in \bbR$, 
\begin{equation}  
\label{e.convex.x2y}
\sup_{t \in \bbR} \left( 2t x - t^2 y \right)  = 
\begin{cases}  
\frac{x^2}{y} & \text{ if } y > 0, \\
0 & \text{ if } x = y = 0, \\
+\infty & \text{ otherwise }.
\end{cases}
\end{equation}
Hence, for every $\phi \in C^\infty_c([0,1]; \bbR)$, we have 
\begin{equation}\label{eq:bound_I_1}
\int_0^1\int_\bbR (2 \phi m_\eps - \phi^2 \rho_\eps)(t,x) \, \rd x \, \rd t \le  \int_0^1 \int_\bbR \frac{|m_\eps(t,x)|^2}{\rho_\eps(t,x)} \, \rd x \, \rd t \le 2 \mcI + 2\eps.
\end{equation}
Passing to the limit, we obtain that 
\begin{equation}\label{eq:bound_I_2}
\int 2 \phi(t,x) \rd \mathsf{m}(t,x) - \int_0^1 \int_\bbR \phi(t,x)^2 \rho(t,x) \, \rd x \, \rd t \le 2 \mcI .
\end{equation}
Replacing $\phi$ by $\lambda \phi$ for arbitrary $\lambda \in \bbR$ and calculating the discriminant of the resulting quadratic function of $\lambda$, we obtain 
\begin{equation*}  
\left(\int \phi(t,x) \rd \mathsf{m}(t,x)\right)^2 \le 2 \mcI \int_0^1 \int_\bbR \phi(t,x)^2 \rho(t,x) \,\rd x \,  \rd t .
\end{equation*}
By the Riesz representation theorem, we deduce that there exists $v \in L^2([0,1]\times \bbR; \rho(t,x) \rd x \rd t)$ such that $\rd \mathsf m(t,x) = v(t,x) \rho(t,x) \rd x \rd t$. Recalling the first inequality in \eqref{eq:bound_I_1}, and using also~\eqref{e.cubic.bound}, we thus obtain that for every $\phi \in C^\infty_c([0,1];\bbR)$,
\begin{equation}\label{eq:bound_I}
\frac 1 2\int_0^1 \int_\bbR (2\phi v \rho - \phi^2 \rho + \frac{\pi^2}{3}\rho^3)(t,x) \, \rd x \, \rd t \le \mcI.
\end{equation}
Optimizing over $\phi$ yields that 
\begin{equation*}  
\frac 1 2 \int_0^1 \int_\bbR \left( v(t,x)^2 + \frac{\pi^2}{3}\rho(t,x)^2 \right) \rho(t,x) \, \rd x \, \rd t \le \mcI,
\end{equation*}
and we recall that the pair $(\rho, v)$ satisfies the continuity equation \eqref{eq:continuity}. We have thus upper-bounded the right-hand side of \eqref{e.gamma.Gep.G} by the left-hand side of this display, as desired. 
\end{proof}

\subsection{Upper bound}\label{subsec:upper_bound}

Our next goal is to upper bound
$J_\eps(\mu_0, \mu_1)$ by $\liminf_{N \to \infty} J_{N,\eps}(\mu_0, \mu_1)$. 
We will show that from each $\{(x_i, v_i)\}_{i=1}^N \in \adm_N(\mu_0,\mu_1)$,
we can build a continuous pair $(\rho_N,v_N)$ such that $G_\eps(\rho_N,v_N) = G_{N,\eps}(\{(x_i, v_i)\})$. 
A difficulty will arise from the fact that $\rho_N(0,x) \, \rd x$ is almost but not exactly equal to $\mu_0$. As discussed in the paragraph preceding Proposition~\ref{p.gamma.Gep.G}, we prefer to avoid having to construct minimizers for $G_\eps$, so instead of simply letting $N$ tend to infinity to construct some $(\rho,v) \in \adm(\mu_0,\mu_1)$ with $G_\eps(\rho,v) \le J_{N,\eps}(\mu_0,\mu_1)$, we also let $\eps \to 0$ at the same time and obtain a candidate for $G$ in place of $G_\eps$. 
We thus do not really show that $J_\eps \le \liminf_N J_{N,\eps}$, but only the following weaker statement, which is sufficient for our purposes.

\begin{proposition}
\label{p.the.easier.bound}
For every $\mu_0, \mu_1 \in \mcP(\bbR)$  with $\Sigma(\mu_0), \Sigma(\mu_1)$ finite, we have
\begin{equation}  
\label{e.the.easier.bound}
J(\mu_0,\mu_1) \le \lim_{\eps \to 0} \liminf_{N \to +\infty} J_{N,\eps}(\mu_0,\mu_1).
\end{equation}
\end{proposition}

\begin{proof}[Proof of Proposition~\ref{p.the.easier.bound}]
Let $\{(x_i, v_i)\} \in \adm_N(\mu_0,\mu_1)$. 
Without loss of generality, we assume that $x_i(t) < x_{i+1}(t)$ for every $i \in \{1,\ldots, N-1\}$ and $t \in [0,1]$ (indeed, if this were not the case, we could modify the $v$'s ever so slightly to enforce this constraint while changing $G_{N,\eps}(x,v)$ infinitesimally).
For every $t \in [0,1]$ and $i \in \{1,\ldots, N\}$, we set
\begin{equation*}  
f(t,x_i(t)) \coloneqq \frac{i-1}{N-1} \quad \text{ and } \quad v(t,x_i(t)) \coloneqq v_i(t).
\end{equation*}
We extend $f(t,\cdot)$ and $v(t,\cdot)$ to continuous functions that are affine on the intervals $[x_{i}(t), x_{i+1}(t)]$ ($i \in \{1,\ldots, N-1\}$), with $f(t,x) = 0$ for $x \le x_1(t)$ and $f(t,x) = 1$ for $x \ge x_N(t)$; for definiteness, we can also extend $v(t,\cdot)$ so that it is constant on $(-\infty,x_1(t)]$ and on $[x_N(t), +\infty)$, but this will not play any role. Explicitly, for every $i \in \{1,\ldots, N-1\}$ and $y \in [0, x_{i+1}(t) - x_{i}(t)]$, we have
\begin{equation*}  
f(t,x_i(t) + y) = \frac{i-1}{N-1} + \frac y {(N-1)(x_{i+1}(t) - x_i(t))}.
\end{equation*}
Hence, for every $y \in (0, x_{i+1}(t) - x_{i}(t))$, we have
\begin{equation}  
\label{e.partial_xf}
\partial_x f(t,x_i(t) + y) =  \frac 1 {(N-1)(x_{i+1}(t) - x_i(t))}
\end{equation}
and 
\begin{equation*}  
(\partial_t f + v_i(t) \partial_x f)(t,x_i(t) + y) = - \frac{y(v_{i+1}(t) - v_i(t))}{(N-1) (x_{i+1}(t) - x_i(t))^2}.
\end{equation*}
Combining the two previous displays yields that 
\begin{equation*}  
\partial_t f(t, x_i(t) + y) + \left( v_i(t) + \frac{y(v_{i+1}(t) - v_i(t))}{x_{i+1}(t) - x_i(t)} \right) \partial_x f(t,x_i(t) + y) = 0, 
\end{equation*}
and the expression between the large parentheses is $v(t,x_i(t) + y)$. We have thus shown that the continuity equation
\begin{equation}  
\label{e.equation.for.f}
\partial_t f + v \partial_x f = 0
\end{equation}
is satisfied at every point of differentiability of $f$. Since the function $f$ is Lipschitz continuous and vanishes on $\{(t,x) : x \le x_1(t)\}$,
we have for every $t \in [0,1]$ and $x \in \bbR$ that
\begin{equation*}  
f(t,x) = f(0,x) + \int_0^t \partial_t f(s,x) \, \rd s  = \int_{-\infty}^x \partial_x f(t,y) \, \rd y,
\end{equation*}
where the derivatives $\partial_t f$ and $\partial_x f$ are well-defined almost everywhere, and thus for every test function $\phi \in C^\infty_c((0,1)\times \bbR)$, we have
\begin{equation*}  
\int_0^1 \int_\bbR \partial_x \phi \, \partial_t f = -\int_0^1 \int_\bbR \partial_t \partial_x  \phi \, f = \int_0^1 \int_\bbR \partial_t \phi \, \partial_x f .
\end{equation*}
Setting $\rho(t,x) \coloneqq \partial_x f(t,x)$ (see eq.~\eqref{e.partial_xf}), we deduce from \eqref{e.equation.for.f} and the previous display that the pair $(\rho, v)$ satisfies the continuity equation \eqref{eq:continuity} in the weak sense.
For every $i \in \{1,\ldots, N-1\}$ and $t \in [0,1]$, we use the fact that $\rho(t,y)$ is constant over $(x_i(t),x_{i+1}(t))$ and integrates to $1/(N-1)$ there, while $v(t,y)$ is affine on this interval, to obtain that 
\begin{equation*}  
\int_{x_i(t)}^{x_{i+1}(t)} v(t,y)^2\rho(t,y) \, \rd y \le \frac{1}{2(N-1)} (v_{i}(t)^2 + v_{i+1}(t)^2). 
\end{equation*}
Recalling the expression for $\rho$ in \eqref{e.partial_xf}, we also have that 
\begin{equation*}  
\int_{x_i(t)}^{x_{i+1}(t)}  \left( \frac{\rho(t,y)}{1+\eps \rho(t,y)} \right) ^2 \rho(t,y)\, \rd y = \frac{1}{(N-1)^3}  \left( x_{i+1}(t) - x_i(t) + \frac{\eps}{N-1} \right)^{-2} .
\end{equation*}
Combining these estimates yields that 
\begin{multline*}  
\frac 1 2 \int_0^1 \int_\bbR \left(v(t,y)^2 + \frac{\pi^2}{3} \left( \frac{\rho(t,y)}{1+\eps \rho(t,y)} \right) ^2 \right)\rho(t,y)\, \rd y \, \rd t 
\\
\le \frac{1}{2(N-1)} \sum_{i = 1}^{N-1} \int_0^1\left( \frac{v_i^2(t) + v_{i+1}^2(t)}{2} + \frac{\pi^2}{3(N-1)^2} \left( x_{i+1}(t) - x_i(t) + \frac{\eps}{N-1} \right)^{-2}  \right) \rd t,
\end{multline*}
and in particular,
\begin{equation}  
\label{e.comp.Gep.GNep}
G_\eps(\rho,v) \le \left( \frac{N+1}{N-1} \right) ^3 G_{N,\eps}(\{(x_i, v_i)\}).
\end{equation}
At this stage we have described a procedure that, for each $N$, $\eps$ and $\{(x_i, v_i)\} \in \adm_N(\mu_0,\mu_1)$, constructs some $(\rho,v)$ such that \eqref{e.comp.Gep.GNep} holds. As mentioned just before the statement of Proposition~\ref{p.the.easier.bound}, we now need to face the problem that $\rho(0,x) \rd x$ is not exactly equal to $\mu_0$. 
To do so, we pick subsequences $N_k \to +\infty$ and $\eps_k \to 0$ as well as $\{(x^{(k)}_{i},v^{(k)}_{i})\} \in \adm_{N_k}(\mu_0,\mu_1)$ in such a way that 
\begin{equation*}  
\lim_{k \to +\infty} G_{N_k,\eps_k}(\{(x^{(k)}_{i},v^{(k)}_{i})\}) = \lim_{\eps \to 0} \liminf_{N \to +\infty} J_{N,\eps}(\mu_0,\mu_1),
\end{equation*}
which without loss of generality we can assume to be finite.
We next construct $(\rho^{(k)}, v^{(k)})$ as above so that 
\begin{equation*}  
G_{\eps_k}(\rho^{(k)},v^{(k)}) \le \left( \frac{N_k+1}{N_k-1} \right) ^3 G_{N_k,\eps_k}(\{(x^{(k)}_{i},v^{(k)}_{i})\}),
\end{equation*}
and in particular, possibly after extracting a further subsequence to guarantee the existence of the limit in $k$,
\begin{equation}  
\label{e.almost.there}
\lim_{k \to +\infty} G_{\eps_k}(\rho^{(k)},v^{(k)}) \le \lim_{\eps \to 0} \liminf_{N \to +\infty} J_{N,\eps}(\mu_0,\mu_1).
\end{equation}
The key point now is that $\rho^{(k)}(0,x) \rd x$ converges to $\mu_0$ as $k$ tends to infinity. Indeed, writing
\begin{equation*}  
f_k(x) \coloneqq \int_{-\infty}^x \rho^{(k)}(0,u) \rd u \quad \text{ and } \quad f(x) \coloneqq \int_{-\infty}^x \rd \mu_0,
\end{equation*}
we recall that, for every $i \in \{1,\ldots, N_k\}$,
\begin{equation*}  
f_k(x^{(k)}_i(0)) = \frac{i-1}{N_k-1} \quad \text{ and } \quad f(x^{(k)}_i(0)) = \frac{i}{N_k+1},
\end{equation*}
This and the monotonicity of $f_k$ and $f$ yield that 
\begin{equation*}  
\lim_{k \to \infty} \|f_k - f\|_{L^\infty(\bbR)} = 0,
\end{equation*}
which implies the announced convergence in law of $\rho^{(k)}(0,x) \rd x$ to $\mu_0$. 
Notice the strong similarity between eqs.~\eqref{e.almost.there} and~\eqref{e.this.will.come.back}: in both cases, we have a sequence of candidates $(\rho_\eps, v_\eps)$ for $G_\eps$ (possibly along a subsequence of $\eps$'s), with an upper bound on the limit of $G_\eps(\rho_\eps,v_\eps)$, and we aim to construct a candidate $(\rho,v)$ such that $G(\rho,v)$ is smaller than or equal to this limit. The only small difference between our present situation and the one encountered in the proof of Proposition~\ref{p.gamma.Gep.G} comes from the fact that now $\rho^{(k)}(0,x) \rd x$ is not equal to $\mu_0$, but only converges to it as $k$ tends to infinity. This entails essentially no modification to the argument there, so we can conclude that there exists a candidate $(\rho, v) \in \adm(\mu_0, \mu_1)$ such that 
\begin{equation*}  
G(\rho,v) \le \lim_{\eps \to 0} \liminf_{N \to +\infty} J_{N,\eps}(\mu_0,\mu_1),
\end{equation*}
as desired. 
\end{proof}

\subsection{Lower bound}\label{subsec:lower_bound}
We now turn to the converse bound.
\begin{proposition}
\label{p.one.side}
Let $\mu_0 = \rho(0,x) \rd x$, $\mu_1 = \rho(1,x) \rd x \in \mcP_c(\bbR) \cap \mcP_\abc(\bbR)$ with $\rho(0,\cdot), \, \rho(1,\cdot) \in L^3(\bbR)$. For every $\eps > 0$, we have
\begin{equation}  
\label{e.one.side}
\limsup_{N \to \infty}J_{N,\eps}(\mu_0,\mu_1) \le  J(\mu_0,\mu_1).
\end{equation}
\end{proposition}

\subsubsection{Proof strategy and preliminary results}

Ideally, we would aim, for each $(\rho,v) \in \adm(\mu_0,\mu_1)$, to construct some paths $\{(x_i, v_i)\} \in \adm_N(\mu_0,\mu_1)$ such that $G_{N,\eps}(\{(x_i, v_i)\})$ 
is close to $G_\eps(\rho,v)$. At a heuristic level, what one needs to do is clear: for each $i \in \{1,\ldots, N\}$, we simply fix $x_i(0)$ as in eq.~\eqref{eq:def_xi_01}, and let $\partial_t x_i(t) = v(t,x_i(t))$. 

\myskip
The main problem is that a control of the error in this argument would require some estimate on the Lipschitz regularity of $v$ in the space variable. 
Denoting by $(\rho^*,v^*)$ the minimizer for $G$, it is not in general the case that $v^*$ is smooth (and the situation for $G_\eps$ can only be worse). Some regularity results on $(\rho^*,v^*)$ can be found in \cite{guionnet2004first}, and it is shown in particular that $\rho^*$ and $v^*$ are smooth on the set $\{(t,x) \in (0,1) \times \bbR \ : \ \rho^*(t,x) > 0\}$. However $\partial_x v^*$ may indeed blow up near the boundary of this set ---
this is clearly necessary if the number of connected components of $\{x : \rho^*(t,x) > 0\}$ changes with time, as was already discussed in the paragraph below \eqref{eq:sym_MP}. 
We will therefore need to employ a regularization argument, i.e.\ to start by showing that one can find some sufficiently smooth pair $(\rho, v)$ such that $G(\rho,v)$ is close to optimal. This smoothing argument is facilitated by the fact that we know in advance from \cite{guionnet2004first} that $\rho^*$ is compactly supported (under the assumption that $\mu_0$ and $\mu_1$ are), so that we can first lift it away from zero on some sufficiently large compact set, and then regularize $\rho$ and $\rho v$ via a convolution procedure. While one might a priori expect to be able to show that $\limsup_{N \to +\infty} J_{N,\eps}(\mu_0,\mu_1) \le J_\eps(\mu_0,\mu_1)$ for each fixed $\eps$, we again only obtain a result in which the limit $\eps \to 0$ is involved, because doing so allows us to exploit this known property of $\rho^*$. 

\myskip
Before going into the smoothing argument, we first describe how to control the error between the continuous $G_{\eps}(\rho,v)$ and suitable discrete approximations when the vector field $v$ is sufficiently regular. 
We denote by $\reg$ the space of bounded continuous functions $v = v(t,x): \bbR_+ \times \bbR \to \bbR$ that are differentiable in $x$ and with $\partial_x v$ bounded and uniformly continuous. 
Fixing $v \in \reg$, we denote by $(\phi_t(x))_{t \ge 0, x \in \bbR}$ the flow associated with the differential equation
\begin{equation}  
\label{e.def.phitx}
\begin{cases}  
\partial_t \phi_t(x) = v(t,\phi_t(x)),\\
\phi_0(x) = x. 
\end{cases}
\end{equation}
As is well-known, for each $t \ge 0$, the mapping $\phi_t$ is a $C^1$-diffeomorphism; we denote its inverse by $\phi_t^{-1}$. We have
\begin{equation}  
\label{e.partial.phit}
\partial_x \phi_t(x) = \exp \left( \int_0^t \partial_x v(s,\phi_s(x)) \, \rd s\right) ,
\end{equation}
and in particular, the first derivatives of $\phi_t$ and $\phi_t^{-1}$ are bounded over $[0,T]\times \bbR$ for every $T < +\infty$. 

\myskip
For such a regular velocity field $v$ and a function $\rho$ satisfying the continuity equation~\eqref{eq:continuity}, we can represent $\rho(t,x) \, \rd x$ as the image of $\mu_0$ through the mapping~$\phi_t$ (see e.g.\ \cite[Lemma~8.1.6]{ambrosio2008gradient}). That is, assuming that $\mu_0 = \rho(0,x) \, \rd x$, we have for every $f \in C^\infty_c(\bbR)$ that
\begin{equation*}  
\int f(x) \rho(t,x) \, \rd x = \int f(\phi_t(x)) \rho(0,x) \, \rd x.
\end{equation*}
From this, we deduce that 
\begin{equation}  
\label{e.rhot.explicit}
\rho(t,x) = \partial_x \phi_t^{-1} (x) \, \rho(0,\phi_t^{-1}(x)) .
\end{equation}
These observations will be useful in relation with the following result.
\begin{proposition}
\label{p.discretize}
Let $v \in \reg$ and $\rho$ be such that $(\rho,v) \in \adm(\mu_0,\mu_1)$, with $\mu_0 = \rho(0,x) \, \rd x \in \mcP_c(\bbR)$ and $\rho(0,\cdot)$ a continuous function. For every $t \in [0,1]$, we define $(x_i(t))_{1 \le i \le N}$ such that for every $i \in \{1,\ldots, N\}$,
\begin{equation*}  
\int_{-\infty}^{x_{i}(t)} \rho(t,\cdot) = \frac i {N+1}.
\end{equation*}
For every $i \in \{1,\ldots, N\}$, the function $t \mapsto x_i(t)$ is differentiable, with 
\begin{equation}  
\label{e.evol.x}
\partial_t x_i(t) = v(t,x_i(t)),
\end{equation}
and setting $v_i(t) \coloneqq v(t,x_i(t))$, we have
\begin{equation*}  
\left|G_{N,\eps}(\{(x_i, v_i)\}) - G_\eps(\rho,v)\right| \le \frac{5}{\sqrt{N}} \left(\|v\|_{L^\infty}^2 + \eps^{-2} \right) + (1+4\eps^{-1}) \omega(2MN^{-\frac 1 2}),
\end{equation*}
where $M \in \bbR$ is any number such that 
\begin{equation}  
\label{e.supp.M.bound}
\bigcup_{t \in [0,1]} \supp \rho(t,\cdot) \subseteq [-M,M],
\end{equation}
and $\omega$ is any function such that for every $t \in [0,1]$ and $x,y \in \bbR$,
\begin{equation}  
\label{e.modulus.v.rho} 
|v(t,y) - v(t,x)| + |\rho(t,x) - \rho(t,y)| \le \omega(|y-x|).
\end{equation}
\end{proposition}
\begin{remark}  
For $M_0$ such that $\supp \rho(0,\cdot) \subseteq [-M_0,M_0]$, choosing
\begin{equation*}  
M \coloneqq M_0 + \|\partial_x v\|_{L^{\infty}([0,1]\times \bbR)}
\end{equation*}
ensures that \eqref{e.supp.M.bound} holds. Similarly, using  \eqref{e.partial.phit} and \eqref{e.rhot.explicit}, we can choose~$\omega$ in terms of $\|\partial_x v\|_{L^\infty([0,1]\times \bbR)}$ and a modulus of continuity of $\partial_x v$ and $\rho(0, \cdot)$ only to ensure that~\eqref{e.modulus.v.rho} holds, while making sure that $\lim_{r \to 0} \omega (r) = 0$.
\end{remark}
\begin{proof}[Proof of Proposition~\ref{p.discretize}]
Writing 
\begin{equation*}  
f(t,x) \coloneqq \int_{-\infty}^x \rho(t,u) \rd u,
\end{equation*}
we see from the continuity equation \eqref{eq:continuity} that
\begin{equation*}  
\partial_t f + v \partial_x f = 0. 
\end{equation*}
In particular, for each $x \in \bbR$, the function $t \mapsto f(t,\phi_t(x))$ is constant, where $(\phi_t(x))$ is as in~\eqref{e.def.phitx}. 
Using this observation with $x$ replaced by $x_i(0)$ yields~\eqref{e.evol.x}. 
By the intermediate-value theorem, for each $t \in [0,1]$ and $i \in \{1,\ldots, N-1\}$, there exists $y_i(t) \in [x_i(t), x_{i+1}(t)]$ such that 
\begin{equation*}  
\frac 1 {N+1} = f(t,x_{i+1}(t)) - f(t,x_i(t)) = \rho(t,y_i(t)) (x_{i+1}(t) - x_i(t)).
\end{equation*}
We can thus write
\begin{equation*}  
G_{N,\eps}(\{(x_i, v_i)\}) = \frac 1 {2N} \sum_{i = 1}^N \int_0^1 v(t,x_i(t))^2  \rd t + \frac {\pi^2} {6N} \sum_{i = 1}^{N-1} \int_0^1 \left( \frac{ \rho(t,y_i(t))}{1 + \eps \rho (t,y_i(t))} \right)^2  \, \rd t.
\end{equation*}
For each $t \in [0,1]$, we clearly have that $x_i(t) \in [-M,M]$, so there must be fewer than $\sqrt{N}$ indices $i \in \{1,\ldots, N-1\}$ with $x_{i+1}(t) - x_i(t) \ge 2MN^{-\frac 1 2}$. Denoting by $\mcI(t)$ the complementary set of indices in $\{1,\ldots, N-1\}$, we thus have that
\begin{align*}
&\left| G_{N,\eps}(\{(x_i, v_i)\}) - \frac 1 {2N} \int_0^1  \sum_{i \in \mcI(t)}\left(v(t,x_i(t))^2   +  \frac{\pi^2}{3} \left( \frac{ \rho(t,y_i(t))}{1 + \eps \rho (t,y_i(t))} \right)^2 \right) \, \rd t\right| \\
&\le \frac{1}{2\sqrt{N}} \left(\|v\|_{L^\infty}^2 + \frac{\pi^2}{3}\eps^{-2} \right).
\end{align*}
Recalling that $\int_{x_i(t)}^{x_{i+1}(t)} \rho(t,\cdot) = (N+1)^{-1}$, we can write for $i \in \mcI(t)$ that
\begin{equation*}  
\left| v(t,x_i(t)) - (N+1)\int_{x_i(t)}^{x_{i+1}(t)} v(t,x) \rho(t,x) \, \rd x\right| \le \omega(2MN^{-\frac 1 2}),
\end{equation*}
and likewise,
\begin{equation*}  
\left| \left( \frac{ \rho(t,y_i(t))}{1 + \eps \rho (t,y_i(t))} \right)^2 - (N+1)\int_{x_i(t)}^{x_{i+1}(t)} \left( \frac{ \rho(t,x)}{1 + \eps \rho (t,x)} \right)^2 \rho(t,x) \, \rd x\right| 
\le 2\eps^{-1} \omega(2MN^{-\frac 1 2}).
\end{equation*}
Up to these error terms, we can therefore replace the quantity
\begin{equation*}  
\frac 1 {2N}  \int_0^1 \sum_{i \in \mcI(t)}\left(v(t,x_i(t))^2   + \frac{\pi^2}{3} \left( \frac{ \rho(t,y_i(t))}{1 + \eps \rho (t,y_i(t))} \right)^2 \right) \, \rd t
\end{equation*}
by
\begin{equation*}  
\frac {N+1} {2N} \int_0^1 \sum_{i \in \mcI(t)} \int_{x_i(t)}^{x_{i+1}(t)} \left(v(t,x)^2   +  \frac{\pi^2}{3}\left( \frac{ \rho(t,x)}{1 + \eps \rho (t,x)} \right)^2 \right) \rho(t,x)\, \rd x \, \rd t.
\end{equation*}
The difference between this quantity and $G_{\eps}(\rho,v)$ is at most
\begin{equation*}  
\left(\frac{1}{2N} + \frac{1}{2\sqrt{N}}\right) \left(\|v\|_{L^\infty}^2 + \frac{\pi^2}{3}\eps^{-2} \right),
\end{equation*}
so the proof is complete.
\end{proof}

\myskip
As discussed earlier, the proof of Proposition~\ref{p.one.side} requires a smoothing operation on $(\rho,v)$. The smoothing causes a discrepancy on the initial and final measures $\rho(0,x) \, \rd x$ and $\rho(1,x) \, \rd x$. To avoid this, we will squeeze the convolved objects onto the time interval $[\eta,1-\eta]$ for some small $\eta > 0$, and use the next lemma to bridge time $0$ to time $\eta$ and time $1-\eta$ to time $1$.
\begin{lemma}
\label{l.bridging}
Let $\eta \in (0,1]$, and let $\mu_0 = \rho(0,x) \, \rd x, \mu_\eta = \rho(\eta,x) \, \rd x \in \mcP_c(\bbR)$ with $\rho(0,\cdot), \rho(\eta,\cdot) \in L^3(\bbR)$. 
We let $(x_i(0))_{ 1 \le i \le N}$ and $(x_i(\eta))_{1 \le i \le N}$ be defined as in \eqref{eq:def_xi_01}, except with time $1$ replaced by time $\eta$. For each $i \in \{1,\ldots, N\}$, we denote by $(x_i(t))_{t \in [0,\eta]}$ the linear interpolation from $x_i(0)$ to $x_i(\eta)$, and we set $v_i(t) \coloneqq \eta^{-1} (x_i(\eta) - x_i(0))$. We have
\begin{multline*}
\limsup_{N \to +\infty} \frac {1}{2N} \sum_{i = 1}^N \int_0^\eta \left( v_i(t)^2 + \frac{\pi^2}{3(N+1)^2} \left( x_{i+1}(t) - x_i(t)  \right) ^{-2}  \right) \rd t 
\\
\le \frac 1 {2\eta} W_2(\mu_0, \mu_\eta)^2 + \frac{\pi^2\,\eta}{12}\int(\rho(0,x)^3 + \rho(\eta,x)^3)\, \rd x,
\end{multline*}
where $W_2(\mu_0, \mu_\eta)$ denotes the $L^2$-Wasserstein distance between $\mu_0$ and $\mu_\eta$. 
\end{lemma}
\begin{proof}[Proof of Lemma~\ref{l.bridging}]
Since $v_i$ is constant on $[0,\eta]$, we can write
\begin{equation}
\label{e.estim.vpart}
\frac{1}{2N}\sum_{i=1}^N\int_0^\eta v_i(t)^2\,\rd t
=\frac{\eta}{2N}\sum_{i=1}^N \left(\frac{x_i(\eta)-x_i(0)}{\eta}\right)^2
=\frac{1}{2\eta}\,\frac1N\sum_{i=1}^N \left(x_i(\eta)-x_i(0)\right)^2.
\end{equation}
By construction $x_i(0)=F_0^{-1}\!\left(\frac{i}{N+1}\right)$ and $x_i(\eta)=F_\eta^{-1}\!\left(\frac{i}{N+1}\right)$, where $F_0^{-1}$ and $F_\eta^{-1}$ denote the inverse cumulative distribution functions of $\mu_0$ and $\mu_\eta$ respectively. 
Since $\mu_0$ and $\mu_\eta$ have compact support, these functions are bounded and, since they are monotone, they are of bounded variation. As a result, the function
\begin{equation*}  
u\mapsto\Big(F_\eta^{-1}(u)-F_0^{-1}(u)\Big)^2
\end{equation*}
is also of bounded variation. It is therefore Riemann-integrable (see e.g.\ \cite[Corollary 14.3]{carothers}), and thus
\begin{equation}  
\label{e.kinetic.W2}
\lim_{N \to \infty} \frac 1 N \sum_{i=1}^N \left(x_i(\eta)-x_i(0)\right)^2  = \int_0^1 \Big(F_\eta^{-1}(u)-F_0^{-1}(u)\Big)^2 \, \rd u.
\end{equation}
The right-hand side of this expression is $W_2(\mu_0, \mu_\eta)^2$ (see e.g.\ \cite[Chapter~2]{villani2003topics}).

\myskip
For the second term in the integral to be estimated, we first observe that, since the mapping $t \mapsto x_{i+1}(t) - x_i(t)$ is affine and $z \mapsto z^{-2}$ is convex on $(0,+\infty)$, we have by Jensen's inequality that 
\begin{equation*}  
(x_{i+1}(t) - x_i(t))^{-2} \le \left( 1 - \frac{t}{\eta} \right) (x_{i+1}(0) - x_i(0))^{-2} + \frac{t}{\eta} (x_{i+1}(\eta) - x_i(\eta))^{-2}.
\end{equation*}
Integrating in time, we thus obtain that
\begin{equation}  
\label{e.rho3.jensen}
\int_0^{\eta}(x_{i+1}(t) - x_i(t))^{-2} \, \rd t \le \frac{\eta}{2} (x_{i+1}(0) - x_i(0))^{-2} + \frac{\eta}{2} (x_{i+1}(\eta) - x_i(\eta))^{-2}.
\end{equation}
We now proceed to estimate $\sum_{i} (x_{i+1}(0) - x_i(0))^{-2}$; the second term on the right side of the previous display can be estimated in the same way. By H\"older's inequality, we have
\[
\left( \frac 1 {N+1} \right) ^3
= \left(\int_{x_i(0)}^{x_{i+1}(0)} \rho(0,x) \,\rd x\right)^3
\le \big(x_{i+1}(0)-x_i(0)\big)^2 \int_{x_i(0)}^{x_{i+1}(0)} \rho(0,x)^3 \,\rd x,
\]
that is,
\begin{equation*}
\frac{1}{(N+1)^2}\,\big(x_{i+1}(0)-x_i(0)\big)^{-2}
\le (N+1)\int_{x_i(0)}^{x_{i+1}(0)} \rho(0,x)^3 \,\rd x.
\end{equation*}
Summing this inequality over $i$ and dividing by $N$ gives
\[
\frac{1}{N(N+1)^2}\sum_{i=1}^N \big(x_{i+1}(0)-x_i(0)\big)^{-2}
\le \frac{N+1}{N}\int_\bbR \rho(0,x)^3 \,\rd x.
\]
Combining this with \eqref{e.rho3.jensen} yields that
\[
\frac{1}{2N}\sum_{i=1}^N \int_0^\eta \frac{\pi^2}{3(N+1)^2}\big(x_{i+1}(t)-x_i(t)\big)^{-2}\,\rd t
\le \frac{\pi^2\eta}{12}\,\frac{N+1}{N}\int(\rho(0,x)^3 + \rho(\eta,x)^3)\, \rd x.
\]
Taking the limsup as $N \to \infty$, and recalling eq.~\eqref{e.kinetic.W2}, we obtain the announced result.
\end{proof}

\subsubsection{Proof of Proposition~\ref{p.one.side}}

By Theorem~\ref{t.exist.min}, there exists $(\rho,v) \in \adm(\mu_0,\mu_1)$ with $\rho$ compactly supported such that $G(\rho,v) = J(\mu_0,\mu_1)$. 
We let $M \ge 1$ be such that for all $t \in [0,1]$, the support of $\rho(t,\cdot)$ 
is contained in $[-(M-1),M-1]$. We extend $\rho$ and~$v$ to all times $t \in \bbR$ by setting $v(t,x) = 0$ for $t \notin [-1,1]$ and 
\begin{equation*}  
\rho(t,x) = 
\begin{cases}  
\rho(0,x) & \text{ if } t < 0,\\
\rho(1,x) & \text{ if } t > 1.
\end{cases}
\end{equation*}
The extended $(\rho,v)$ still satisfies the continuity equation~\eqref{eq:continuity} in the weak sense, as does the pair $(\indi_{\left[-M,M\right]}(x)/(2M), 0)$. We write $m \coloneqq \rho v$, fix $\delta_1 \in (0,\frac 1 4]$, and set
\begin{equation*}  
\rho_{\delta_1}(t,x) \coloneqq (1-\delta_1) \rho(t,x) + \frac{\delta_1}{2M} \indi_{\left[-M,M\right]}(x) \qquad \text{ and } \qquad m_{\delta_1} \coloneqq (1-\delta_1) m.
\end{equation*}
Using the convexity of the mapping in eq.~\eqref{e.convex.x2y}, as well as that of $r \mapsto r^3$ over $[0,+\infty)$, we have
\begin{equation*}  
\frac 1 2 \int_0^1 \int_\bbR \left( \frac{m_{\delta_1}(t,x)^2}{\rho_{\delta_1}(t,x)} + \frac{\pi^2}{3} \rho_{\delta_1}(t,x)^3 \right) \, \rd x \, \rd t \le (1-\delta_1) G(\rho,v) + \frac{\pi^2}{6}\frac{\delta_1}{(2M)^2}. 
\end{equation*}
We give ourselves $\delta_2 \in ( 0, \frac 1 4]$ and a nonnegative even function $\varphi \in C^\infty_c(\bbR)$ supported in $[-1,1]$ with $\int \varphi = 1$, and set
\begin{equation*}  
\varphi_{\delta_2} (t,x) \coloneqq \delta_2^{-2} \varphi(t/\delta_2) \varphi(x/\delta_2), \qquad \rho_{\delta} \coloneqq \rho_{\delta_1} \star \varphi_{\delta_2}, \quad \text{ and } \quad m_{\delta} \coloneqq m_{\delta_1} \star \varphi_{\delta_2},
\end{equation*}
where $\star$ denotes the convolution in the time and space variables, and where we write $\delta = (\delta_1, \delta_2)$. By linearity, the pair $(\rho_{\delta}, m_{\delta})$ still satisfies the continuity equation (in the sense that $\partial_t\rho_{\delta}+ \partial_x m_{\delta} = 0$). Notice also that $\rho_\delta$ is constant on $(-\infty,-\delta_2]$, as well as on $[1+\delta_2, +\infty)$. For convenience, we introduce the shorthand 
\begin{equation*}  
{\mu}_{0,\delta} \coloneqq \rho_\delta(-\delta_2, x) \, \rd x, \quad {\mu}_{1,\delta} \coloneqq \rho_\delta(1+\delta_2, x) \, \rd x \quad \in \mcP_c(\bbR).
\end{equation*}
The function $\rho_{\delta}(-\delta_2,\cdot)$ is the convolution (in space only) of $\rho_{\delta_1}(0,\cdot)$ with $\delta_2^{-1} \varphi(\, \cdot /\delta_2)$; and similarly for $\rho_{\delta}(1+\delta_2,\cdot)$. 
In particular, we have that 
\begin{equation}
\label{e.W2.convergence}
\lim_{\delta \to 0} W_2(\mu_0, \mu_{0,\delta}) = \lim_{\delta \to 0} W_2(\mu_1, \mu_{1,\delta})  = 0.
\end{equation}
We give ourselves $\eta \in [\delta_2,\frac 1 4]$, and use convexity to write that
\begin{align}  
\label{e.m.delta.bound}
& \frac 1 2 \int_{-\eta}^{1+\eta} \int_\bbR \left( \frac{m_{\delta}(t,x)^2}{\rho_{\delta}(t,x)} + \frac{\pi^2}{3} \rho_{\delta}(t,x)^3 \right) \, \rd x \, \rd t
\\
\notag
& \qquad \le
\frac 1 2 \int_{-\eta}^{1+\eta} \int_\bbR \left( \frac{m_{\delta_1}^2}{\rho_{\delta_1}} + \frac{\pi^2}{3} \rho_{\delta_1}^3 \right) \star \varphi_{\delta_2}(t,x) \, \rd x \, \rd t
\\
\notag
& \qquad \le
\frac 1 2 \int_{-\eta-\delta_2}^{1+\eta + \delta_2} \int_\bbR \left( \frac{m_{\delta_1}(t,x)^2}{\rho_{\delta_1}(t,x)} + \frac{\pi^2}{3} \rho_{\delta_1}(t,x)^3 \right) \, \rd x \, \rd t
\\
\notag
& \qquad   \le (1-\delta_1) G(\rho,v) + \frac{\pi^2}{6}\frac{\delta_1}{(2M)^2} (1+2\eta+2\delta_2)+ (\eta+\delta_2) \frac{\pi^2}{6} \int_\bbR (\rho(0,x)^3 + \rho(1,x)^3) \, \rd x. 
\end{align}
Since we imposed $\eta \ge \delta_2$, we have that 
\begin{equation}
\label{e.}
\rho_\delta(-\eta,x) \, \rd x = \mu_{0,\delta} \quad \text{and} \quad \rho_\delta(1+\eta,x) \, \rd x = \mu_{1,\delta}.
\end{equation}
The point of the parameter $\delta_1$ is to ensure a uniform lower bound on the density, namely that
\begin{equation}  
\label{e.rho.lb}
\mbox{for every $x \in [-M,M]$}, \ \ \rho_{\delta}(t,x) \ge \frac {\delta_1}{4M}.
\end{equation}
We set 
\begin{equation*}  
v_{\delta}(t,x) \coloneqq 
\begin{cases}  
\frac{m_{\delta}(t,x)}{\rho_{\delta}(t,x)} & \text{ if } \rho_{\delta}(t,x) > 0, \\
0 & \text{ otherwise}. 
\end{cases}
\end{equation*}
By construction, the pair $(\rho_{\delta}, v_{\delta})$ satisfies the continuity equation \eqref{eq:continuity}, with $v_{\delta} \in L^2([0,1]\times \bbR; \rho_{\delta}(t,x)  \rd x \rd t)$. Thanks to the smoothing introduced with the parameter $\delta_2$, the function $v_{\delta}$ is infinitely differentiable on the set $\{(t,x) : \rho_{\delta}(t,x) > 0\}$, and in particular on $[0,1]\times [-M,M]$ by~\eqref{e.rho.lb}. 
Recalling  that $\rho(t,\cdot)$ and $m(t,\cdot)$ are supported in $[-M+1,M-1]$, we also see that for every $|x| \ge M-\frac 1 2$, the quantity $\rho_{\delta}(t,x)$ does not depend on $t$ and $v_{\delta}(t,x) = 0$. 
In particular, we have that $v_{\delta} \in \reg$. 
We are thus in position to appeal to Proposition~\ref{p.discretize} for the pair $(\rho_\delta,v_{\delta})$. 
One problem though is that if we apply it as stated, this proposition will not give us trajectories that belong to $\adm(\mu_0,\mu_1)$, since the initial and final measures will not be quite right. To remedy this, we take the $(\rho_\delta, v_\delta)$ trajectory on the time interval $[-\eta,1+\eta]$ and shrink it to $[\eta, 1-\eta]$, and we then leverage Lemma~\ref{l.bridging} to bridge the gap from time $0$ to time $\eta$ and from time $1-\eta$ to time $1$. That is, we set
\begin{equation*}  
    \begin{dcases}
    \rho_{\delta,\eta}(t,x) &\coloneqq \rho_{\delta} \left( \frac{1+2\eta}{1-2\eta}\left( t - \frac 1 2 \right)+ \frac 1 2, x  \right), \\
    v_{\delta,\eta}(t,x) &\coloneqq \frac{1+2\eta}{1-2\eta} \, v_{\delta} \left( \frac{1+2\eta}{1-2\eta}\left( t - \frac 1 2 \right)+ \frac 1 2, x  \right).
    \end{dcases}
\end{equation*}
Notice that the pair $(\rho_{\delta,\eta}, v_{\delta,\eta})$ still satisfies the continuity equation. 
We have
\begin{multline}  
\label{e.expand.time}
\int_{\eta}^{1-\eta}  \int_\bbR \left( v_{\delta,\eta}(t,x)^2 + \frac{\pi^2}{3} \rho_{\delta,\eta}(t,x)^2  \right) \rho_{\delta,\eta}(t,x) \, \rd x \, \rd t 
\\
\le \frac{1+2\eta}{1-2\eta} \int_{-\eta}^{1+\eta} \int_\bbR \left( v_{\delta}(t,x)^2 + \frac{\pi^2}{3} \rho_{\delta}(t,x)^2  \right) \rho_{\delta}(t,x) \, \rd x \, \rd t ,
\end{multline}
and the latter integral has been estimated in \eqref{e.m.delta.bound}. 
Slightly adapting Proposition~\ref{p.discretize} so that the time interval $[0,1]$ there is replaced by $[\eta,1-\eta]$, 
we obtain trajectories $\{(x_{\delta,\eta,i}(t),v_{\delta,\eta,i}(t))\}$ defined for $t \in [\eta,1-\eta]$, with $\partial_t x_{\delta,\eta,i} = v_{\delta,\eta,i}$, such that for all $i \in \{1, \cdots, N\}$:
\begin{equation*}  
    \begin{dcases}
        \int_{-\infty}^{x_{\delta,\eta,i}(\eta)} \rd \mu_{0,\delta}(x) \, \rd x &= \frac {i}{N+1}, \\
        \int_{-\infty}^{x_{\delta,\eta,i}(1-\eta)} \rd \mu_{1,\delta}(x) \, \rd x &= \frac {i}{N+1},
    \end{dcases}
\end{equation*}
and furthermore
\begin{multline*}  
\bigg| \frac 1 {2N} \sum_{i = 1}^N \int_{\eta}^{1-\eta}\left( v_{\delta,\eta,i}(t)^2 + \frac{\pi^2}{3(N+1)^2} \left( x_{\delta,\eta,i+1}(t) - x_{\delta,\eta,i}(t) + \frac \eps {N+1} \right) ^{-2}  \right) \rd t \\
- \frac 1 2 \int_{\eta}^{1-\eta} \int_\bbR \left(v_{\delta,\eta}(t,x)^2 + \frac {\pi^2}{3} \left(\frac{\rho_{\delta,\eta}(t,x)}{1+\eps \rho_{\delta,\eta}(t,x)}\right)^2\right) \rho_{\delta,\eta}(t,x) \, \rd x \, \rd t \bigg| \le \mathrm{err}(N,\delta,\eta),
\end{multline*}
where $\mathrm{err}(N,\delta,\eta)$ is an error term such that, for each fixed $\delta,\eta$, we have $\lim_{N \to +\infty} \mathrm{err}(N,\delta,\eta) = 0$.
We next extend the trajectories $\{(x_{\delta,\eta,i}(t),v_{\delta,\eta,i}(t))\}$ to the full interval $t \in [0,1]$ using Lemma~\ref{l.bridging}\footnote{
Technically, this yields trajectories for which the velocities have jumps, while we required the velocities to be continuous in the definition of $\adm_N(\mu_0,\mu_1)$ (see Section~\ref{subsubsec:def_discretization}), but it is easy to see that the set $\adm_N(\mu_0,\mu_1)$ can be extended to include such trajectories without causing a change in the value of $J_{N,\eps}(\mu_0,\mu_1)$.}.
Recalling eq.~\eqref{e.W2.convergence}, as well as eqs.~\eqref{e.m.delta.bound} and \eqref{e.expand.time}, we send $N$ to infinity, then $\delta$ to zero, and finally $\eta$ to zero, to obtain that
\begin{equation*}  
\limsup_{N \to \infty} J_{N,\eps}(\mu_0,\mu_1) \le J(\mu_0,\mu_1),
\end{equation*}
as desired. \hfill
$\qed$

\subsection{Remark on an alternative formulation}\label{subsec:final_remark_proof}

We remark that, by taking a formal $N \to \infty$ limit, Theorem~\ref{thm:main} suggests the following possible alternative formulation for $J(\mu_0,\mu_1)$. For each $\mu \in \mcP(\bbR)$ and $\omega \in [0,1]$, let
\begin{equation*}  
F_{\mu}^{-1}(\omega) \coloneqq \inf \left\{ s \ge 0  \ : \  \mu((-\infty,s]) \ge \omega\right\} 
\end{equation*}
denote the inverse cumulative distribution function of $\mu$. Let $\adm'(\mu_0,\mu_1)$ denote the space of all continuous functions $Z \in C([0,1]^2 ;\bbR)$ that satisfy the following properties.
\begin{enumerate}  
\item We have $Z(0,\cdot) = F_{\mu_0}^{-1}$ and $Z(1,\cdot) = F_{\mu_1}^{-1}$.
\item For every $t \in [0,1]$, the function $Z(t,\cdot)$ is non-decreasing.
\item The function $Z$ is differentiable on $(0,1)^2$.
\end{enumerate}
We speculate that for every $\mu_0,\mu_1 \in \mcP_c(\bbR)$ with $\Sigma(\mu_0)$ and $\Sigma(\mu_1)$ finite, the quantity $J(\mu_0,\mu_1)$ is equal to
\begin{equation}  
\label{e.new.var}
\inf_{Z \in \adm'(\mu_0,\mu_1)} \frac 1 2 \int_0^1 \int_0^1 \left((\partial_t Z(t,\omega))^2 + \frac{\pi^2}{3(\partial_\omega Z(t,\omega))^2}\right) \, \rd \omega \, \rd t,
\end{equation}
and also to
\begin{equation}  
\label{e.new.var.eps}
\lim_{\eps \to 0} \inf_{Z \in \adm'(\mu_0,\mu_1)} \frac 1 2 \int_0^1 \int_0^1 \left((\partial_t Z(t,\omega))^2 + \frac{\pi^2}{3(\eps + \partial_\omega Z(t,\omega))^2}\right) \, \rd \omega \, \rd t.
\end{equation}
Notice that, as in the discrete approximation, the set $\adm'(\mu_0,\mu_1)$ is convex and the functionals appearing inside the infima in eqs.~\eqref{e.new.var} and in \eqref{e.new.var.eps} are strictly convex in $Z$. 
Moreover, the quantity in eq.~\eqref{e.new.var.eps} is more directly related to the discretization we perform in Section~\ref{sec:numerics}.

\printbibliography

\newpage
\appendix 
\addtocontents{toc}{\protect\setcounter{tocdepth}{1}} %Only sections titles in the TOC for the appendix

\section{Additional properties and technicalities}\label{sec_app:properties}
\subsection{Scaling and invariance properties of the HCIZ integral}\label{subsec_app:rescaling_invariance}

We make here two simple remarks on properties of the HCIZ integral and the hydrodynamical system of Theorem~\ref{thm:limit_hciz}.

\subsubsection{Scaling parameter $\theta \geq 0$}

If we introduce the scaling parameter $\theta \geq 0$,
eq.~\eqref{eq:I_HCIZ_matytsin} becomes
\begin{align}\label{eq:I_HCIZ_wtheta}
    I_\HCIZ(\theta, \mu, \nu) &= - \frac{3}{4} - \frac{1}{2} \log \theta 
    + \frac{\theta}{2} \left[\int \mu(\rd x) x^2 + \int \nu(\rd x) x^2 \right]
    - \frac{1}{2} [\Sigma(\mu) + \Sigma(\nu)] 
    \\ \nonumber
    & - \frac{1}{2} \inf_{(\rho, v)} \int_0^1 \rd t \, \int \rd y \, \rho(t,y) \Big[v(t,y)^2 + \frac{\pi^2}{3} \rho(t,y)^2\Big],
\end{align}
where the conservation of mass equation of eq.~\eqref{eq:continuity} is unchanged, and the boundary conditions read:
\begin{equation}\label{eq:boundary_cond_wtheta}
    \lim_{t \to 0} \rho(t,y) \, \rd y = \sqrt{\theta} \# \mu \hspace{30pt} \textrm{and} \hspace{30pt}
    \lim_{t \to 1} \rho(t,y) \, \rd y = \sqrt{\theta} \# \nu.
\end{equation}
Notice that we can rescale variables as:
\begin{equation}
    \begin{dcases}
    \trho(u, x) \coloneqq \sqrt{\theta} \rho(\theta u, \sqrt{\theta} x) , \\
    \tv(u, x) \coloneqq \sqrt{\theta} v(\theta u, \sqrt{\theta} x),
    \end{dcases}
\end{equation}
which yields
\begin{align}\label{eq:rescaling_hciz}
    \int_0^1 \rd t \, \int \rd y \, \rho(t,y) \Big[v(t,y)^2 + \frac{\pi^2}{3} \rho(t,y)^2\Big] &= \int_0^{\theta^{-1}} \rd u \, \int \rd x \, \trho(u,x) \Big[\tv(u,x)^2 + \frac{\pi^2}{3} \trho(u,x)^2\Big],
\end{align}
The continuity equation~\eqref{eq:continuity} is still satisfied by $(\trho, \tv)$, and this rescaling changes the boundary conditions to:
\begin{equation}\label{eq:boundary_cond_rescaled}
    \lim_{u \to 0} \trho(u,x) \, \rd x = \mu \hspace{30pt} \textrm{and} \hspace{30pt}
    \lim_{u \to \theta^{-1}} \trho(u,x) \, \rd x = \nu.
\end{equation}
Informally, eq.~\eqref{eq:rescaling_hciz} allows accounting for $\theta$ by rescaling either the interpolation time or the space variable.

\subsubsection{Galilean transformation}

The equations above are also invariant under Galilean transformations. If $(\rho, v)$ satisfy eqs.~\eqref{eq:continuity} and~\eqref{eq:euler_2nd_equation}, then so do
\begin{align*}
    \begin{dcases}
        \trho(t, x) &\coloneqq \rho(t, x + \alpha t),  \\
        \tv(t, x) &\coloneqq v(t, x + \alpha t) - \alpha,
    \end{dcases}
\end{align*}
for any $\alpha \in \bbR$.
Plugging this relation into eq.~\eqref{eq:I_HCIZ_wtheta}, we recover that, if $\nu_\alpha(x) \coloneqq \nu(x - \alpha)$,
\begin{align*}
    I_\HCIZ(\theta, \mu,  \nu_\alpha) = I_\HCIZ(\theta, \mu,  \nu) + \alpha \, \theta \, \EE_{\mu}[X],
\end{align*}
which is immediate from the finite-dimensional expression of eq.~\eqref{eq:def_Id}.

\subsection{Remarks on time discretization}\label{subsec_app:time_discretization_scheme}

We briefly detail here why the discretization scheme in time that we use in eq.~\eqref{eq:def_GNDt_eps} creates a time discretization error that we expect to be of order $\mcO(\|\Dt\|_\infty^2) = \mcO(1/T^2)$, as validated by the numerical experiments in Appendix~\ref{sec_app:additional_numerics}.

\myskip
We consider a single trajectory $x \in \mcC^2([0,1], \bbR)$. Let $F$ be twice continuously
differentiable on a neighborhood of $x([0,1])$, then the trapezoidal rule yields
\begin{equation*}
    \left|\int_0^1 F(x(t))\,\rd t-\sum_{k=0}^T\omega_kF(x_k)\right|
    \leq C \|\Dt\|_\infty^2,
\end{equation*}
where the weights $\omega_k$ are defined in eq.~\eqref{eq:def_wk_Dt}, and $C > 0$ only depends on $\{\|x^{(p)}\|_\infty, \|F^{(p)}\|_\infty\}_{0 \leq p \leq 2}$.
The same convergence order holds for the discretized kinetic-energy term in eq.~\eqref{eq:def_GNeps}, as
shown by the following elementary statement.
\begin{lemma}\label{lemma:kinetic_discretization}
Let $x\in \mcC^2([0,1], \bbR)$ and let $0=t_0<t_1<\cdots<t_T=1$. For $k\in[T]$, set
$\Dt_k\coloneqq t_k-t_{k-1}$, $x_k\coloneqq x(t_k)$, and
\begin{equation*}
    v_k\coloneqq\frac{x_k-x_{k-1}}{\Dt_k}.
\end{equation*}
Then
\begin{equation*}
    0\leq \int_0^1|\partial_t x(t)|^2\,\rd t-\sum_{k=1}^T\Dt_kv_k^2
    \leq \frac{\|\Dt\|_\infty^2}{\pi^2}
    \int_0^1|\partial_t^2x(t)|^2\,\rd t.
\end{equation*}
In particular, as $\|\Dt\|_\infty \downarrow 0$, the discrete kinetic energy converges to its continuous counterpart
with error $\mcO(\|\Dt\|_\infty^2)$.
\end{lemma}

\begin{proof}[Proof of Lemma~\ref{lemma:kinetic_discretization}]
Denote $I_k\coloneqq(t_{k-1},t_k)$.
Expanding the square and
using that $\int_{I_k}\partial_t x\,\rd t=\Dt_k v_k$ yields
\begin{equation*}
    \int_{I_k}(\partial_t x(t)-v_k)^2\,\rd t
    =\int_{I_k}(\partial_t x(t))^2\,\rd t- \Dt_k v_k^2.
\end{equation*}
And summing over $k$:
\begin{equation}\label{eq:kinetic_variance_identity}
    \int_0^1|\partial_t x(t)|^2\,\rd t-\sum_{k=1}^T \Dt_kv_k^2
    =\sum_{k=1}^T\int_{I_k}(\partial_t x(t)-v_k)^2\,\rd t \geq 0.
\end{equation}
We next apply the one-dimensional Poincaré--Wirtinger inequality on each interval:
for any $f$ such that $f, \partial_t f \in L^2(I_k)$, with $\bar f \coloneqq \Dt_k^{-1} \int_{I_k} f(t) \rd t$, we have
\begin{equation}\label{eq:poincare_wirtinger_interval}
    \int_{I_k}(f-\bar f)^2\,\rd t
    \leq\frac{\Dt_k^2}{\pi^2}\int_{I_k}|\partial_t f|^2\,\rd t.
\end{equation}
Indeed, using the Fourier expansion in $L^2(I_k)$:
\begin{equation*}
    g(t) \coloneqq f(t) - \bar{f} = \sqrt{\frac{2}{\Dt_k}}\sum_{n\geq1}a_n\cos\left(\frac{n\pi (t-t_{k-1})}{\Dt_k}\right),
\end{equation*}
Parseval's identity then gives
\begin{equation*}
    \int_{I_k}g(t)^2\,\rd t=\sum_{n\geq 1}a_n^2,
    \qquad
    \int_{I_k}|\partial_t g(t)|^2\,\rd t
    =\sum_{n\geq 1}\left(\frac{n\pi}{\Delta t_k}\right)^2 a_n^2.
\end{equation*}
Since $n^2\geq 1$, eq.~\eqref{eq:poincare_wirtinger_interval} follows.
Applying it to eq.~\eqref{eq:kinetic_variance_identity} then yields immediately the result.
\end{proof}

\subsection{Derivation of eq.~\eqref{eq:G_NDt_3}}\label{subsec_app:eq_G_NDt_3}

We detail here the derivation of the identity
\begin{align}\label{eq:to_show_G_NDt_3}
\nonumber
   & \min_{(x_{1,k})_{1 \leq k \leq T-1}} 
    \left\{
    \frac {1}{N} \sum_{k=1}^T \frac{1}{\Delta t_k} \sum_{i = 1}^N
    \left[x_{1,k} - x_{1,k-1} + \frac{\Delta t_k}{N+1} \sum_{j=1}^{i-1} \rd \delta_{j,k}\right]^2\right\} \\ 
    &= [\ox(0) - \ox(1)]^2 + 
    \frac{1}{(N+1)^2} \sum_{k=1}^T \Delta t_k \, (\rd \delta_k^\T) L_{N-1}^{-1} (\rd \delta_k).
\end{align}
As mentioned in the main text, this is a simple convex quadratic minimization problem. 
Recall that $x_{i,k} \coloneqq x_{1, k} + (N+1)^{-1} \sum_{j=1}^{i-1} \delta_{j,k}$.
The stationarity condition over $(x_{1,k})$ gives for all $1 \leq k \leq T-1$,
\begin{equation*}
    \frac{1}{N} \sum_{i=1}^N \left[\frac{x_{i,k} - x_{i,k-1}}{\Dt_k} - \frac{x_{i,k+1} - x_{i,k}}{\Dt_{k+1}}\right] = 0.
\end{equation*}
Letting $\ox_k \coloneqq (1/N) \sum_{i=1}^N x_{i,k}$, this can be rewritten as
\begin{equation*}
    \frac{\ox_{k} - \ox_{k-1}}{\Dt_k} = c,
\end{equation*}
for some constant $c \in \bbR$ and all $k \in \{1, \cdots, T\}$.
Thus we have 
\begin{align*}
    \ox_k = \ox_0 + \sum_{l=1}^k \Dt_l \left(\frac{\ox_{l} - \ox_{l-1}}{\Dt_l}\right) = \ox_0 + c t_k,
\end{align*}
with $t_k \coloneqq \sum_{l=1}^k \Dt_l$; using the boundary conditions, we obtain
\begin{align}\label{eq:ox_k}
    \ox_k = (1-t_k) \ox_0 + t_k \ox_T. 
\end{align}
Notice that 
\begin{equation*}
    \ox_k = x_{1,k} + \frac{1}{N(N+1)} \sum_{i=1}^N \sum_{j=1}^{i-1} \delta_{j,k} = x_{1,k} + \frac{1}{N+1} \sum_{j=1}^{N-1} \left(1 - \frac{j}{N}\right)\delta_{j,k}.
\end{equation*}
Therefore, by eq.~\eqref{eq:ox_k}, the minimizer of eq.~\eqref{eq:to_show_G_NDt_3} satisfies 
\begin{align*}
    x_{1,k} = (1-t_k) \ox_0 + t_k \ox_T - \frac{1}{N+1} \sum_{j=1}^{N-1} \left(1 - \frac{j}{N}\right)\delta_{j,k}.
\end{align*}
Notice that this expression remains valid for the fixed boundary variables, i.e.\ for $k \in \{0, T\}$.
In particular, we have for any $1 \leq k \leq T$
\begin{align}
    \label{eq:dx_1_min}
    x_{1,k} - x_{1,k-1} = \Dt_k (\ox_T - \ox_0) - \frac{\Delta t_k}{N+1} \sum_{j=1}^{N-1} \left(1 - \frac{j}{N}\right)\rd \delta_{j,k}.
\end{align}
We now plug back eq.~\eqref{eq:dx_1_min} into eq.~\eqref{eq:to_show_G_NDt_3}:
\begin{align*}
     &\min_{(x_{1,k})_{1 \leq k \leq T-1}} 
    \left\{
    \frac {1}{N} \sum_{k=1}^T \frac{1}{\Delta t_k} \sum_{i = 1}^N
    \left[x_{1,k} - x_{1,k-1} + \frac{\Delta t_k}{N+1} \sum_{j=1}^{i-1}\rd \delta_{j,k}\right]^2\right\}
    \\ 
    &= 
    \frac {1}{N} \sum_{k=1}^T \Delta t_k \sum_{i = 1}^N
    \left[\ox_T - \ox_0 + \frac{1}{N+1} \sum_{j=1}^{N-1} \left(\indi\{j \leq i-1\} - 1 + \frac{j}{N} \right)\rd \delta_{j,k}\right]^2, \\ 
    &\aeq [\ox_T - \ox_0]^2 + \frac {1}{N} \sum_{k=1}^T \Delta t_k \sum_{i = 1}^N
    \left[\frac{1}{N+1} \sum_{j=1}^{N-1} \left(\indi\{j \leq i-1\} - 1 + \frac{j}{N} \right)\rd \delta_{j,k}\right]^2, \\
    &= [\ox_T - \ox_0]^2 + \frac {1}{(N+1)^2} \sum_{k=1}^T \Delta t_k \sum_{j,l=1}^{N-1} (\rd \delta_{j,k}) (\rd \delta_{l,k}) M_{jl},
\end{align*}
using in $(\rm a)$ that $\sum_{i=1}^N \left(\indi\{j \leq i-1\} - 1 + j/N \right) = 0$ for all $1 \leq j \leq N-1$, 
and with 
\begin{align*}
    M_{jl} &\coloneqq \left[\frac{1}{N} \sum_{i=1}^N \left(\indi\{j \leq i-1\} - 1 + \frac{j}{N} \right) \left(\indi\{l \leq i-1\} - 1 + \frac{l}{N} \right)\right], \\ 
    &= \frac{1}{N} \left[\min(j,l) - \frac{jl}{N}\right].
\end{align*}
This completes the derivation of eq.~\eqref{eq:to_show_G_NDt_3}, since one can easily check that $M = L_{N-1}^{-1}$, as defined in eq.~\eqref{eq:defs_G_NDt_3}.

\subsection{Minimization algorithm with a rescaling parameter $\theta$}\label{subsec_app:rescaling_algorithm}

Recall the discussion in Section~\ref{subsec:scaling_theta}.

\noindent
\textbf{Objective function --}
We let $t \coloneqq \theta u \in (0,1)$.
We still consider time steps $\Dt = (\Dt_k)_{k=1}^T$, with $\sum_{k=1}^T \Dt_k = 1$.
It is a straightforward computation to check that eq.~\eqref{eq:G_NDt_4} becomes, under rescaling:
\begin{align}\label{eq:G_NDt_theta}
    &G^\topth_{N, \Dt, \eps}(\delta) \\ 
    \nonumber
    &= \frac{\theta}{2} \left[\ox(0) - \ox(1) \right]^2 + \frac{\pi^2}{6\theta N} \sum_{k=1}^T \sum_{i=1}^{N-1} \omega_k \sigma_\eps(\delta_{i,k}) + \frac{\theta}{2 (N+1)^2} \sum_{k=1}^T \Delta t_k \, (\rd \delta_k^\T) L_{N-1}^{-1} (\rd \delta_k), 
\end{align}
The objective is still to minimize this functional over all $(\delta_{i,k})$ for $1 \leq i \leq N-1$ and $1 \leq k \leq T-1$, under the constraint $\delta_{i,k} \geq 0$ and the same boundary conditions as in eq.~\eqref{eq:boundary_delta_x1}.

\myskip 
\textbf{Time schedule --}
We use an argument similar to the one sketched in Section~\ref{subsec:time_schedule}. At time $u \in (0,\theta^{-1})$, one checks easily from eq.~\eqref{eq:dbb} and Appendix~\ref{subsec_app:rescaling_invariance} that $\rho(u,x) \, \rd x$ is the law of a free Brownian bridge
\begin{align}\label{eq:dbb_theta}
    X_u \coloneqq (1-t) X_0 + t X_1 + \sqrt{\frac{t (1-t)}{\theta}} S,
\end{align}
where $t \coloneqq \theta u \in (0,1)$, and recall that $X_0$ has distribution $\mu$, $X_1$ has distribution $\nu$, and $S$ is a semicircular variable, free with $(X_0, X_1)$.
We fix an even integer $T \geq 2$ denoting the number of discretization points, 
and we focus on the regime $t \in [0,1/2]$. We let $\tT \coloneqq T/2$.

\myskip
Eq.~\eqref{eq:dbb_theta} shows that the evolution of $X_u$ happens on two timescales: a timescale $t \sim \theta$, and a second timescale $t = \Theta(1)$. This is particularly important for $\theta \ll 1$, as we might miss relevant dynamics if we use a naive discretization with too few points close to $t = 0$.
Moreover, for $t \downarrow 0$, we have $X_u \simeq X_0 + \sqrt{t/\theta} S$.
This motivates a discretization of $t \in [0,1/2]$ as $t_k = g_\theta(k/\tT)$ for $k \in \{0, \cdots, \tT\}$,
with $g_\theta$ chosen such that:
\begin{enumerate}[nosep,label=$(\roman*)$,leftmargin=*]
    \item $g_\theta(0) = 0$, $g_\theta(1) = 1/2$, and $g_\theta$ is strictly increasing.
    \item $g_{\theta = 1}(v) = v^2/2$, to be consistent with the case $\theta = 1$ of Section~\ref{subsec:time_schedule}.
    \item $g_{\theta}(v) \sim \theta v^2/2$ for $v \downarrow 0$.
    \item As $\theta \downarrow 0$, $g_\theta(v) = \Theta(\theta)$ for each fixed $v < 1/2$ and $g_\theta(v) = \Theta(1)$ for each fixed $v > 1/2$.
    This ensures that, in the limit $\theta \downarrow 0$, half of the discretized points cover the timescale $t \sim \theta$.
\end{enumerate}
An example of a simple function satisfying all the properties above is:
\begin{equation}\label{eq:def_gtheta}
    g_\theta(v) \coloneqq \frac{-1 + 4 v^2 + \sqrt{(1-4v^2)^2 + 4 \theta v^2 (3+\theta)}}{4 (3+\theta)},
\end{equation}
and this is the choice we make in our numerical simulations. We show the behavior of this function in Fig.~\ref{fig:gtheta}.
All in all, the time steps we use for a generic $\theta \geq 0$ are:
\begin{equation}\label{eq:def_tk_sqrt_schedule_theta}
    t_k \coloneqq
    \begin{dcases}
        g_\theta\left(\frac{k}{T/2}\right) &, \hspace{1cm} (k \in \{0, \cdots, T/2\}),\\
        g_\theta\left(\frac{T-k}{T/2}\right) &, \hspace{1cm} (k \in \{T/2+1, \cdots, T\}).
    \end{dcases}
\end{equation}
and we let $\Delta t_k \coloneqq t_k - t_{k-1}$.
\begin{figure}[ht]
    \centering
    \includegraphics[width=0.7\textwidth]{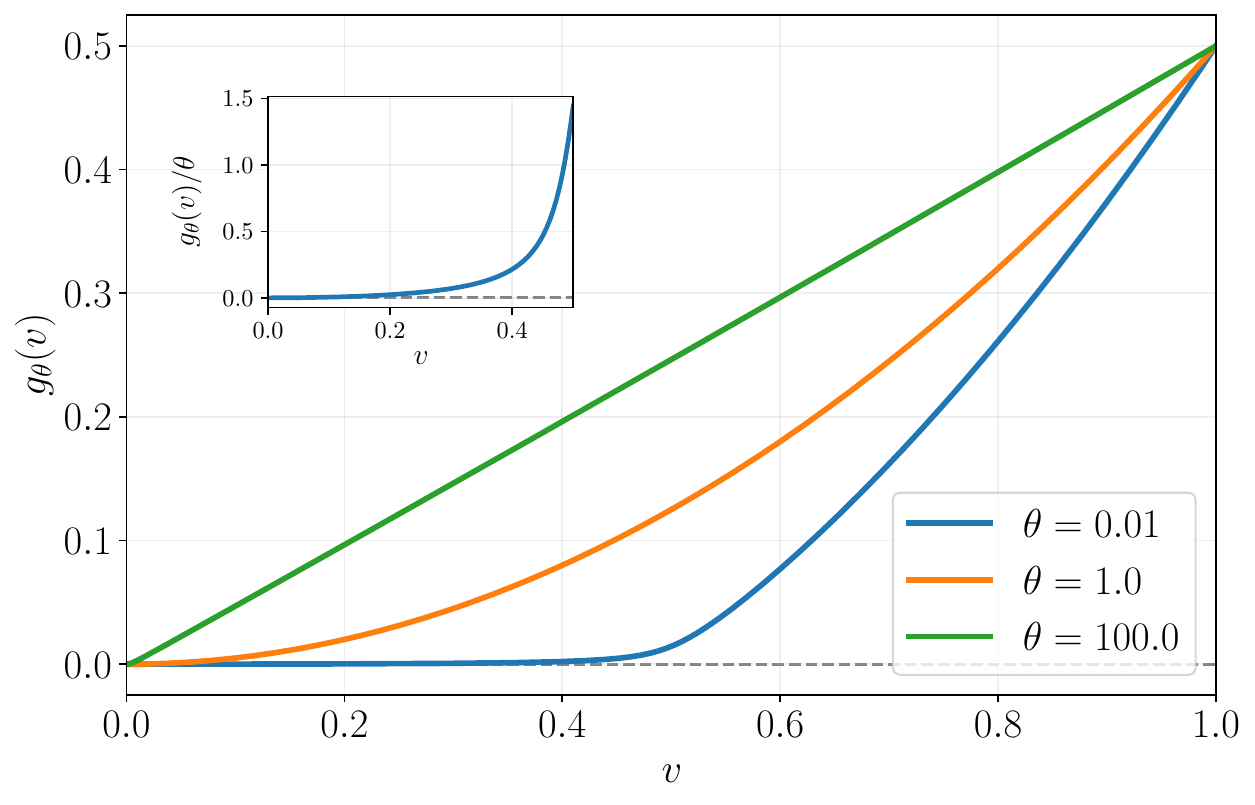}
    \caption{
        The function $g_\theta(v)$ defined in eq.~\eqref{eq:def_gtheta}, for different values of $\theta$.
    }
    \label{fig:gtheta}
\end{figure}

\myskip 
\textbf{Preconditioning --}
The preconditioning procedure of Section~\ref{subsec:preconditioning}, i.e.\ the approximation to the Hessian of $G^\topth_{N, \Dt, \eps}$, generalizes straightforwardly. For completeness, we state that in this case the Hessian reads: 
\begin{align}\label{eq:Hess_G_NDt_theta}
    \Hess \, G^\topth_{N,\Dt, \eps} &= \frac{\pi^2}{6 \theta N} \Diag\left[w_k \sigma_\eps''(\delta_{i,k})\right] + \frac{\theta}{(N+1)^2} L_{N-1}^{-1} \otimes L_{T-1}(\Dt),
\end{align}
while the approximation used in the preconditioner is: 
\begin{align}\label{eq:H_theta}
    &H^\topth(\delta) \coloneqq 
    \\ 
    \nonumber
    &\Id_{N-1} \otimes \Diag\left[\left(\frac{\pi^2 w_k}{6 \theta N(N-1)} \sum_{i=1}^{N-1}
    \sigma_\eps''(\delta_{i,k})\right)_{k=1}^{T-1}\right] \Id_{T-1} + \frac{\theta}{(N+1)^2} L_{N-1}^{-1} \otimes L_{T-1}(\Dt).
\end{align}

\subsection{Preconditioned conjugate gradient}\label{subsec_app:pcg}

We recall here the preconditioned conjugate gradient procedure (see e.g.~\cite{barrett1994templates}) in 
Algorithm~\ref{algo:pcg}.
\begin{algorithm}[t]
\SetAlgoLined
\KwResult{The unique solution $x = \PCG(H ; A, b)$ to the linear system $Ax = b$}
\textbf{Input: } A positive-definite matrix $A$, a vector $b$, a positive-definite matrix $H$ (the \emph{preconditioner}), an initial guess $x_0$, and a precision target $\eps > 0$\;
\emph{Initialize}  $r_0 = b - A x_0$, $z_0 = H^{-1} r_0$, $p_0 = z_0$\;
\While{$\|r_k\|_2 > \eps \|b\|_2$}{
\begin{equation*}
    \begin{dcases}
        \alpha_k & = \frac{r_k^\T z_k}{p_k^\T A p_k}, \\
        x_{k+1} &= x_k + \alpha_k p_k, \\
        r_{k+1} &= b - A x_{k+1} = r_k - \alpha_k A p_k, \\
        z_{k+1} &= H^{-1} r_{k+1}, \\
        p_{k+1} &= z_{k+1} + \frac{r_{k+1}^\T z_{k+1}}{r_k^\T z_k} p_k, \\
    \end{dcases}
\end{equation*}
$k = k + 1$\;
}
\caption{Preconditioned conjugate gradient (PCG).\label{algo:pcg}}
\end{algorithm}
The classical conjugate gradient method corresponds to $H = \Id$. 
However, this method can suffer from slow convergence if the condition number $\kappa(A) \coloneqq \lambda_{\max}(A) / \lambda_{\min}(A)$ is too large.
To remedy this issue,
if $H = EE^T$ is the Cholesky decomposition,
Algorithm~\ref{algo:pcg} effectively applies the classical conjugate gradient method to the linear system $E^{-1} A (E^{-1})^\T y = E^{-1} b$, where $y = E^\T x$. 
One thus aims to choose a preconditioner $H$ such that both the following hold.
\begin{enumerate}[label=$(\roman*)$,leftmargin=*]
    \item $\kappa(E^{-1} A (E^{-1})^\T)$ is much smaller than $\kappa(A)$, i.e.\ $H$ is a ``good approximation'' to $A$.
    \item Solving linear systems of the type $H x = z$ can be done at a low computational cost.
\end{enumerate}

\section{Runtime and convergence analysis experiments}\label{sec_app:additional_numerics}
We report in this section results of experiments on the total runtime of our algorithm (Appendix~\ref{subsec_app:runtime_analysis})
and its convergence properties (Appendix~\ref{subsec_app:cv_analysis}).

\subsection{Runtime analysis}\label{subsec_app:runtime_analysis}

We analyze here the total runtime of Algorithm~\ref{algo:newton_cg}. We take $\mu = \nu = \sigma_{\sci}$ as boundary densities, 
and vary the parameter $\theta \in \{0.01, 1.0\}$.
Runtime includes optimization time but does not include the time required to set up the problem 
or the time required to certify convergence of the last step.

\myskip 
\textbf{Scaling with the time discretization parameter $T$ --}
\begin{figure}[!htbp]
    \centering
    \includegraphics[width=1.0\textwidth]{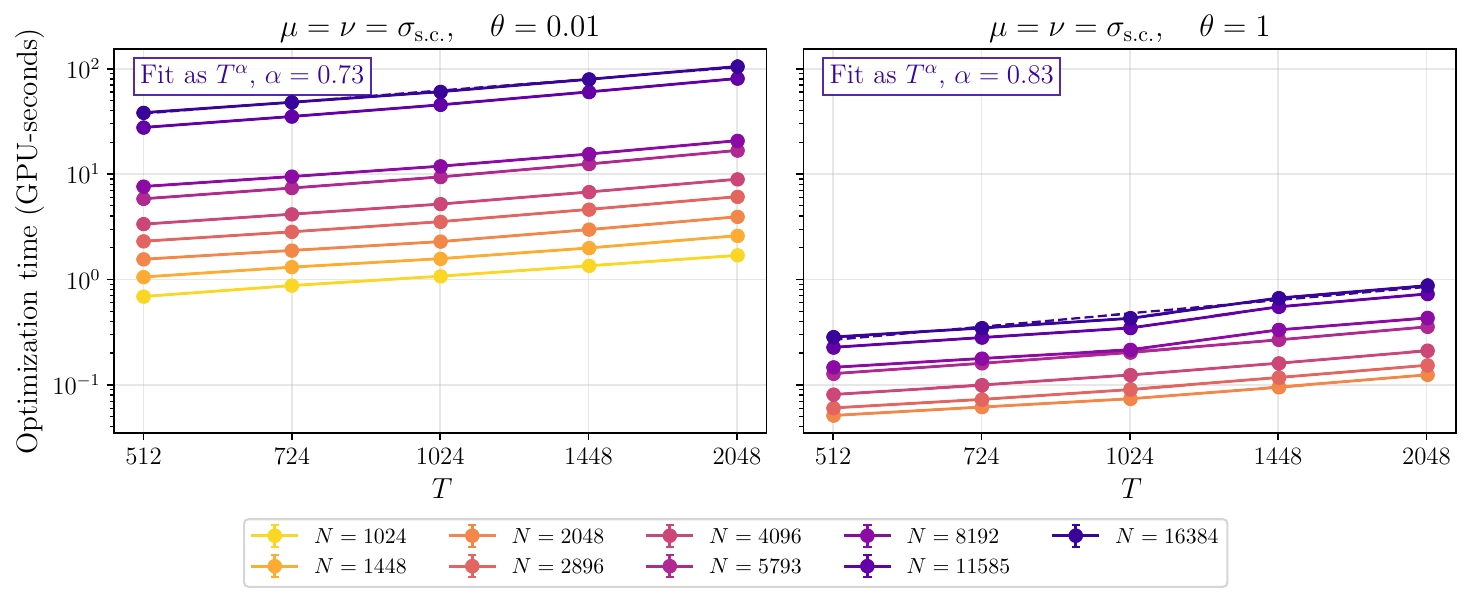}
    \caption{
        Runtime scaling analysis for increasing $T$, on a single RTX8000 GPU.
        Each point is averaged over 5 runs, error bars are very small (almost invisible). 
    }
    \label{fig:runtime_scaling_T}
\end{figure}
In Fig.~\ref{fig:runtime_scaling_T} we show how the total runtime scales for large $T$, at fixed number of particles $N$.
Notice that the number of operations per iteration is $\mcO(T)$ (for a fixed value of $N$ and increasing $T$). The observed scaling is even slower, 
exhibiting a very favorable runtime scaling as $T$ increases.

\myskip 
\textbf{Scaling with the number of particles $N$ --}
\begin{figure}[!htbp]
    \centering
    \includegraphics[width=1.0\textwidth]{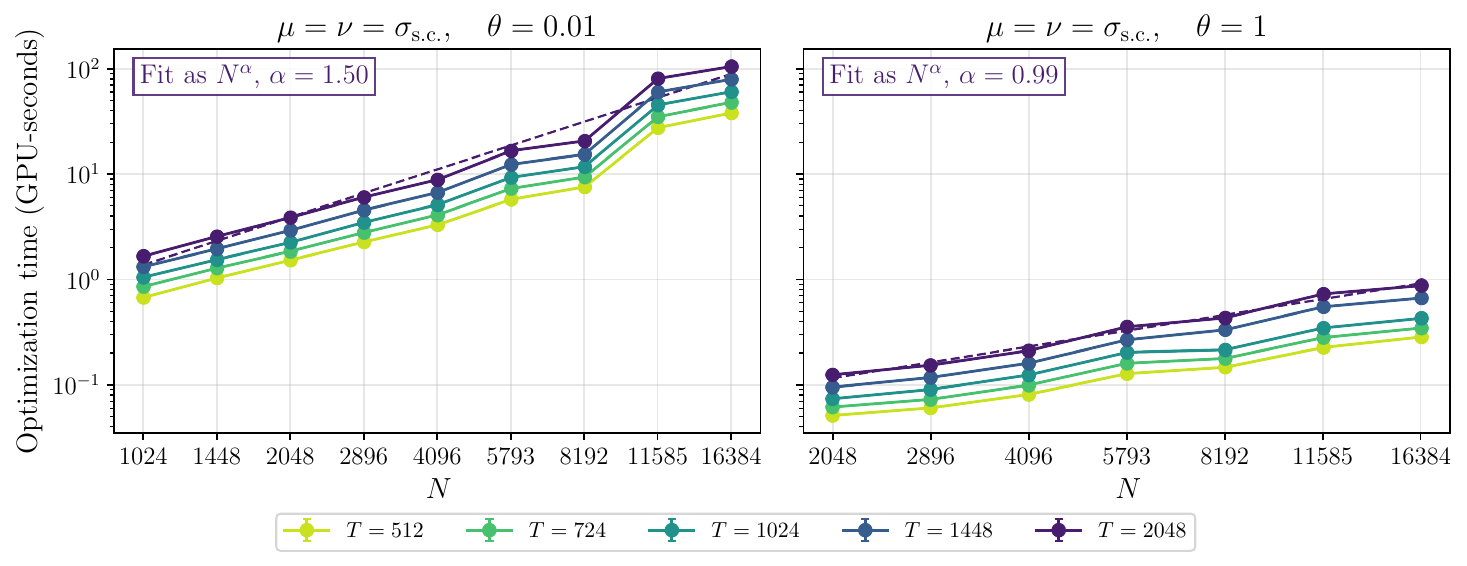}
    \caption{
        Runtime scaling analysis for increasing $N$, on a single RTX8000 GPU.
        Average over 5 runs each time, error bars are very small (almost invisible). 
    }
    \label{fig:runtime_scaling_N}
\end{figure}
In Fig.~\ref{fig:runtime_scaling_N} we show how the total runtime scales for large $N$, at fixed time grid size $T$.
We find two clear behaviors: for $\theta = 1$ the total runtime scales as $\mcO(N)$, while for $\theta = 0.01$ it appears to scale rather as $\mcO(N^{3/2})$.
Notice that the number of operations per PCG iteration is $\mcO(N \log N)$ (for a fixed value of $T$).
The observed scalings therefore suggest that the total number of PCG iterations should increase with $N$ for smaller $\theta$, signaling a more ill-conditioned problem. 
This is compatible with what we observe on the total number of PCG iterations below (see Fig.~\ref{fig:PCG_iterations_scaling}).

\myskip 
\textbf{Joint scaling in $N,T$ --}
\begin{figure}[!htbp]
    \centering
    \includegraphics[width=1.0\textwidth]{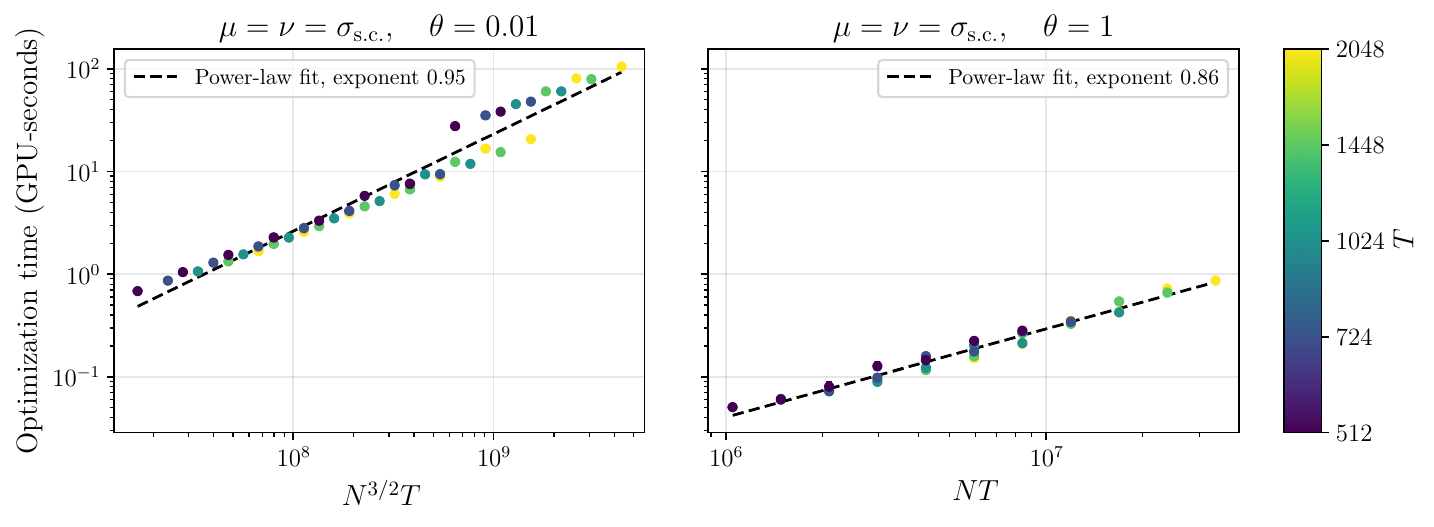}
    \caption{
        Runtime scaling analysis for increasing both $N$ and $T$, on a single RTX8000 GPU.
        Average over 5 runs each time, error bars are very small (almost invisible). 
    }
    \label{fig:runtime_scaling_NT}
\end{figure}
In Fig.~\ref{fig:runtime_scaling_NT} we show the runtime scaling as a function of a single parameter, either $N^{3/2} T$ (for $\theta = 0.01$) or $NT$ (for $\theta = 1.0$). 
In both cases we find that these single parameters represent relatively well the scaling of the total runtime, compatible with the fixed-$N$ and fixed-$T$ experiments shown in Figs~\ref{fig:runtime_scaling_T} and~\ref{fig:runtime_scaling_N}.

\myskip 
\textbf{Total number of PCG iterations --}
\begin{figure}[!htbp]
    \centering
    \includegraphics[width=1.0\textwidth]{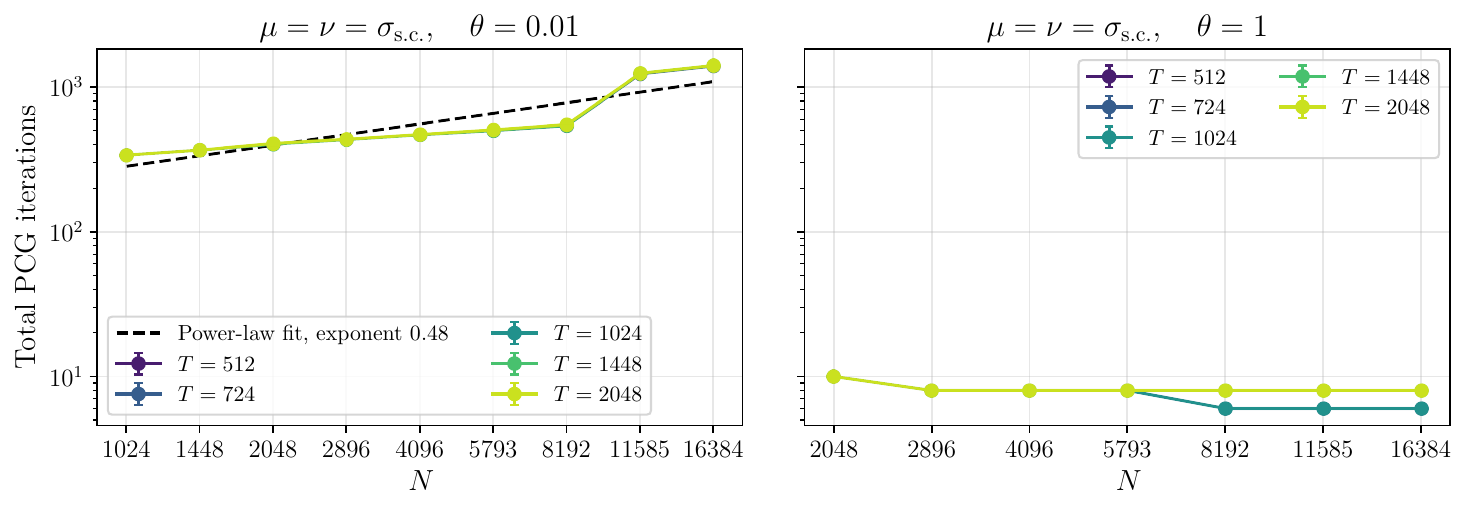}
    \caption{
        Total number of PCG iterations performed (across all Newton updates), as a function of $N$, for a given value of $T$.
    }
    \label{fig:PCG_iterations_scaling}
\end{figure}
Finally, in Fig.~\ref{fig:PCG_iterations_scaling} we show the total number of PCG iterations across all Newton updates, as a function of $N$ and for a range of values of $T$.
One can make two remarks:  first, the number appears independent of $T$ in both cases studied, suggesting that refining the time discretization does not worsen the conditioning of the problem.
Secondly, for $\theta = 1$ the problem is (after pre-conditioning) relatively well-conditioned even as we increase $N$, since the total number of PCG iterations is roughly constant.
However, for $\theta = 0.01$, the number of iterations appears to scale approximately as $N^{1/2}$, 
which is likely due to a Hessian near the global minimum which is more ill-conditioned despite the preconditioning.
Recall that since each PCG update requires $\mcO(N T \log N)$ iterations, multiplying the number of PCG iterations by $N T \log N$ gives a rough estimate of the computation time, compatible with the 
figures shown above.

\subsection{Convergence analysis}\label{subsec_app:cv_analysis}

In this section, we analyze the convergence of the output of Algorithm~\ref{algo:newton_cg} $J_{N, \Dt, \eps}$ (see eq.~\eqref{eq:def_JNDt}) to its asymptotic limit as $N, T \to \infty$, 
i.e.\ $J(\theta, \mu, \nu)$ of eq.~\eqref{eq:def_J}.
Throughout this section, we have $\eps = 10^{-7}$ fixed, we denote $J_{N, \Dt,\eps}(\theta, \mu, \nu)$ as $J_{N, T}$ to highlight its dependency on $N, T$, and denote simply $J(\theta, \mu, \nu)$ as $J$.
For concreteness, we perform the analysis in two cases for the boundary measures:
\begin{enumerate}[label=\textbf{($\bC$\arabic*)},ref=\textbf{($\bC$\arabic*)}]
    \item
    \label{cv_analysis:sc_to_sc}
    We set $\mu = \nu = \sigma_{\sci}$, with the scaling $\theta = 0.01$.
    In this case the value of $J$ is known analytically, so that both the successive
    differences and the signed total error $J_{N,T} - J$ are accessible.
    \item 
    \label{cv_analysis:smp_to_sc}
    We use the same setting as Fig.~\ref{fig:sym_mp_to_sc}:
    $\mu = \mu_{\sMP,\kappa}$ the symmetrized Marchenko-Pastur law
    (see eq.~\eqref{eq:sym_MP}) with $\kappa = 2$, $\nu = \sigma_{\sci, 0.5}$, and
    $\theta = 1$.
    Here $J$ is not known analytically, and we rely on successive differences only to estimate the convergence rates.
\end{enumerate}
Anticipating on the results of this section, the experiments below are all consistent with
the two leading discretization errors being of different orders \emph{and} essentially
decoupled, i.e.\ with an error model
\begin{equation}\label{eq:error_model}
    J_{N, T} = J + \frac{a}{T^2} - \frac{b}{N} + \smallO(T^{-2}) + \smallO(N^{-1}),
\end{equation}
where $a, b > 0$ depend on $(\theta, \mu, \nu)$ but not on $(N, T)$.
Note in particular the opposite signs: the time discretization overestimates $J$, while the
particle discretization underestimates it.

\subsubsection{Convergence as $T \to \infty$ and $\|\Dt\|_\infty \to 0$}

We first study the error introduced by the time discretization, for fixed values of the number of particles $N$.
We analyze~\ref{cv_analysis:sc_to_sc} in Fig.~\ref{fig:convergence_analysis_T_sc_sc}, 
\begin{figure}[!htbp]
    \centering
    \includegraphics[width=1.0\textwidth]{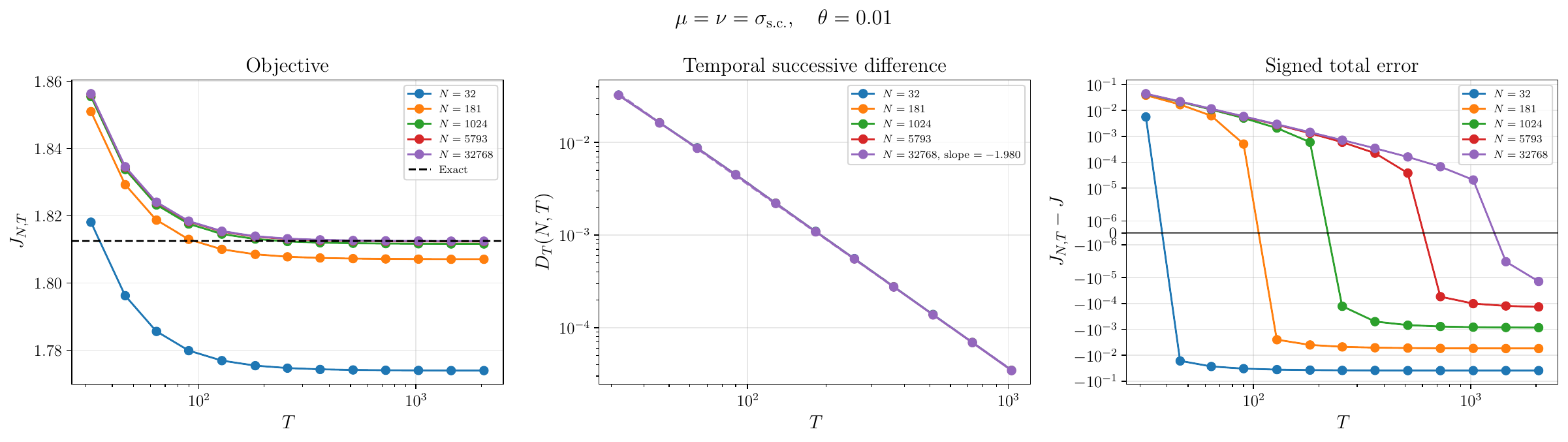}
    \caption{
        Convergence analysis as the number of time discretization steps $T$ grows,
        for~\ref{cv_analysis:sc_to_sc}.
        For different values of $N$ we show the value of the objective alongside its exact
        asymptotic value (left), the successive difference $D_T(N, T) \coloneqq |J_{N, 2T} - J_{N, T}|$ (middle), and the signed total error
        $J_{N,T} - J$ (right).
        The curves in the middle plot are superimposed: the time discretization error is
        essentially $N$-independent, and a least-squares fit is consistent with a $1/T^2$ decay.
        In the right plot, the error changes sign at an $N$-dependent value of $T$: it is
        positive as long as the time discretization dominates, and saturates at the
        (negative) finite-$N$ bias, in agreement with eq.~\eqref{eq:error_model}.
    \label{fig:convergence_analysis_T_sc_sc}
    }
\end{figure}
and~\ref{cv_analysis:smp_to_sc} in Fig.~\ref{fig:convergence_analysis_T_smp_sc}.
\begin{figure}[!htbp]
    \centering
    \includegraphics[width=1.0\textwidth]{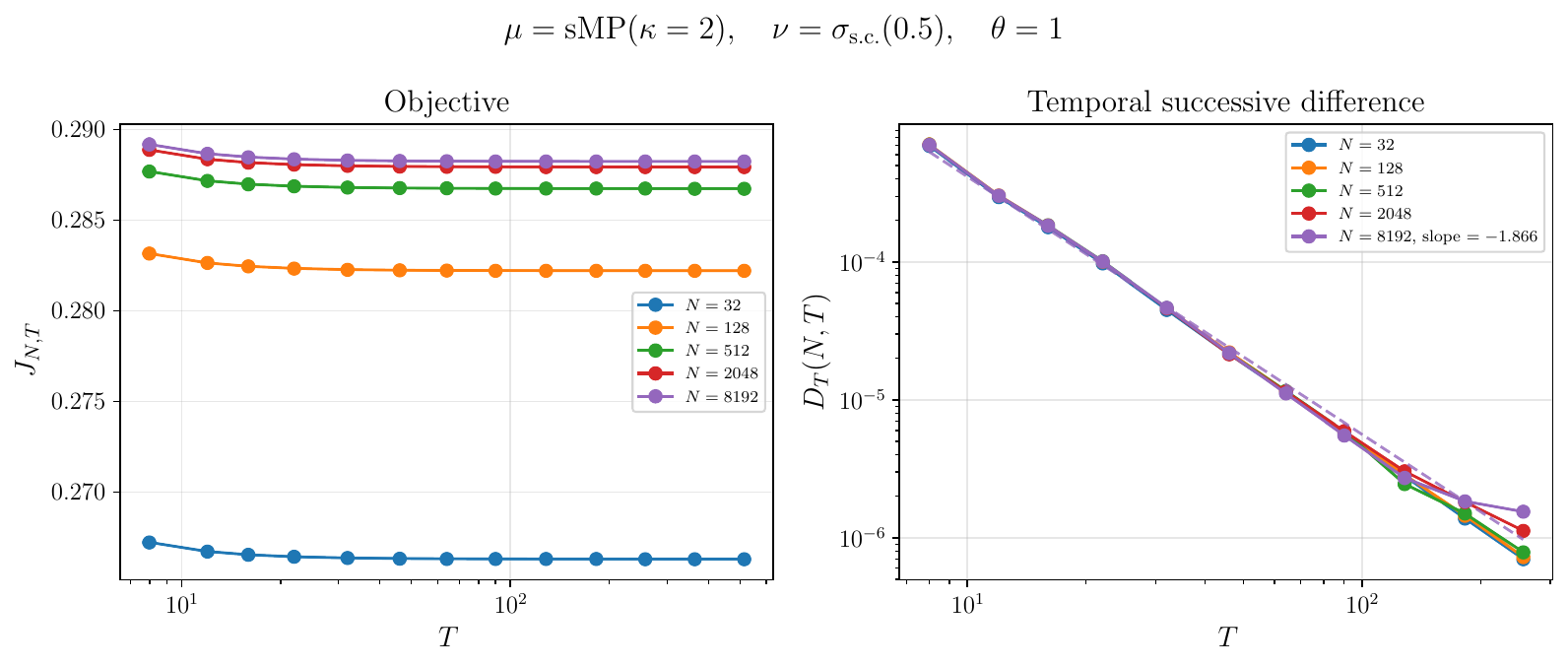}
    \caption{
        Convergence analysis as the number of time discretization steps $T$ grows,
        for~\ref{cv_analysis:smp_to_sc}.
        For different values of $N$ we show the value of the objective (left) and the
        successive difference $D_T(N, T) \coloneqq |J_{N, 2T} - J_{N, T}|$ (right).
        The objective is already within $\sim 10^{-4}$ of its large-$T$ limit for
        $T \simeq 30$, which is why we restrict to $T \leq 512$ here.
        The fitted slope is again compatible with a $1/T^2$ decay.
        }
    \label{fig:convergence_analysis_T_smp_sc}
\end{figure}
In both cases the time-discretization error is compatible with a $1/T^2$ decay, in accordance
with the order of our discretization scheme, and is only weakly dependent on $N$.
We also note that the objective is already very close to its large-$T$ limit for moderate
values of $T \sim 10^2$, so that in practice the accuracy of Algorithm~\ref{algo:newton_cg}
is often limited by $N$ rather than by $T$.

\subsubsection{Convergence as $N \to \infty$}

We then study the convergence rate of the objective as $N \to \infty$, for fixed values of the time discretization grid size $T$.
We analyze~\ref{cv_analysis:sc_to_sc} in Fig.~\ref{fig:convergence_analysis_N_sc_sc}, 
\begin{figure}[!htbp]
    \centering
    \includegraphics[width=1.0\textwidth]{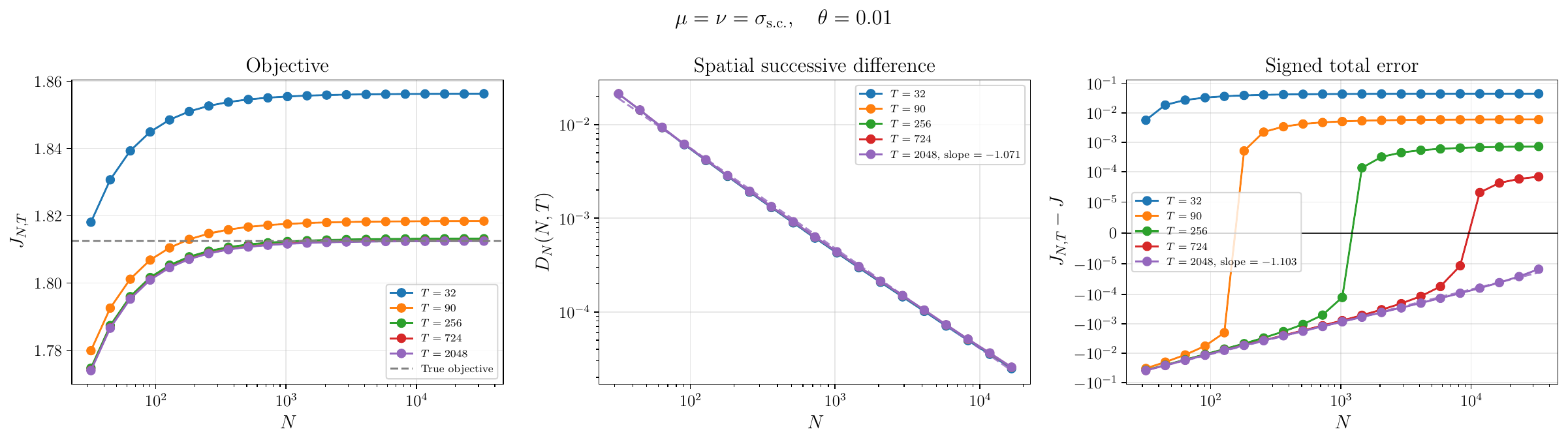}
    \caption{
        Convergence analysis as the number of particles $N$ grows,
        for~\ref{cv_analysis:sc_to_sc}.
        For different values of $T$ we show the value of the objective alongside its exact
        asymptotic value (left), the successive difference $D_N(N, T) \coloneqq |J_{2N, T} - J_{N, T}|$ (middle), and the
        signed total error $J_{N,T} - J$ (right).
        The curves in the middle plot are superimposed, so that the finite-$N$ error is
        essentially $T$-independent, and is compatible with eq.~\eqref{eq:error_model}. The fitted slope is compatible
        with a $1/N$ leading-order decay.
        }
    \label{fig:convergence_analysis_N_sc_sc}
\end{figure}
and~\ref{cv_analysis:smp_to_sc} in Fig.~\ref{fig:convergence_analysis_N_smp_sc}.
\begin{figure}[!htbp]
    \centering
    \includegraphics[width=1.0\textwidth]{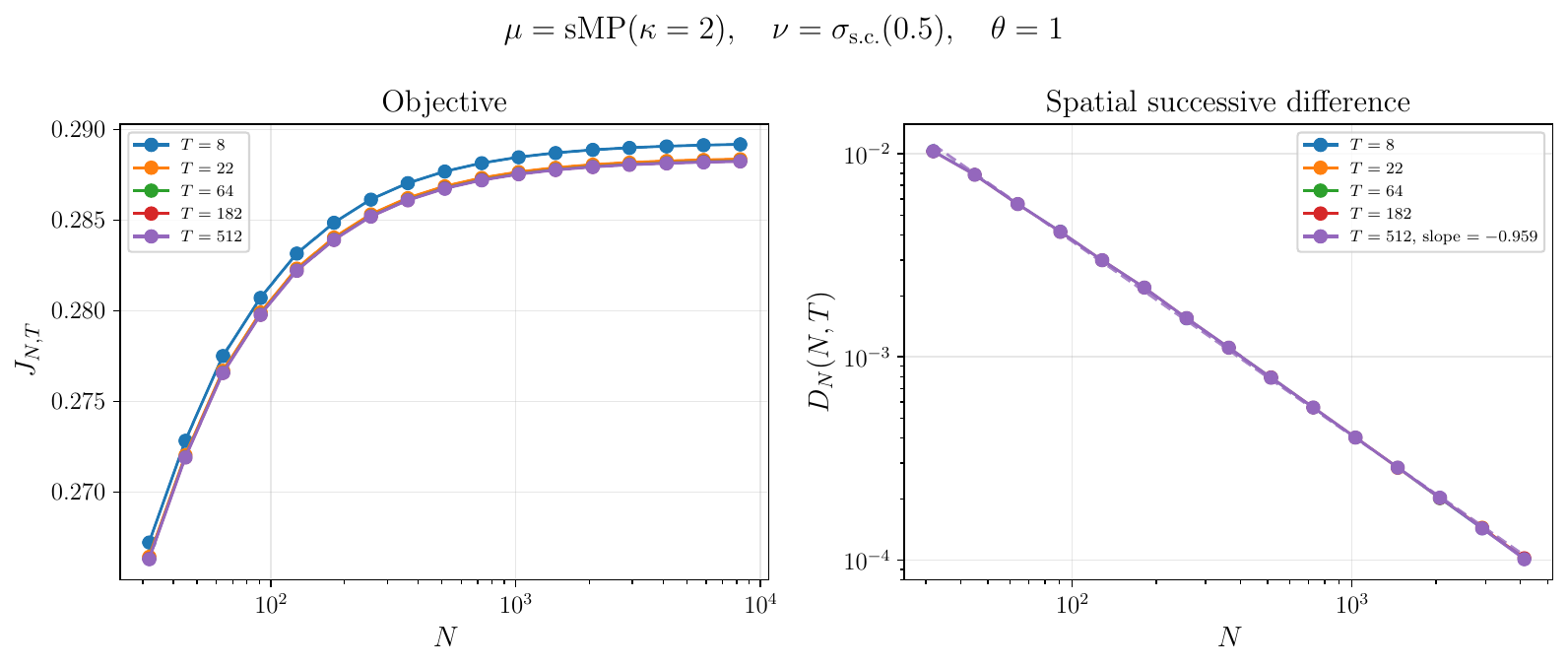}
    \caption{
        Convergence analysis as the number of particles $N$ grows,
        for~\ref{cv_analysis:smp_to_sc}.
        For different values of $T$ we show the value of the objective (left) and the
        successive difference $D_N(N, T) \coloneqq |J_{2N, T} - J_{N, T}|$ (right).
        As in Fig.~\ref{fig:convergence_analysis_N_sc_sc} the successive differences for
        different $T$ are superimposed, and the fitted slope is
        compatible with a $1/N$ decay.
        }
    \label{fig:convergence_analysis_N_smp_sc}
\end{figure}
The finite-$N$ error is thus compatible with a $1/N$ decay in both cases, and nearly
independent of $T$, consistently with the prediction of eq.~\eqref{eq:error_model}.

\subsubsection{Richardson extrapolation in \texorpdfstring{$N$}{N}}
\label{subsubsec_app:richardson}

The $1/N$ decay observed above motivates the use of Richardson
extrapolation~\cite{richardson1911approximate}. 
The idea can be summarized as follows. If, at fixed $T$:
\begin{equation*}
    J_{N, T} = J_{N = \infty, T} + \frac{a_1}{N} + \frac{a_2}{N^\alpha} + \smallO(N^{-\alpha}),
\end{equation*}
for some $\alpha > 1$ and some coefficients $a_1, a_2 \in \bbR$, then letting
\begin{equation*}
    J_{N, T}^{(R)} \coloneqq 2 J_{N, T} - J_{N/2, T},
\end{equation*}
one has
\begin{equation}\label{eq:def_richardson}
    J_{N, T}^{(R)} = J_{N = \infty, T} - \frac{2 (2^{\alpha - 1} - 1) a_2}{N^\alpha}
    + \smallO(N^{-\alpha}),
\end{equation}
i.e.\ the leading $1/N$ term is eliminated.
We apply this procedure for~\ref{cv_analysis:sc_to_sc} in Fig.~\ref{fig:richardson_extrapolation_sc_sc},
\begin{figure}[!htbp]
    \centering
    \includegraphics[width=1.0\textwidth]{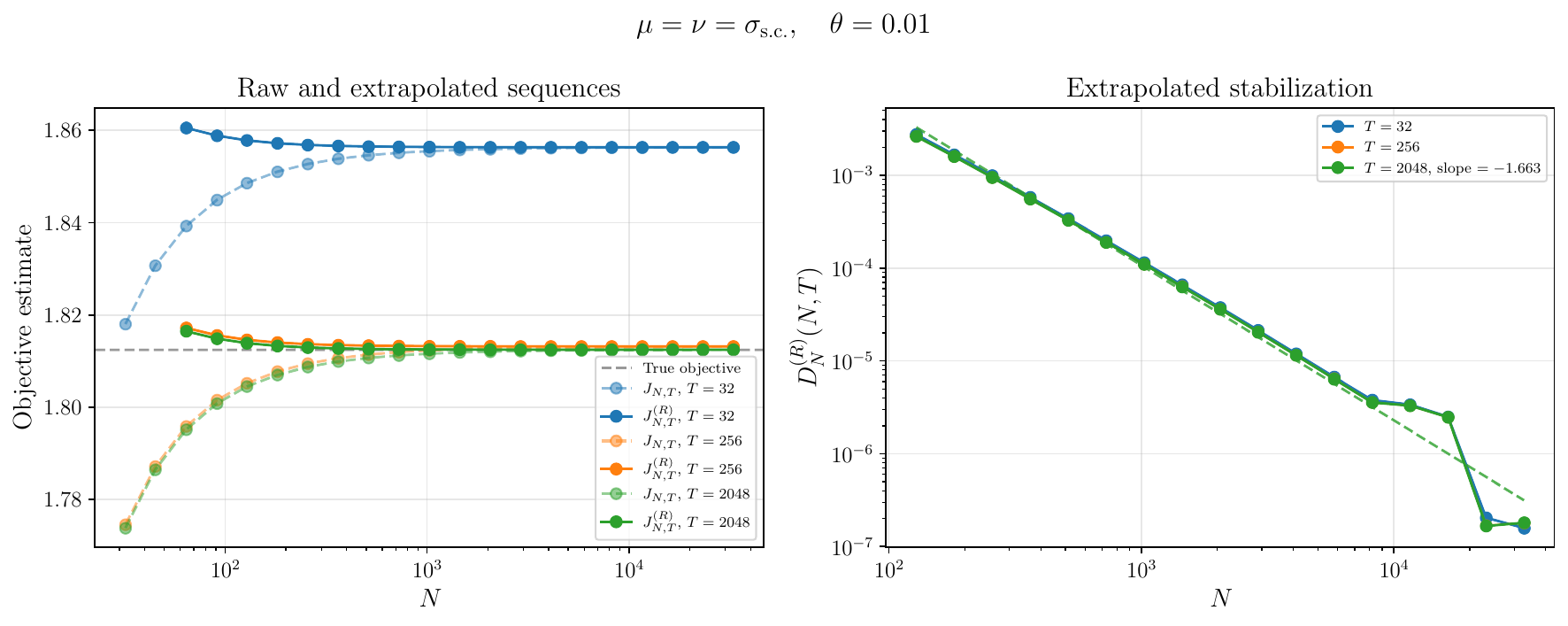}
    \caption{
        Richardson extrapolation (see eq.~\eqref{eq:def_richardson})
        for~\ref{cv_analysis:sc_to_sc}.
        We plot the raw and extrapolated objective estimates (left), as well as the
        successive difference
        $D_N^{(R)}(N, T) \coloneqq |J_{2N, T}^{(R)} - J_{N, T}^{(R)}|$ (right).
        }
    \label{fig:richardson_extrapolation_sc_sc}
\end{figure}
and to~\ref{cv_analysis:smp_to_sc} in Fig.~\ref{fig:richardson_extrapolation_smp_sc}.
\begin{figure}[!htbp]
    \centering
    \includegraphics[width=1.0\textwidth]{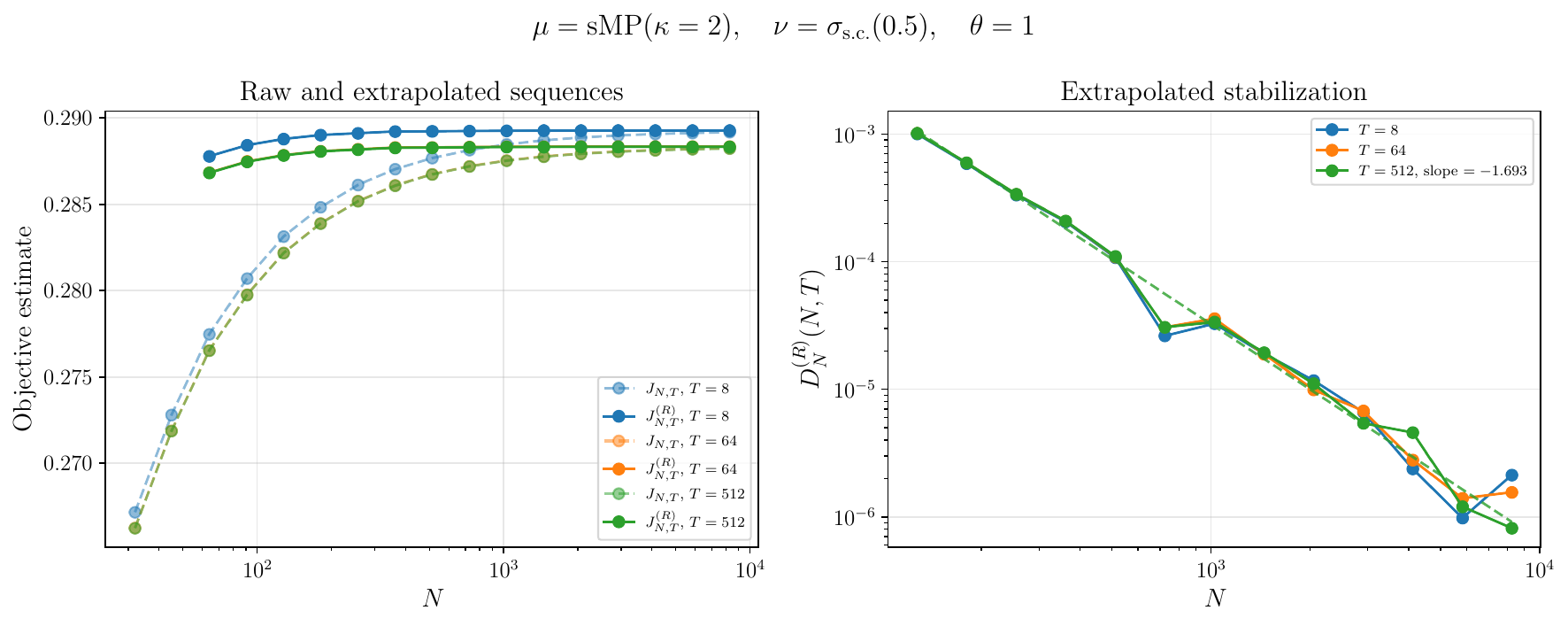}
    \caption{
        Richardson extrapolation (see eq.~\eqref{eq:def_richardson})
        for~\ref{cv_analysis:smp_to_sc}.
        We plot the raw and extrapolated objective estimates (left), as well as the
        successive difference $D_N^{(R)}(N, T) \coloneqq |J_{2N, T}^{(R)} - J_{N, T}^{(R)}|$ (right).
        }
    \label{fig:richardson_extrapolation_smp_sc}
\end{figure}
We observe a clear acceleration of the convergence, compatible with eq.~\eqref{eq:def_richardson} 
with an exponent $\alpha \in [1.6, 1.7]$.
In practice the gain is substantial: reaching an error of $10^{-4}$
in~\ref{cv_analysis:sc_to_sc} requires $N \simeq 10^4$ for the raw sequence against
$N \simeq 10^3$ for the extrapolated one.

\end{document}